\documentclass[aap,preprint]{imsart}

\RequirePackage{amsthm,amsmath,amsfonts,amssymb}
\RequirePackage[numbers]{natbib}
\usepackage{calc, latexsym,url}
\usepackage{xcolor}
\RequirePackage[colorlinks,citecolor=blue,urlcolor=blue]{hyperref} 

\startlocaldefs
\theoremstyle{plain}

\newtheorem{thm}{Theorem}[section]
\newtheorem{ass}{Assumption}
\newtheorem{lemma}[thm]{Lemma}
\newtheorem{corollary}[thm]{Corollary}
\newtheorem{prop}[thm]{Proposition}

\theoremstyle{remark}

\newtheorem{defn}[thm]{Definition}

\newtheorem{rem}[thm]{Remark}

\newtheorem{Condition}{Condition}

\newcommand{\R}{\mathbb{R}}

\newcommand{\N}{\mathbb{N}}

\newcommand{\Q}{\mathbb{Q}}
\newcommand{\Z}{\mathbb{Z}}
\newcommand{\E}{\mathbb{E}}

\renewcommand{\P}{\mathbb{P}}

\newcommand{\calY}{\mathcal{Y}}

\newcommand{\mbb}[1]{\mathbb{#1}}

\newcommand{\1}{1\hspace{-0.098cm}\mathrm{l}}

\newcommand{\dd}{{\mathrm{d}}}

\newcommand{\Ex}[1]{\ensuremath{\mathbb{E} \left[#1 \right]}}
\newcommand{\Prob}[1]{\ensuremath{\mathbb{P} \left(#1 \right)}}

\endlocaldefs

\begin{document}

\begin{frontmatter}
\title{The effective strength of selection in random environment}
\runtitle{Selection in random environment}

\begin{aug}
\author[A]{\fnms{Adri\'an} \snm{Gonz\'alez Casanova}\ead[label=e1]{agonz591@asu.edu}\orcid{0000-0002-1744-6506}}, 
\author[B]{ \fnms{Dario} \snm{Span\`o}\ead[label=e3]{d.spano@warwick.ac.uk}\orcid{0000-0001-5585-1943}}
\and
\author[C]{\fnms{Maite} \snm{Wilke-Berenguer}\ead[label=e4]{maite.wilkeberenguer@hu-berlin.de}\orcid{0000-0002-6057-495X}} 
\address[A]{School of Mathematical and Statistical Sciences, Arizona State University \printead{e1}
}

\address[B]{Department of Statistics, University of Warwick, \printead{e3}}

\address[C]{Institut f\"ur Mathematik, Humboldt-Universit\"at zu Berlin, \printead{e4}}
\end{aug}

\begin{abstract}
We analyse a family of two-types Wright-Fisher models with selection in a random environment and skewed offspring distribution. We provide a calculable criterion to quantify the impact of different shapes of selection on the fate of the weakest allele, and thus compare them. The main mathematical tool is duality, which we prove to hold, also in presence of random environment (quenched and in some cases annealed), between the population's allele frequencies and genealogy, both in the case of finite population size and in the scaling limit for large size. Duality also yields new insight on properties of branching-coalescing processes in random environment, such as their long term behaviour.
\end{abstract}



\end{frontmatter}

\allowdisplaybreaks 


\section{Introduction and overview of content} \label{sec:branching}

In population genetics the selective fitness of a species, or an allele, is widely thought to be permeable to the influence of environmental factors which may vary randomly in time: in certain generations, the population may be subject to particularly stressful external conditions (extreme temperatures, cataclysms, abrupt invasions of pathogens, hurricanes etc) making the selective advantage of some allelic types unusually more pronounced than in other generations. This could be due to a better ability to secure resources, or to a lower sensitivity to stress,  or to various other reasons \citep{ huracain, G73, KL73, KL74}.
\\
Thus if we consider a population comprising only two allelic types $0$ and $1$, say, where type $1$ is assumed to be always fitter than type $0$, it is intuitively clear that the possible occurrence of such \lq\lq cataclysmic\rq\rq generations would enhance the probability of extinction of type $0$.
It is, however, less obvious if rare but strong selective events put type $0$ more at risk of extinction than a small but constant-in-time selective pressure. The problem is reminiscent of similar questions arising in experimental biology, where, for example, some detrimental substance (antibiotic) is inoculated in a population of bacteria, and there is an interest in determining whether a constant administration of the substance in low concentration dosage is more effective at wiping out the population, than a more occasional inoculation with higher dosages of varying concentration \cite{pena}.
This paper aims to study the long term effects of rare selective events and quantify how big and frequent cataclysms owed to random environment must be, in order to wipe out a family as effectively as constant weak selection in a steady environment. 
\\ 

To this purpose, we will construct a suitable family of jump-diffusion allele-frequency models, generalising the well-known family of one-dimensional Lambda-Fleming-Viot models to encompass frequency-dependent selection in a time-varying random environment. It will turn out that, even in the simple case of independent and identically distributed ({\em i.i.d.}) environment, it is not sufficient to compare the total rate $\alpha_{\mathfrak s},$ say, of rare selective events against the rate $w$ of classical constant weak selection: it is necessary to take into account also information about the actual \lq\lq shape\rq\rq of rare selection, reflecting the specific action of the random environment. This will lead us to a rigorous derivation of a shape parameter $\alpha^*\geq 1$, obtained in Section \ref{sec:longterm} (formula \eqref{eq:alpha*}), as a non-trivial explicit function of both the random environment and of the offspring distribution influenced by it.
\\
The central result in this regard is Theorem \ref{thm:XlongtermGriffiths}. We will demonstrate that the correct way to measure the \emph{effective strength of selection}, combining the effect of both rare and weak selection, is via the quantity
\begin{equation}\alpha_{\rm Eff}:=\alpha_{\mathfrak s}\alpha^*+w.
\label{eq:alphatot}
\end{equation}
Indeed, we will find (Theorem \ref{thm:XlongtermGriffiths}) a critical value $c\beta^*$ (see formula \eqref{eq:beta*})  for the total selection $\alpha_{\rm Eff}$, separating regimes leading to almost sure ultimate extinction ($\alpha_{\rm Eff}> c\beta^*$), from regimes where the weaker type 0 has a chance to survive and fixate ($\alpha_{\rm Eff}<c\beta^*$). This extends results by Foucart \cite{Foucart13} and Griffiths \cite{GriffithsLambda} who independently discovered this threshold for the case of classical weak selection.  While Foucart works with the branching-coalescing process of the population's genealogy, Griffiths introduces a helpful representation of the generator of the forward-in-time diffusion process of the allele frequencies and combines this with a Lyapunov-function approach. We propose (Section \ref{sec:ltr}) an extension of Griffiths' approach to find the critical value for Lambda Fleming-Viot models with frequency-dependent selection and random environment. A non-trivial step in this endeavour requires closing a gap in \cite{GriffithsLambda} (caused by an exchange of limit and integration not  justified {\em a priori}), which we have achieved with our Lemma \ref{lem:amazinggrace}. This Lemma, which also covers the case analysed by Griffiths, can be extended to include the models in \cite{ADSelection} and it may be of independent interest for uses even beyond population genetics.
\\

It will turn out that the critical value $c\beta^*$  for  $\alpha_{\rm Eff}$ does not depend on the selection parameters, nor on the environment. This has interesting and, to some extent, counter-intuitive implications: In a model with purely rare selective events $(w=0)$, the effective critical threshold for $\alpha_{\mathfrak s}$ to guarantee extinction is given by the ratio $c\beta^*/\alpha^*,$ which is never smaller than $c\beta^*.$ This implies that, given the same genetic drift, a higher minimum level of rare selective intensity $\alpha_{\mathfrak s}$ is required to guarantee extinction, compared to a model with purely constant weak selection ($\alpha_{\mathfrak s}=0$) for which the same $c\beta^*$ is still the actual threshold for $w$. Due to the form of the shape parameter $\alpha^*,$ it will be apparent that this is the case even for choices of environment
typically favouring large \lq\lq cataclysms\rq\rq, from which we would expect a faster pressure towards extinction. For a comparison among specific examples of rare selection mechanisms see Remark \ref{rem:geom_v_Bin}.\\

The analysis of the longterm effects is embedded in an extended analysis of aspects of the modelling of such cataclysims with further results that may be of independent interest and which we briefly describe here.

\subsection{Finite population model} To derive our generalised Lambda Fleming-Viot process, we will propose, in Section \ref{sec:selection_RE}, an individual-based construction, based on a \lq\lq random graph\rq\rq\ approach, first proposed in \cite{ADSelection}, namely: a graph with deterministic vertices, representing individuals at each generation, and random edges being thrown between any two vertices for which a potential parental relation exists (cf. Definition \ref{def:DASG} in Section \ref{subsec:DASG}). On this graph, a suitable notion of selection is defined, associated to the idea that, at each generation, every individual chooses a random number of potential parents from the previous generation and inherits the fittest type among those found in their sampled pool. Thus the larger the typical pool's size, the less likely it is to inherit the least fit type. In a population of size $N$, at each generation $g\in\Z$, selection is thus parametrised by the probability distribution $Q$, say, of the vector of pool sizes $K_g=\{K_{(g,i)}:i=1,\ldots,N\}$. The distribution of $Q$ will depend randomly upon time, through a sequence of $[0,1]$-valued random variables, representing the time-varying random environment. \\
The selection mechanism is thus fully specified in terms of the time-varying distribution of pool sizes of potential parents, but the actual impact of selection depends also on how each individual chooses each such pool, and this choice is in turn affected by whether or not, at any given generation, the population undergoes an extreme reproductive event (or Lambda-event). In generations without extreme reproductive events, all individuals will sample their own pool of potential parents independently at random; in those with extreme reproduction, choices made by distinct individuals may be correlated. The occurrence and size of extreme reproductive events at each generation will be modelled by an additional, independent sequence of $[0,1]$-valued random variables. \\

\subsection{Infinite population model}  We will then focus on the dynamics under an {\em i.i.d.} environment and, taking the population size $N$ to infinity and rescaling time suitably, we will prove weak convergence of both allele frequencies (Theorems \ref{lem:converges} and \ref{thm:convergence_rare_and_weak}) and ancestral processes (Theorem \ref{thm: ancestral converg}).   A brief description is the following.
The limit ancestral process $Z,$ counting the number of lineages ancestral, at any time, to a (sample from a) given generation, will evolve as a continuous-time Markov chain with state space $\N\cup\{\infty\}$, with positive jumps from $n$ to $n+k$ occurring at rate
\begin{equation}
\int_{[0,1]} \Prob{\sum_{j=1}^nK_{y,j}=n+k} \frac{\Lambda_{\mathfrak s} (dy)}{\Ex{K_{y,1}-1}}+ wn\delta_{1,k}
\label{eq:br}
\end{equation}
for $k \in \N$ and with negative jumps from $n$ to $n-k$ occurring at rate
\begin{equation}
c\int_{[0,1]}\binom{n}{k+1}  y^{k+1}(1-y)^{n-k-1} \frac{\Lambda_{\mathfrak c}(dy)}{y^2} + \sigma\binom{n}{2}\delta_{1,k}
\label{eq:co}
\end{equation}
for all $k\in\{1,...,n-1\},$ where $w\geq 0$ and $c\geq 0$ are constants, $\Lambda_{\mathfrak s}$ and $\Lambda_{\mathfrak c}$ are two finite measures with no atoms in $0$ 
and, for each $y\in[0,1],$
 $\{K_{y,j}\}_{j\in\N}$ is an \emph{iid} family of $\N\cup\{\infty\}$-valued random variables with common distribution $Q(y)$ parametrised by $y\in [0,1]$, where $Q(y)=\delta_1$ if and only if $y=0$. We will refer to  $Z=(Z(t):t\geq 0)$ as the \emph{branching-coalescing process in random environment (BCRE)} (cf. Definition \ref{def:BCPRE}). Positive jumps are driven by $\Lambda_{\mathfrak s}$ and correspond to a lineage branching off into one \emph{actual} and one or more \emph{virtual} parental lineages, as an effect of selective pressure, whereas negative jumps, driven by $\Lambda_{\mathfrak c}$, correspond to two or more lineages coalescing into a common ancestor, as an effect of pure genetic drift.
 The measures $\Lambda_{\mathfrak s}$ and $\Lambda_{\mathfrak c}$ are the intensities of two Poisson processes governing the occurrence and size of rare selective events (cataclisms) and of the extreme reproduction events, respectively.  Loosely speaking, for every $y\in[0,1]$, $Q(y)$ represents the distribution of the size of the pool of potential parents which a typical individual in the limit population will choose when the environment has value $y$. The choice of $Q$ is arbitrary as long as it satisfies a minimal key condition of integrability, as explained later on in Section \ref{sec:key cond}. The constant $w$ is responsible for the classical form of \emph{weak} genic selection, acting constantly and deterministically in time. Correspondingly, $\sigma$ is a weight associated to the classical Kingman dynamics (binary mergers, i.e. non-extreme reproduction).\\
Known standard genealogical processes are thus subsumed as particular cases of our model: in particular when $\Lambda_{\mathfrak s}$ and $\Lambda_{\mathfrak c}$ are null on $(0,1]$, the rates \eqref{eq:br} and \eqref{eq:co} describe the dynamics of the so-called Ancestral Selection Graph of Krone and Neuhauser \cite{KN}. \\
As for the forward-in-time allele frequency process, we will prove that a scaling limit approximation of the 0-allele frequency evolution is described by a jump-diffusion process $(X(t):t\geq 0)$ arising as the unique strong solution to the SDE
\begin{align}\label{eq:SDEX_intro}
\dd X(t) 	& = \int (\E[X(t^-)^{K_y}\,\mid\, X(t^-)]-X(t^-) )\ \mathcal N_{\mathfrak s}(\dd t,\dd y) - wX(t)(1-X(t))\dd t \notag\\
		& \qquad \qquad + c\int\int z(\1_{\{u \leq X(t-)\}} - X(t-))\ \mathcal N_{\mathfrak c}(\dd t, \dd z, \dd u) \notag \\
		& \qquad \qquad \qquad \qquad +\sqrt{\sigma  X(t)(1-X(t))}\dd B(t)
\end{align}
where:  $(B(t) : t\geq 0 )$ is a standard Brownian motion; $\mathcal N_{\mathfrak s}$ is a Poisson point measure on $[0,\infty[\times [0,1]$ with intensity $\dd s\otimes (\E[K_y-1])^{-1}\1_{]0,1]}\Lambda_{\mathfrak s}(dy)$, with $K_y\equiv K_{y,1}$, $\Lambda_{\mathfrak s}(dy)$ and $w$ exactly as in \eqref{eq:br};
$\mathcal N_{\mathfrak c}$ is a Poisson point measure on $[0,\infty[\times]0,1]\times]0,1[$ with intensity $\dd s\otimes z^{-2}\Lambda_{\mathfrak c}(\dd z)\otimes \dd u$, with $\Lambda_{\mathfrak c}(\dd z) $ exactly as in \eqref{eq:co}.  We will refer to the process $X$ as the \emph{two-type Fleming-Viot process with weak and rare selection (FVWRS)} (cf.\ Definition \ref{def:X_weak_and_rare}). For $c=0$ this process reduces to a model introduced by Bansaye, Caballero and M\'el\'eard in \cite{caballero}. The classical Wright-Fisher diffusion process with genic selection is obtained when $c=0$ and $\mathcal N_{\mathfrak s}([0,\infty[\times [0,1])=0$.\\
\subsection{Duality} The construction summarised above (and described in full details in Section \ref{subsec:DASG}) shows clearly how the dynamics (and fate) of the weak allele frequency is intrinsically linked to the dynamics of the ancestry of any sample from the population, which can be read off directly from the population's graph, as the connected component subtending the sample.\\
In models with selection in deterministic environment, a well-established method to study such a link is based on the notion of \emph{stochastic duality}, which has attracted the interest of an increasing number of researchers in a variety of areas of probability \cite{GKR, JK14,POS}, including population genetics \cite{seed,EG, ADSelection, HA07, M99}: If $X=(X(t):t\geq0)$ and $Z=(Z(t):t\geq 0)$ are two stochastic processes, respectively with state space $E$ and $F$,  then $X$ and $Z$ are said to be \emph{dual} to each other \emph{with respect to the duality function} $h:E\times F\to \R$ if, for every $(x,y)\in E\times F$,
\begin{equation}
\forall \, t\geq0:\qquad \Ex{h(X(t),y)\mid X(0)=x}=\Ex{h(x,Z(t))\mid Z(0)=y}.
\label{eq:dual}
\end{equation}
\\
We will extend the use of duality to analyse population models with time-varying random environment. We will first show in Lemma \ref{lem:consampleduality} that, for finite $N$, a \emph{quenched } version of the duality relation \eqref{eq:dual} holds conditionally on the environment's sample path, when $X$ describes the forward-in-time process of $0$-allele frequencies and $Z$ the corresponding backward-in-time ancestral process, tracking the number of surviving potential ancestral lineages. In this case, the duality function $h$ is an appropriate generalisation of the so-called \emph{sampling duality} function originally proposed in \cite{M99} and later adapted by \cite{ADSelection} to models with deterministic selection. 
For the case of {\em i.i.d.} environment, it turns out that samping duality also holds in an \emph{annealed} form (i.e. unconditionally on the environment), as will be shown in Proposition \ref{prop:iidduality}. In addition, as observed in \cite{M99} and given in Proposition \ref{prop:hypergeometricsampling}, more than one duality relation holds between $X$ and $Z$.\\

We will then prove in Lemma \ref{lem:duality} that sampling duality will converge to an \emph{annealed moment duality} relation between the two limit processes, i.e. one of the form \eqref{eq:dual} with the choice $h(x,n)=x^n$. \\ 

Although random environment was introduced quite a few decades back both in the literature of branching processes (e.g.\ \cite{AK71,SW}) and in population genetics models (\cite{KL73,KL74}), 
it is currently attracting a renewed interest in both communities (see  \cite{AGKV,BPCH,caballero,BH,GKV,GKV1,HLX} for branching processes and \cite{caballero,BEK, BH, HPP15} in the context of population genetics). 
 While the  present paper draws its main motivation from open problems about selection in population genetics, it is our opinion that the above-mentioned duality property is also interesting from the point of view of the theory of branching (and coalescing) processes in a random environment.  In the present work, Lemma \ref{lem:fellerZ} shows that not only does the key 
Condition \ref{MasterCondition} ensure strong existence of the solution of Equation \eqref{eq:SDEX_intro}, but it also implies conservativeness of the BCRE defined by \eqref{eq:br}-\eqref{eq:co}, and this time in presence of both coalescence and random environment. As a beautiful consequence of moment duality the solution of Equation \eqref{eq:SDEX_intro} is Feller, cf.\ Lemma \ref{lem:fellerX}.
 We believe that 
the results and methods in this paper could help to shed further light on the connections between these two families and it is plausible that several of the results presented here can be extended to other types of non-\emph{iid} environmental processes.\\

 \subsection{A key condition on the parameters of rare selection}\label{sec:key cond}
The action of {rare selection in random environment} is parametrised by two objects: the {rare selection mechanism} is given by a kernel $\Q=\{Q(y): y \in [0,1]\}$ where, for every $y$, $Q(y)$ is a distribution on $\N\cup\{\infty\}$; the {random environment} itself is determined by a measure $\mu$ on $[0,1]$, the space of parameters of $\Q$, whose one-to-one relation with $\Lambda_{\mathfrak s}$ we now explain. 
\\
As a central assumption for all our results, we will require 
that the kernel $\Q=\{Q(y): {y\in [0,1]}\}$ and the measure $\mu$ satisfy 
the following condition:
\begin{Condition}\label{MasterCondition}
The kernel $\Q$ and the measure $\mu$ satisfy
\begin{align}\label{eq:MasterCondition}
\int_{[0,1]}\;\E[K_y-1] \;\mu(\dd y) < \infty.
\end{align}
where $K_y$ is a ($\N\cup\{\infty\}$-valued) random variable with distribution $Q(y)$. \\
In addition, assume that $Q(y) = \delta_1$ if and only if $y=0$
\end{Condition}
\noindent Note that the  main part of Condition \ref{MasterCondition} is \eqref{eq:MasterCondition}. This is essential for most results, while the second part is mostly for simplicity. The results would still have analogous formulations where one would have to take into account that, e.g. singularities might appear for other values of $y$ that $y=0$. In addition, for the sake of simplicity we will also always assume that $Q(y) = \delta_1$ if and only if $y=0$. Note that Condition \ref{MasterCondition} implies the representation $\mu = \mu(\{0\})\delta_0 + \mu^+$ where $\mu^+$ has density $y \mapsto (\E[K_y-1])^{-1}\1_{]0,1]}(y)$ with respect to a finite measure which we denote by $\Lambda_{\mathfrak s}$. It will become apparent, however, that the atom in zero of $\mu$ has no effect on the selection mechanism, hence we will assume, without loss of generality, $\mu(\{0\}) = 0$. Since $\mu$ and $\Lambda_{\mathfrak s}$ are suitably equivalent, we will make use of both representations. 
\\
\ \\
The choice of $[0,1]$ as parameter space for $Q(y)$ is entirely arbitrary and one can replace it with any general parameter space ${\cal Y}$. The use of $[0,1]$ has the advantage of being rich enough and interpretable, encompassing most known models of non-balancing selection. For example, by setting $Q(y)$ to be the Geometric distribution on $\N$ with parameter $1-y$ for $y\in [0,1[$ and $Q(1)\equiv \delta_{\{\infty\}}$, and choosing $\Lambda_{\mathfrak s}= a\delta_{y^*},$ ($y^*\in]0,1]$) one recovers known population models with selection in constant environment: 
For $a=1$ and $y^*=0,$ the model  is neutral (no selection), in which case the ancestry is described by a coalescent with multiple collisions \cite{MS,pitman} with merger sizes governed by the parameter measure $\Lambda_{\mathfrak c}$ and the frequency process is the two-type Fleming-Viot process \cite{BLG03}. For $a=y^*\rightarrow 0$ the model specialises to the haploid weak selection model, whose genealogy is given by the Ancestral Selection Graph of \cite{KN} and its dual is the Wright-Fisher diffusion with weak selection \cite{Kimura}. If, in addition, $c=0$ 
the model reduces to Kingman's coalescent process \cite{Kingman} with the Wright-Fisher diffusion as 
 its dual \cite{Kimura}. With random environment (non-degenerate $\Lambda_{\mathfrak s}$), the duality property \eqref{eq:dual} has not been established before.
\\
The duality implies that the allele 0 will become extinct with probability one if and only if its ancestral process does not admit a stationary distribution, which happens when branching events occur at a sufficiently high rate to ultimately outperform the coalescence events. 
\\


\subsection{Outline of the paper}
The structure of this paper follows the approach of first summarising all important results while postponing the proofs to the last section. The results are presented in the order logical from a modeling point of view. The reader mainly interested in the result on the longterm behaviour (Theorem \ref{thm:XlongtermGriffiths}), may go directly to Section \ref{sec:longterm} (briefly visiting Section \ref{subsec:diffusionconvergence} for the existence of the object under study).
The paper is structured as follows. We begin, in Section \ref{subsec:DASG}, with the construction of the \emph{Wright-Fisher graph with selection in random environment} incorporating the possible occurrence of highly skewed offspring distributions (Lambda reproduction mechanism); in Sections \ref{sec:forward frequencies} and \ref{sec:discrete_genealogy} we will define the {frequency process of the $0$-allele} and the so-called {block-counting process of the ancestry of a sample}. Subsequently, in Section \ref{sec:discreteduality}, we prove the quenched and annealed sampling duality results. Section \ref{sec:scaling limit} is concerned with the scaling limits of the processes introduced before. \ref{subsec:diffusionconvergence} establishes the two-type Fleming-Viot process with selection $X$ as the scaling limit of the frequency process of the 0-allele, if the environment is chosen to be \emph{iid}. Section \ref{subsec:duality_convergence} begins with the duality of $X$ and the branching coalescing process in random environment $Z$. We prove the Feller property and conservativeness of $Z$ and deduce the Feller property of $X$ from this. Duality is then also used to prove that $Z$ arises as 
 the scaling limit of the block-counting process of the ancestry of the finite Wright-Fisher graph.
The long-term behaviour of the scaling limits and their dependence on the strength of selection is then analysed in Section \ref{sec:longterm} subdivided in the result for $X$ in Section \ref{subsec:longterm_X} and its translation through duality for $Z$ in Section \ref{subsec:ergodicity_Z}. All proofs are postponed to Section \ref{sec:proofs}, where they appear in the same order as previously the results.


\section{Modeling Selection in random environment}\label{sec:selection_RE} 

This section is devoted to the construction and analysis of an individual-based, finite-population model. It is based on a random graph describing the ancestral relations between the generations in the populations, extending \cite{ADSelection} to encompass randomly varying selection. \\
We first describe how selection, the random environment governing it and extreme reproductive events are parametrised in the graph. The population's forward-in-time allele frequency process (Section \ref{sec:forward frequencies})  and its backward-in-time genealogical process (Section  \ref{sec:discrete_genealogy}) then both arise as functions of the same graph.  This construction in particular then yields the (sampling) duality between the two processes (Section \ref{sec:discreteduality}).


\subsection{The discrete ancestral selection graph with random environment and skewed offspring distribution} \label{subsec:DASG}

Consider a population of fixed, finite size $N \in \N$, with discrete, non-overlapping generations indexed by $g \in \Z$.   
Denote: $[N]:=\{1,\ldots,N\}$ and $V_N:=\Z\times[N]$. Each individual in the history of the population is identified by a point $v=(g,i)\in V_N$, and we will write $g(v)=g$ and $i(v)=i$ to indicate, respectively, the \emph{generation} of $v$ and its \emph{label} in $[N]$. Negative and positive values of $g$ will then index past and future generations, respectively, with respect to an arbitrarily chosen \lq\lq present generation\rq\rq $g=0$. 
\\
\noindent  $V_N$ is the (deterministic) set of vertices of what is to become our random graph. The randomness, of course, lies in the ancestral relations: an edge will be drawn between any two vertices where one vertex is a potential parent of the other. \noindent 
The dynamics of the random graph will have to incorporate two features: selection in a random environment, and possibly skewed offspring distribution when extreme reproductive events occur. The former is determined by \emph{how many} potential parents an individual chooses, the latter by how the individuals choose the {\em labels} of their parents. Let us now describe the two mechanisms in more detail.\\

\textbf{Selection in random environment} 
At each generation $g$, we assign to every individual $v=(g,i)\in V$ a random number $K_{v}\in\N\cup\{\infty\}$, representing the number of potential parents it will choose - with replacement -  from among the $N$ individuals in generation $g-1$. The distribution of the collection $\{K_v:v\in V_N\}$ will be fully specified by
\begin{itemize}
 \item[(i)] the law of an environmental process, modelled as a sequence of random variables $\bar Y=(Y_g:g\in\Z)$ taking values in a measurable state space ${\cal Y}$; 
 \item[(ii)] a given probability kernel $\{Q(y,\cdot):y\in{\cal Y}\}$ from ${\cal Y}$ to $\N\cup\{\infty\}$; 
\end{itemize}
 and the following\vspace{-8pt}
\begin{ass}\label{ass:sel}
For every $g$, given $Y_g=y \in \calY,$ the random variables $K_{(g,i)}, i\in[N]$, are conditionally \emph{iid} with common distribution $Q(y,\cdot)$.\\
In addition, conditioned on a realisation of $\bar Y$, the random vectors $(K_{(g,i)}: i\in[N]), g \in \Z$, are \emph{independent.}
 \end{ass}
 \vspace{-10pt}
\noindent As mentioned in the introduction, throughout this paper we will mostly take ${\cal Y}=[0,1]$ for convenience. Note that we do not specify the law of the environmental process in this section. The scaling limit results in Section \ref{sec:scaling limit} will focus mainly on the assumption of the environment $\bar Y$ being an {\em iid} sequence.\\

\textbf{Highly-skewed offspring distribution}
\noindent The highly-skewed offspring distribution is modelled through correlation in how the individuals choose the labels of their potential parents from the previous generation. 
The strength of the correlation will be determined by $\Lambda$, a distribution on $[0,1]$ with $\Lambda(\{0\})=0$, and $c \in [0,1]$.   
Denote with $\mathcal U_{[N]}$ the discrete uniform distribution on $[N]$. \\
The choice of parental labels is then modelled as follows.
\vspace{-8pt}
\begin{ass}\label{ass:skewed}
Let $(C_g:g\in\Z)$ be \emph{iid} $[0,1]$-valued random variables with distribution 
\begin{equation*}
 (1-c)\delta_0 + c\Lambda 
\end{equation*}
 and $(I_g:g\in\Z)$ {\em iid} random variables with distribution $\mathcal U_{[N]}$, such that $(C_g:g\in\Z)$, $(I_g:g\in\Z)$ and $((K_{(g,i)}: i\in[N]): g \in \Z)$ are independent.\\
At each generation $g\in \Z$, given realisations $C_g = a$, $I_g=j$ and $(K_{(1,g)}, \ldots, (K_{(N,g)})=(k_{(1,g)}, \ldots, k_{(N,g)})$ the labels of the $\sum_{i=1}^N k_{(g,i)}$ potential parents of all individuals in that generation are chosen with replacement from $[N]$ as \emph{iid} random variables with distribution
\begin{equation}
a\delta_{j}+(1-a)\mathcal U_{[N]}.\label{eq:genrepr}
\end{equation}
In addition, conditioned on a realisation of 
$((K_{(g,i)}: i\in[N]): g \in \Z)$, the choices of labels across generations are also independent.
\end{ass}
In other words, the whole population at generation $g$, collectively chooses a distribution ${Q}(Y_g)$, a strength of correlation $C_g$ and marks a special label $I_g$ picked uniformly from $\{1,\ldots,N\}$. Then every individual $v=(i,g)$ will independently sample its number of choices for potential parents $K_{v}$ according to the same $Q(Y_g)$ and subsequently flip $K_{v}$ \emph{iid} coins with a random bias $C_g$. Whenever any such coin returns a success, the $I_g$-th individual from generation $g-1$ is chosen as a potential parent for $v$; for all coins returning failure, the potential parent is picked uniformly at random, independently of all other random variables. \\
By construction, ultimately each individual potential parent is, marginally, uniformly chosen from $[N]$. \\
Notice that, since choices are made \emph{with replacement}, 
the actual number of \emph{distinct} potential parents, chosen by an individual $v\in V_N$, might be  less than $K_v$. Repeated choices will, however, be invisible to the graph, i.e. no multiple edges will be drawn between any two distinct vertices.
We assume that all the above random variables are defined on the same underlying probability space $(\Omega,{\cal F},\mbb P)$. We denote by $[w,v]$ a directed edge from $w$ to $v$, for  $w,v \in V_N$ and by ${\mathfrak N}$ the Borel sigma-algebra of $\N\cup\{\infty\}$ under the discrete topology.

With this, we obtain the definition of the \emph{Wright-Fisher graph with selection in random environment and skewed offspring distribution}. 

\begin{defn}\label{def:DASG} For every $N \in \N$ let $\bar Y = (Y_g)_{g \in \Z}$ be a $[0,1]$-valued process, $\Q:[0,1]\times {\mathfrak N}\to [0,1]$ a probability kernel, $c \in [0,1]$, and $\Lambda$ a distribution on $[0,1]$. 

\noindent The \emph{Wright-Fisher graph with selection $\Q$ in random environment $\bar Y$ and skewed offspring distribution governed by $c$ and $\Lambda$} -- denoted WF$(N,\bar Y, \Q, c, \Lambda)$ -- is given by the graph 
$$G_N=G_N(\omega)=(V_N,E_N(\omega)),\ \ \ \omega\in\Omega,$$
 with deterministic set of vertices $V_N$, and {random} set of edges $E_N=E_N(\omega)$ formed by the rule:
\emph{$[u,v]$ is an edge of $E_N$ if and only if $u$ is a potential parent of $v$, where potential parents are chosen according to Assumptions \ref{ass:sel} and \ref{ass:skewed}. }
\end{defn}
On a given Wright-Fisher graph WF$(N,\bar Y, \Q, c, \Lambda)$, we now introduce the two-point type space $\{0,1\}$, where $0$ will be the \emph{weak} (less fit) type and $1$ the \emph{strong} (fitter) type. The population's allele frequencies dynamics and the corresponding backward-in-time genealogy can then both be obtained as functions of the same WF$(N,\bar Y, \Q, c, \Lambda)$ graph as we see in the following sections.


\subsection{Inheritance and two-type allele frequencies dynamics}\label{sec:forward frequencies}We will assign arbitrarily types $0$ or $1$ to all the individuals in a fixed generation $g_0 \in \Z$, chosen to be the starting generation of our process.\\
The mechanism of inheritance is modelled as follows: in each subsequent generation, every individual will take on the type of its \emph{fittest} potential parent, that is: it will inherit type 0 if and only if \emph{all} its potential parents are of type 0. Clearly, this mechanism describes the advantage of the fitter type 1 and selection is fully explained in terms of multiple parents choice.
\\ This rule assigns types to all vertices in $\{v \in V_N \mid g(v)\geq g_0\}$. Let $\xi(v) \in \{0,1\}$ denote the type of $v$ and note that for every $g>g_0$ the vector of types $(\xi(g,i): i \in [N])$ is exchangeable. Define $[N]_0:=\{0, \ldots, N\}$, $[N]_0/N:=\{0, 1/N, \ldots, 1\}$. 

\begin{defn}[Allele-frequency process]\label{def:discrete_frequency}
The $0$-allele frequency process started in $g_0\in \Z$ in a Wright-Fisher graph \emph{WF($N,\bar Y,\Q, c, \Lambda$)} is the $[N]_0/N$-valued process $X^{N, g_0}=(X^{N, g_0}(g):g\geq g_0)$ describing the proportion of $0$-alleles in each generation $g \geq g_0$:
 \begin{align*}
  X^{N,g_0}(g):= \frac{1}{N}\sum_{\{v:g(v)=g\}} (1-\xi(v)),\ \ g\geq g_0.
 \end{align*}
\end{defn}
The forward-in-time dynamics of the allele frequency process can be read off quite immediately from the graphical construction. For every $y\in[0,1]$, denote by $\varphi_y$ the probability generating function of the distribution $Q(y,\cdot):$
\begin{equation*}
 \varphi_y(x):=\sum_j Q(y,\{j\})x^j.
\end{equation*}

\begin{prop}\label{pr:discretefreq} Conditionally on a realisation of the environment $\bar Y=\bar y\in[0,1]^{\Z}$ the frequency process
$X^{N,g_0}$ is a $[N]_0/N$-valued time-inhomogeneous Markov chain with one-step transition probabilities
\begin{align}
\label{eq:discretefreq}
\P&\left(X^{N,g_0}(g+1)=z\mid X^{N,g_0}(g)=x, \bar Y=\bar y\right)\\
			 & = (1-c){\rm Bin}(Nz\mid N, \varphi_{y_g}(x))
 +c \ \E\left[{\rm Bin}\Big(Nz\mid\ N,\varphi_{y_g}\big(VB_x + (1-V)x\big)\Big)\right].\notag \end{align}
for any $z,x\in[N]/N$, where: ${\rm Bin}(\cdot\mid m,x)$ is the binomial probability mass function with parameter $m,x$; $V \sim \Lambda$; $B_x$ is a Bernoulli random variable with success parameter $x$; and $V$ and $B_x$ are mutually independent. 
\end{prop}
The detailed explanation can,  of course, be found in the proof in Section \ref{subsec:proofs_discrete}, but let us briefly shed some light on the role of the probability generating function in this context: Assume a generation $g$ without skewed offspring distribution. Given $Y_g=y_g\in [0,1]$ as the realisation of the random environment in that generation the individual $(g,1)$ samples the number $K_{(g,1)}$ of choices of potential parents according to $Q(y_g, \cdot)$  and then independently samples as many potential parents from the previous generation. Given that the frequency of weak 0-type is $x$ in the previous generation, the probability for each choice to be of the weak 0-type is $x$. Since this individual will only be of the weak 0-type, if \emph{all} its choices are of the weak 0-type, the probability of this event is precisely $\E[x^{K_{(g,1)}}]=\varphi_{y_g}(x)$.
 
\begin{rem}\label{rem:geometric_important}
Our graph encompasses several models of allele frequency evolution already known in the population genetics literature. Consider $\mathcal Y=[0,1]$. Models without skewed offspring distribution are those where $c = 0$. The classical neutral Wright-Fisher model, for example, is obtained by setting $c = 0$ and $Q(y,\{1\})=1$ (i.e. $Q(y, \cdot) = \delta_{1}$)  for every $y\in[0,1]$. Indeed, in this case the environment does not affect random mating and every individual choose its only one parent independently, uniformly at random, so that  $\varphi_y(x)=x$ for every $x,y$ and the probability \eqref{eq:discretefreq} becomes Binomial with parameter $x$. \\
Focusing still on the case $c = 0$,
a key role is played by the choice of geometric kernels $Q$ given, for every $y\in[0,1[$, by 
\begin{equation*}
 Q(y, \cdot)=\sum_{i=1}^\infty y^{i-1}(1-y)\delta_i=:\text{Geo}(1-y) 
\end{equation*}
and $Q(1, \cdot):=\delta_{\{\infty\}}$. Indeed, with such a choice of $Q$, if $\bar Y$ is taken to be a constant, deterministic process, that is, for some $y\in[0,1[,$ $Y_g=y$ for every $g\in\Z$, then $X^{N,g_0}$ reduces to the ordinary Wright-Fisher model with selection parameter $y$ (see \cite{ADSelection}, Example 2.3) and the neutral Wright-Fisher model corresponds to the case $y=0$. If, with the same choice of $Q$, the environment $\bar Y=(Y_g:g\in \Z)$ is the $N$-th instance of a sequence of processes in the domain of attraction a spectrally negative L\'evy process $\bar Y$, then it can be shown that the frequency process $X^{N,g_0}$ falls within a class of Wright-Fisher models with selection with random environment recently introduced in \cite{caballero}. 
 \end{rem}


\subsection{Backward-in-time ancestral process}\label{sec:discrete_genealogy}
 We now turn to defining the ancestral process arising from a WF($N,\bar Y,\Q, c,  \Lambda$) graph.
An individual $(g(v)-r, i) \in V_N$ living $r\geq 1$ generations before $v$ is called a \emph{potential ancestor of $v$}, if there exists a path  in the graph $(V_N,E_N)$ connecting $v_0:=(g(v)-r, i)$ to $v_r:=v$, that is: there exist $v_1,\ldots,v_{r-1}\in V_N$ such that $[v_l, v_{l+1}]\in E_N$ for every $l=0,\ldots, r-1$. 
Denote by $A^N(v)$ the set of \emph{all} such ancestors of $v$. 
For $n\in[N]$, the  \emph{potential ancestry of the sample $v_1, \ldots, v_n$} is the set $A^N(v_1, \ldots, v_n):=\bigcup_{j=1}^n A^N(v_j)$.\\
Now, given an arbitrary collection $\bar v=\{v_1, \ldots, v_n\} \subseteq \{g_0\}\times [N]$  of distinct individuals chosen from generation $g_0\in\Z$, the \emph{potential ancestors of the sample $\{v_1, \ldots, v_n\} $ alive $l$ generations back in time} are the vertices in the set
 \begin{align*}
  A^{N}_{l}(\bar v):=\{w \in A^{N}(v_1, \ldots, v_n) \mid g(w) = g_0-l \}.
 \end{align*}
 Thus $A^N(\bar v)=\bigcup_{l=1}^\infty A^N_{l}(\bar v)$. \\ 
 The following function of the ancestry will play a key role.
\begin{defn}\label{def:discrete_genealogy}
Let $\bar v \subseteq \{g_0\}\times [N]$ be a sample of individuals from generation $g_0$. The \emph{block-counting process of the potential ancestry of $\bar v$} is the $[N]_0$-valued process $Z^{N,g_0}_{\bar v}= (Z^{N,g_0}_{\bar v}(g): g \leq g_0)$ defined by $ Z^{N,g_0}_{\bar v}(g_0):=|\bar v|$ and 
\begin{align*}
 Z^{N,g_0}_{\bar v}(g):= \vert A^N_{g_0-g}(\bar v)\vert,\ \ \ g\leq g_0,
\end{align*}
where $|A|$ denotes the cardinality of the set $A$.\end{defn}
The block-counting process thus tracks the number of \emph{distinct} potential ancestors of a sample and forgets their labels. 
Note that the distribution of $Z^{N,g_0}_{\bar v}$ actually depends on $\bar v$ only through the sample size $|\bar v|$. We omit the subscript $\bar v$ when there is no danger of confusion.\\
Conditionally on a realisation of the environment $\bar Y=\bar y\in[0,1]^{\Z}$ both the ancestral process $A^N_{g_0-\bullet}(\bar v)$ and its  block-counting process $Z^{N,g_0}_{\bar v}$ are time-inhomogeneous Markov chains.


\subsection{Sampling dualities}\label{sec:discreteduality}

As expected from the construction, the processes introduced in Definition \ref{def:discrete_frequency} and \ref{def:discrete_genealogy} have a duality relation determined by the graph. Indeed, we will identify two dualities (for the same processes): in the first case, the duality function is related  to generating functions, in the second, the duality function is given by a hypergeometric function. Both are so-called \emph{sampling-dualities}, where the duality function is given by a suitable sampling  probability as will be explained in detail below. See \cite{M99} for sampling-dualities for the basic Wright-Fisher  model. \\ 
Note that both processes $X^{N, g_0}$ and $Z^{N, g_0}$ need the additional time-parameter $g_0$ to indicate where they are anchored with respect to the absolute time, i.e.\ the time of the random environment $\bar Y=(Y_g)_{g \in \Z}$. Observe also, that the quantities $X^{N, g_0}(g)$ and $Z^{N, g_0}(g)$ actually depend on the environment process $\bar Y$  only through its coordinates between starting and current time i.e. through the variables  $Y_{g_0+1}, \ldots, Y_{g}$ and $Y_{g+1}, \ldots, Y_{g_0}$ respectively. 

Recall that we denote  by $\varphi_y$ the probability generating function of the distribution $Q(y,\cdot)$, i.e. $\varphi_y(x):=\E[{x^{K_y}}]$ if $K_y \sim Q(y, \cdot)$.
Let $V \sim \Lambda$ and for each $x \in [0,1]$ let $B_x$ be Bernoulli with success parameter $x$ and let $V$ and $B_x$ be independent. With this, define the function $H: [0,1]\times\N_0\times\Z\times [0,1]^{\Z} \rightarrow [0,1]$ by
\begin{align}\label{eq:conditional_sampling_duality_function}
 H(x,n; g, \bar y ):=(1-c)(\varphi_{y_g}(x))^n + c\ \E\left[\Big(\varphi_{y_g}\big(VB_x + (1-V)x\big)\Big)^n\right], 
\end{align}
which will be our (first sampling) duality function in the following
\begin{lemma}[Quenched sampling duality]\label{lem:consampleduality}
Let $X^N$ and $Z^N$ be, respectively, the $0$-allele frequency process and the block-counting process of a Wright-Fisher graph \emph{WF($N,\bar Y,\Q, c, \Lambda$)}.\\	
\noindent Let $H$ be the function given in \eqref{eq:conditional_sampling_duality_function}.
\noindent Then, for all $x \in [N]_0/N$, $n \in [N]_0$ and $r,s \in \Z$ with $r < s $,  $\P$-a.s.
\begin{align}
\E\Big[H\Big(X^{N,r}(s-1),n;& s,\bar Y\Big) \mid X^{N,r}(r)=x,\,\bar Y\Big] \label{eq:sdual}\\
						& = \E\Big[H\Big(x,Z^{N,s}(r+1); r+1,\bar Y\Big)\mid Z^{N,s}(s)=n, \,\bar Y\Big].\notag 
\end{align}
\end{lemma}

\begin{rem}\label{rem:observations_on_sampling_duality} Let us give some  context to this result.
\begin{enumerate}
 \item 
The equality \eqref{eq:sdual} establishes a quenched version of \ ``sampling duality'', where by quenched we mean conditional on the realisation of the environment $\bar Y$. The name refers to the fact that the duality function $H(x,n,g; \bar y)$ is precisely the conditional probability, given $\bar Y=\bar y=(y_g)_{g \in \Z}$, of all individuals in a sample of size $n$ individuals from generation $g$ are of the weak type 0, given that the $0$-type frequency was $x$ in the previous generation. This will be more apparent in the proof of Lemma \ref{lem:consampleduality} in Section \ref{subsec:proofs_discrete}. It is a direct consequence of the graph-construction. 
 \item For $c=0$, i.e. with no {skewed offspring distribution}, and any choice of $\mathbb Q$ and $\bar Y$ resulting in  $K_v= 1$ almost surely for every $v\in V_N$, the duality function in \eqref{eq:conditional_sampling_duality_function} is the moment-function $H(x,n,g; \bar y)=x^n$ and the equality \eqref{eq:sdual} reduces to the well-known \emph{moment duality} between the $N$-finite Wright-Fisher allele frequency process and the corresponding block-counting process of the genealogy \cite{M99}.
 \item For arbitrary $c$ and $\mathbb Q$, but constant, deterministic environment, the \emph{sampling duality} for Wright-Fisher graphs with selection was proved in Proposition 2.9 in \cite{ADSelection}. 
  \item In its general form, the duality function $H$ is time-dependent, but only through the environment process. However some specific choices of environment process (e.g. \emph{iid} environment) induce convenient forms of time-homogeneity. See also Proposition \ref{prop:iidduality} below.
  \item Since both processes $X^N$ and $Z^N$ are constructed on the same random graph, the duality stated in Lemma \ref{lem:consampleduality} is of strong type, i.e. can be made sense of ``almost surely''. See also \cite{ADSelection}, Remark 2.10. 
 \end{enumerate}
\end{rem}

The next Proposition shows that a similar duality property holds in an annealed form (i.e.\ unconditionally on the environment $\bar Y$) when the environment is described by an \emph{iid} process. With the choice of an \emph{iid}  environment, the distributions of $X^N$ and $Z^N$ do not depend on the specific choice of starting time $g_0$, whence we will normally assume, without loss of generality, that $g_0=0$ and omit the corresponding superscript. Write
\begin{equation*}
 \P^n(\,\cdot\,) := \P( \,\cdot\,\mid Z^{N,g_0}(g_0) = n)\quad \text{ and }  \quad \P_x(\, \cdot\,):= \P(\,\cdot\,\mid X^{N, g_0}(g_0)=x)
\end{equation*}
for any $x \in [N]_0/N$ and $n \in [N]_0$.

Again, let $V \sim \Lambda$ and for each $x \in [0,1]$ let $B_x$ be Bernoulli with success parameter $x$, independent of $V$. If the random environment $\bar Y=(Y_g)_{g \in \Z}$ is a sequence of \emph{iid} random variables with common distribution $\mu$, define the function $H_{\mu}: [0,1]\times\N \rightarrow [0,1]$ by
\begin{align}\label{eq:iid_duality_function}
 H_{\mu}(x,n)&:= \mathbb{E}[H(x,n;0,\bar Y)]	\notag\\
		& =(1-c)\E[\varphi_{Y_0}(x)^n]   + c\ \E\left[\varphi_{Y_0}\big(VB_x + (1-V)x\big)^n\right],
\end{align}
where $H$ is the duality function defined  in \eqref{eq:conditional_sampling_duality_function} and used in Lemma \ref{lem:consampleduality}.

\begin{prop}\label{prop:iidduality}
 Assume that the random environment $\bar Y=(Y_g)_{g \in \Z}$ is a sequence of \emph{iid} random variables with common distribution $\mu$. Let $H_{\mu}$ be the function given in \eqref{eq:iid_duality_function}
Then for all $x \in [N]_0/N$, $n \in \N_0$ and $g\geq 0$
\begin{align}
\E_{x}\left[H_{\mu}\left(X^{N,0}(g),n\right)\right] = \E^{n}\left[H_{\mu}(x,Z^{N,0}\left(-g)\right)\right].
\label{eq:ucdual}
\end{align}
\end{prop}

This proposition is interesting in its own right,  but will also be helpful in proving the convergence of the scaling limits of the ancestral process in Theorem \ref{thm: ancestral converg}. Note that analogous observations to those in Remark  \ref{rem:observations_on_sampling_duality} also hold  for $H_{\mu}$ defined in \eqref{eq:iid_duality_function} and Proposition \ref{prop:iidduality}.

As in \cite{M99} we  can also observe a second duality from the graphical construction. For evey $N \in \N$,  define a function $\tilde H_N: [0,1]\times\N_0 \rightarrow [0,1]$ as
\begin{align}\label{eq:diskretesamplingdualitywithoutreplacemen_functiont}
 \tilde H_N(x,n):=\prod_{k=1}^n\frac{xN-k+1}{N-k+1}.
\end{align}

\begin{prop}[Hypergeometric sampling duality]\label{prop:hypergeometricsampling}
 Let $X^N$ and $Z^N$ be, respectively, the $0$-allele frequency process and the block-counting process of a Wright-Fisher graph \emph{WF($N,\bar Y,\Q, c, \Lambda$)}.\\	
\noindent Let $\tilde H_N$ be the function given in \eqref{eq:diskretesamplingdualitywithoutreplacemen_functiont}.
\noindent Then, for all $x \in [N]_0/N$, $n \in [N]_0$ and $r,s \in \Z$ with $r < s $, 
\begin{align}\label{eq:diskretesamplingdualitywithoutreplacement}
\E\Big[\tilde H_N\Big(X^{N,r}(s),n\Big) \mid X^{N,r}(r)=x\Big] = \E\Big[\tilde   H_N\Big(x,Z^{N,s}(r)\Big)\mid Z^{N,s}(s)=n\Big].
\end{align}
\end{prop}

In this duality, the function does not depend on the random environment even without the assumption of an \emph{iid}  environment. The effect of the random environment is present solely through its effect on the distribution of $X^{N,r}(s)$ and $Z^{N,s}(r)$, respectively. In contrast, the duality function given in \eqref{eq:conditional_sampling_duality_function} does not depend on the population size $N$. 

As described in \cite{M99},  both dualities are to be expected when we have a Wright-Fisher graph and both are \emph{sampling dualities} in the sense that the duality function can be expressed as the probability of sampling $n$ weak individuals. In the case of  \eqref{eq:conditional_sampling_duality_function}, we consider this event conditionally on the frequency of weak alleles in the \emph{previous} generation being $x$, thus giving rise to a ``sampling with replacement'' which yields the probability geerating functions, while in the case of \eqref{eq:diskretesamplingdualitywithoutreplacemen_functiont}, the event is considered conditionally on the frequency of weak alleles in the \emph{same} generation. This also explains the explicit appearance of the random envronment in \eqref{eq:conditional_sampling_duality_function}, but not in \eqref{eq:diskretesamplingdualitywithoutreplacemen_functiont}. For more details see the proofs of the statements in Section \ref{subsec:proofs_discrete}.

\section{Scaling limit processes}\label{sec:scaling limit}

 From now on we will focus solely on models with \emph{iid} random environment and always consider annealed results, i.e. we do not condition on the random environment, but rather average it out.

\subsection{Forward scaling limit in \emph{iid}\ environment: rare selection}\label{subsec:diffusionconvergence}
 
 We will study existence of a continuous-time Markov process $X=(X(t) : t \geq 0)$ defined as the solution to the SDE
\begin{align}\label{eq:SDEX}
\dd X(t) 	& = \int (\E[X(t^-)^{K_y}\,\mid\, X(t^-)]-X(t^-) )\ \mathcal N_{\mathfrak s}(\dd t,\dd y) - wX(t)(1-X(t))\dd t \notag\\
		& \qquad \qquad + c\int\int z(\1_{\{u \leq X(t-)\}} - X(t-))\ \mathcal N_{\mathfrak c}(\dd t, \dd z, \dd u) \notag \\
		& \qquad \qquad \qquad \qquad +\sqrt{\sigma X(t)(1-X(t))}\dd B(t)
\end{align}
where  $(B(t) : t\geq 0 )$ is a Brownian motion, $\mathcal  N_{\mathfrak c}$ is a Poisson point measure on $[0,\infty[\times]0,1]\times]0,1[$ with intensity $\dd s\otimes z^{-2}\Lambda_{\mathfrak c}(\dd z)\otimes \dd u$, where $\Lambda_{\mathfrak c}$ is a distribution on $[0,1]$ with no atom in 0, $c, w \geq 0$, $\mathcal N_{\mathfrak s}$ is a Poisson point measure on $[0,\infty[\times [0,1]$ with intensity $\dd s\otimes\mu(\dd y)$ and  $K_y$ is an $\N\cup\{\infty\}$-valued random variable with distribution $Q(y)$ parametrised by $y\in [0,1]$.
\\Note that, if $X$ exists,  the compensator of $\mathcal  N_{\mathfrak c}$ is zero and Condition \ref{MasterCondition} guarantees that the compensator of $\mathcal  N_{\mathfrak s}$ is finite. As a result,  $X$ solves the martingale problem associated with the generator $\mathcal A$,  given by
\begin{align}\label{eq:generatorX}
\mathcal Af(x)	& =\int_{[0,1]}\left[ f(\E[x^{K_y}])-f(x)\right]\ \mu(\dd y) - wx(1-x)f'(x) \notag \\
		& \qquad +c\int_{[0,1]} \big\{x\left[f(x(1-z)+z)-f(x)\right]+(1-x)\left[f(x(1-z))-f(x)\right]\big\}\frac{1}{z^2}\Lambda_{\mathfrak c}(\dd z) \notag \\
		& \qquad \qquad +\sigma x(1-x)\frac{f''(x)}{2}
\end{align}
for every $C^2$ function $f:[0,1]\rightarrow \R.$

\begin{lemma}\label{lem:SDEXsolution}
Assume Condition \ref{MasterCondition}. Then there exists a unique strong solution to \eqref{eq:SDEX}.
\end{lemma}

\begin{defn}\label{def:X_weak_and_rare}
 For a distribution $\Lambda_{\mathfrak c}$ on $[0,1]$ with $\Lambda_{\mathfrak c}(\{0\})=0$, a measure $\mu$ on $[0,1]$ and a kernel $\Q$ satisfying Condition \ref{MasterCondition} and constants $w, c, \sigma \geq 0$, define the \emph{two-type Fleming-Viot process with weak and rare selection (FVWRS) parametrised by $\Q$, $\mu$ and $w$ and $\Lambda_{\mathfrak c}$, $c$ and $\sigma$} as the unique strong solution $X=(X(t):t\geq 0)$ to \eqref{eq:SDEX}.
\end{defn}

We have now all the elements to prove the convergence of the 0-allele frequency process in the Wright-Fisher graph to the solution $X=(X(t):t\geq 0)$ of \eqref{eq:SDEX}. For clarity, we will first focus on the case of pure \emph{rare} selection, i.e.\ with $w=0$. In the pre-limit model, the parameters $\mu,c$ and the random environment process $\bar Y$ will now be indexed by the corresponding population size $N$: $\mu_N, c_N, \bar Y^N$. We will choose an \emph{iid}\ environment with a distribution $\mu_N$ assigning a large weight on the event of \emph{no selection} (hence we have rare selection) but such that, when selection occurs, its strength does not scale with $N$. Recall that we consider annealed results, here, i.e. not conditional on the environment, but rather averaged over it. 

\begin{thm}\label{lem:converges}
Let $(\rho_N)_{N\in\N}$ be a sequence of positive numbers converging to 0.\\
For a constant $\sigma \in [0, \infty[$ assume 
\begin{align}\label{eq:assumption_for_WF}
 \lim_{N\rightarrow \infty}\rho_N^{-1} N^{-1}=\sigma.
\end{align}
Let $\Lambda_{\mathfrak c}$ be a distribution on $[0,1]$ with $\Lambda_{\mathfrak c}(\{0\})=0$ and $(c_N)_{N \in \N}$ a sequence in $[0,1]$ converging to 0.
For an $\alpha \in (0,1/2)$ define
\begin{align*}
 \Lambda^{\alpha}_N(\dd u):= u^{-2}\1_{[N^{-\alpha}, 1]}(u)\Lambda_{\mathfrak c}(\dd u) \quad \text{ and } \quad \bar\Lambda^{\alpha}_N:=\frac{1}{\Lambda^{\alpha}_N([0,1])}\Lambda^{\alpha}_N
\end{align*}
and assume that for a constant $c \in [0, \infty[$\,,
  \begin{align}\label{eq:assumptions_on_skewedoffspring}
   \lim_{N\rightarrow \infty}\rho_N^{-1} \frac{c_N}{\Lambda^{\alpha}_N([0,1])}=c.
  \end{align}
  Let $\Q$ be a probability kernel from $[0,1]$ to $\N\cup\{\infty\}$ and $\mu$ a measure in $[0,1]$ such that Condition \ref{MasterCondition} holds.
In particular, $\mu$ is $\sigma$-finite, so let  $(I_N)_{N\in \N}$ be an increasing sequence of sets such that we can define finite measures
\begin{align*}
\gamma_N(\cdot):=\mu (\cdot \cap I_N)\quad  \text{  and  }\quad  \bar\gamma_N := \frac{1}{\gamma_N([0,1])}\gamma_N
\end{align*}
and such that $\gamma_N([0,1]) \rho_N$ converges to 0.\\
Lastly, assume
\begin{align}\label{eq:unnecessary_assumption}
 \lim_{N\rightarrow \infty}\gamma_N([0,1])c_N = 0.
\end{align}
For each $N \in \N$, let $X^N$ be the 0-allele frequency process started in 0 in a $WF( N, \bar Y^N, \Q, c_N, \bar\Lambda^{\alpha}_N)$ graph as introduced in Definition \ref{def:discrete_frequency}, where $\bar Y^N=(Y^N_g)_{g \in \Z}$ is an \emph{iid} environment with common distribution 
 \begin{align*}
  \bar\mu_N:=(1-\gamma_N([0,1])\rho_N)\delta_0+\gamma_N([0,1])\rho_N \bar\gamma_N.
 \end{align*}
If $X^N_0 \xrightarrow{\;w\;} x\in[0,1]$, as $N\to\infty,$
\begin{align*}
\left(X^N\left(\left\lfloor\rho_N^{-1}t\right\rfloor\right): t \geq 0 \right)\; \Longrightarrow\; (X(t):t\geq0), 
\end{align*}
where $X=(X(t):t\geq0)$ is the two-type Fleming-Viot process with rare selection, i.e.\ the unique strong solution to \eqref{eq:SDEX} (for $w=0$) started in $X_0=x$. 
\end{thm}
Recall that, due to the \emph{iid} property of the environment, the distribution of $X^N$ does not depend on the starting time-point $g_0=0$. 

 Observe also that the mechanisms for skewed  offspring distribution and rare selection are very similar. The role of $c_N$ in the skewed offspring distribution is played by $\gamma_N([0,1])\rho_N$ for rare selection: they both describe the probability of an extreme, but rare event. In both cases we truncate a possibly infinite measure -- $y^{-1}\Lambda_{\mathfrak c}(\dd y)$ and $\mu$ respectively -- to obtain distributions for the finite population model. The analogue of assumption \eqref{eq:assumptions_on_skewedoffspring} for the case of rare selection also holds trivially as it is given by
 \begin{align*}
  \rho_N^{-1}\frac{\gamma_N([0,1])\rho_N}{\gamma_N([0,1])}=1
 \end{align*}
 The role of assumption \eqref{eq:unnecessary_assumption} becomes more clear when rewritten as
 \begin{align*}
  \lim_{N\rightarrow \infty}\frac{\gamma_N([0,1])\rho_Nc_N}{\rho_N} = 0
 \end{align*}
which uncovers that \eqref{eq:unnecessary_assumption} ensures that the probability of both rare events occurring simultaneously converges to 0 sufficiently quickly, in the appropriate time rescaling. Such an assumption simplifies the proof and it could be possibly weakened but this is beyond the scope of this paper.  
 
\begin{rem}[$\Lambda$-selection ] \label{rem:delta0} Recall from the introduction in Section \ref{subsec:DASG} that the measure $\Lambda_{\mathfrak c}$, is a distribution on $[0,1]$ with no atom in 0 which describes the extreme events resulting from a skewed offspring distribution through $y^{-2}\Lambda_{\mathfrak c}(\dd y)$ and recall also from the discussion after Condition \ref{MasterCondition}, that $\Lambda_{\mathfrak s}$ is likewise a finite measure on $[0,1]$ with no atom in 0 describing the selection mechanism through $\mu(\dd y) = (\E[K_y -1])^{-1}\1_{]0,1]}(y)\Lambda_{\mathfrak s}(\dd y)$. The extent of the parallel between these two mechanisms becomes more transparent with the following observations. If for a sequence $\Lambda^n_{\mathfrak c},$ $n\in \N$, of distributions we have $\Lambda^n_{\mathfrak c}\rightarrow \sigma\delta_0$ weakly,  as $n \rightarrow \infty$, (in the Prohorov-metric, for some $\sigma>0$), then the component of the generator in \eqref{eq:generatorX} corresponding to these jumps converges
\begin{align*}
\int_{[0,1]} \big\{x\left[f(x(1-z)+z)-f(x)\right]+(1-x)&\left[f(x(1-z))-f(x)\right]\big\}\frac{1}{z^2}\Lambda^n_{\mathfrak c}(\dd z)\\
							& \qquad \qquad \qquad   \xrightarrow{n\rightarrow \infty} \sigma \frac{1}{2}x(1-x)f''(x),
\end{align*}
uniformly in $x\in [0,1]$. The limit is the generator of the Wright-Fisher diffusion. For this reason it is common to find in the literature the interpretation that identifies the atom $\Lambda_{\mathfrak c}(\{0\})$ with the intensity $\sigma$ of the \lq\lq Wright-Fisher noise\rq\rq\  or, as will be seen later, of the \lq\lq Kingman component\rq\rq\ of the dual coalescence-mechanism (see for example \cite{Sch00} or, in a spatial set-up, in \cite{BEV10}).

Naturally, the question arises whether an analogous interpretation holds for the selection mechanism, i.e. whether in this sense rare selection converges to weak selection when the measures governing selection converge to a Dirac-measure in 0 and we show here that this is indeed the case. Let $\Lambda^n_{\mathfrak s}$, $n \in \N$ be a sequence of finite measures on $[0,1]$ such that $\Lambda^n_{\mathfrak s} \rightarrow w\delta_0$ weakly as $n \rightarrow \infty$ in the Prohorov-metric for a $w>0$. Let $\mathcal A^n_{\mathfrak s}$ be the summand in the generator \eqref{eq:generatorX} responsible for the rare selection mechanism induced by $\Lambda^n_{\mathfrak s}$. Consider the kernel given by $Q(y)  = (1-y)\delta_1 + y \delta_2$, $y \in [0,1]$, i.e.\ the mechanism with at most binary branching. Then Condition \ref{MasterCondition} holds automatically and implies that  $\mu$ has density $1/y$ with respect to $\Lambda^n_{\mathfrak s}$ on $(0,1]$.
Applying Taylor's expansion around $x$ reveals
\begin{align*}
A_{\mathfrak s}f(x)=\int_{[0,1]} \left\{f(x)-x(1-x)yf'(x)+{\cal O}(y^2) - f(x)\right\} & \frac{1}{ y}\Lambda^n_{\mathfrak s}(\dd y)\\
										& \hspace{-5pt} \xrightarrow{\;n \rightarrow \infty\;} -wx(1-x)f'(x)			
\end{align*}
which is the component responsible for genic (\emph{weak}) selection in \eqref{eq:generatorX}. The same result holds true for the choice of \emph{geometric} $\Q$ (see Remark \ref{rem:geometric_important}) and one might believe this to be true for every general kernel. Under some mild assumptions this is true, if one allows for a more general class of frequency dependent (weak) selection as studied in \cite{ADSelection}. The following conditions on the kernel $\Q=(Q(y):y \in [0,1])$, ensure convergence of the generator for rare selection to the generator of genic (weak) selection 
\begin{ass} \label{ass: weak conv}
Assume $y \mapsto Q(y, A)$ is continuous for every $A \in \mathfrak N$ and that
\begin{equation}\label{eq:nsw}
 \sum_{k=2}^{\infty} \lim_{y \downarrow 0} \frac{\P(K_y=k)}{\E[K_y-1]} = \lim_{y \downarrow 0}\;\frac{\P(K_y\geq 2)}{\E[K_y-1]}. 
\end{equation}
\end{ass}
\noindent Then, as we show below, for any $f \in C^2([0,1])$
\begin{align}\label{eq:convergencetoatom}
 A_{\mathfrak s}f(x)	& =\int_{[0,1]}	\left[f(\E[x^{K_y}])-f(x)\right]\ \mu(\dd y)	\xrightarrow{n\rightarrow \infty} -ws(x)x(1-x)f'(x)
\end{align}
uniformly in $x \in [0,1]$, where $s(x) = \sum_{k=2}^{\infty}\eta(k)\sum_{l=0}^{k-2}x^l$ (for $\eta(k)$ defined in \eqref{eq:definition_of_etak}), which are the selection coefficients studied in \cite{ADSelection}. The observations below also immediately  imply that if if asymptotically when approaching 0 the kernels put most of the weight on binary branches, i.e.
\begin{align*}
 \lim_{y \downarrow 0}\;\frac{\P(K_y=2)}{\E[K_y-1]} = 1 
\end{align*}
then $\eta(2)=1$ and $\eta(k)=0$, $k\geq 3$, i.e.\ and $s(x) \equiv1$ and the limit is classig weak selection.

Let us now see why \eqref{eq:convergencetoatom} holds. 
Using Taylor's expansion of order 1 around $x$ we see
\begin{align*}
A_{\mathfrak s}f(x)&=\int_{[0,1]}  \left[f(\E[x^{K_y}])-f(x)\right]\ \mu(\dd y)	 =\int_{[0,1]} \left[f(\E[x^{K_y}])-f(x)\right] \frac{1}{\E[K_y-1]}\ \Lambda^n_{\mathfrak s}(\dd y)\\
		& = \int_{[0,1]} (\E[x^{K_y}]-x)f'(x)\frac{1}{\E[K_y-1]}\ \Lambda^n_{\mathfrak s}(\dd y) \\
		& \qquad \qquad + \int_{[0,1]} \frac{1}{2}(\E[x^{K_y}]-x)^2f''(\xi_x)\frac{1}{\E[K_y-1]}\ \Lambda^n_{\mathfrak s}(\dd y) 
\end{align*}
for a suitable $\xi_x \in [\E[x^{K_y}],x]$. We can now consider the two integrals separately and show that the first converges to the desired limit, while the second vanishes. Note that in order to use weak convergence of distributions,  we need functions that are continuous in $y\in[0,1]$. The first part of Assumption \ref{ass: weak conv} ensures that the expressions that we will encounter are continuous for $y \in (0,1]$, but $y=0$ needs particular care due to the term $1/\E[K_y-1]$. (Recall that $ Q(0)=\delta_1$.) The following two observations are useful:
\begin{align}
 &\P(K_y \geq 2)\leq\sum_{k=1}^{\infty}(k-1)\P(K_y=k)\leq \E[K_y-1]\label{eq:weakconvergence_unnecessaryassumption}\\
 &\vert \E\left[x^{K_y}-x\right]\vert 	 = \vert x(\P(K_y=1)-1)+\sum_{k=2}^{\infty}x^k\P(K_y=k)\vert \leq 2\P(K_y\geq 2).\label{eq:weakconvergence_simpleobservation}
\end{align}
With this we easily show that the remainder term vanishes:
\begin{align*}
 \bigg\vert\int_{[0,1]}& (\E[x^{K_y}]-x)^2 f''(\xi_x)\frac{1}{\E[K_y-1]}\ \Lambda^n_{\mathfrak s}(\dd y)\bigg\vert 
					   \overset{\eqref{eq:weakconvergence_simpleobservation}}{\leq} 2\Vert f''\Vert_{\infty}\int_{[0,1]}\P(K_y\geq 2) \Lambda^n_{\mathfrak s}(\dd y)\\
					  & \xrightarrow{n\rightarrow \infty} 
					  2\Vert f''\Vert_{\infty}\P(K_0 \geq 2)w = 0,
\end{align*}
since $y \mapsto \P(K_y\geq 2)$ is continuous in $y=0$. Similarly,
\begin{align*}
 \int_{[0,1]} (\E[x^{K_y}]-x)	& f'(x)\frac{1}{\E[K_y-1]}\ \Lambda^n_{\mathfrak s}(\dd y) \xrightarrow{n\rightarrow \infty} wf'(x)\lim_{y\downarrow  0}\frac{\E[x^{K_y}]-x}{\E[K_y-1]}
\end{align*}
which we know exists and is bounded by 1 by \eqref{eq:weakconvergence_unnecessaryassumption}. To calculate this limit, we define for $y>0$ 
\begin{align*}
  \eta^y(1):= 1- \frac{\P(K_y\geq 2)}{\E[K_y-1]}, \quad  \eta^y(k):=\frac{\P(K_y=k)}{\E[K_y-1]}, k\geq 2
\end{align*}
and likewise
\begin{align}\label{eq:definition_of_etak}
 \eta(1):= 1- \lim_{y \downarrow  0}\frac{\P(K_y\geq 2)}{\E[K_y-1]}, \quad  \eta(k):=\lim_{y \downarrow  0}\frac{\P(K_y=k)}{\E[K_y-1]}, k\geq 2.
\end{align}
Then $\eta^y:=(\eta^y(k))_{k \in \N}$ and $\eta:=(\eta(k))_{k \in \N}$ are distributions (the latter by Assumption \eqref{eq:nsw}) such that $\lim_{y \downarrow 0}\eta^y(k)= \eta(k)$ for every $k \in \N$. Since thus $\eta^y$ converge weakly to $\eta$, 
\begin{align*}
 \lim_{y\downarrow  0}\frac{\E[x^{K_y}]-x}{\E[K_y-1]} =  \lim_{y\downarrow  0}\E_{\eta^y}[x^K-x] = \E_{\eta}[x^K-x] = \sum_{k \in \N}(x^k-x)\eta_k = -x(1-x)s(x)
\end{align*}
proving \eqref{eq:convergencetoatom}.
\end{rem}

We may, however, also obtain the diffusion with \emph{weak} selection from the particle system directly. In order to obtain both weak and rare selection simultaneously, we have to carefully design our random environment. $\mu$ will still be responsible for the \emph{rare} mechanism only. The kernel of distributions will now consist of distributions from two sets: a general one describing the rare selection and a set of geometric distributions describing weak selection. As before, in the particle system we will have to ensure to rarely make use of the former, while we weaken the effect of the latter as $N\rightarrow \infty$. 
In order to accomodate the different scalings of the mechanisms without overcomplicating notation, we exceptionally allow the environment to take values in $[-1,1]$.

\begin{thm}\label{thm:convergence_rare_and_weak}
Retain all the assumptions and notation of Theorem \ref{lem:converges}. For a sequence $(w_N)_{N\in\N}$ in $[0,1[$ and a constant $w\in[0, \infty[$\, assume 
  \begin{align*}
 \lim_{N\rightarrow \infty}\rho_N^{-1} w_N=w.
  \end{align*}
Let $\tilde \Q:=(Q(y): y \in ]-1,0[)$ be such that $ Q(y)=\text{Geo}_{\N}(1+y)$, for $y \in [-1,0[$.

For each $N \in \N$, let $X^N$ be the 0-allele frequency process in a $WF(N,\bar Y^N,\tilde\Q\cup\Q, c_N, \bar\Lambda^{\alpha}_N)$ graph
and \emph{iid} $[-1,1]$-valued environment $\bar Y^N$ with common distribution 
 \begin{align*}
  \mu_N:=(1-\gamma_N([0,1])\rho_N)\delta_{-w_N}+\gamma_N([0,1])\rho_N \bar\gamma_N.
 \end{align*}
Then, if $X^N_0 \xrightarrow{\;w\;} x\in [0,1]$, as $N\to\infty,$
\begin{align*}
\left(X^N\left(\left\lfloor\rho_N^{-1}t\right\rfloor\right): t \geq 0 \right)\; \Longrightarrow\; (X(t):t\geq0), 
\end{align*}
where $X=(X(t):t\geq0)$ is the two-type Fleming-Viot process with weak and rare selection, i.e.\ the unique strong solution to \eqref{eq:SDEX} with $X_0=x$. 
 \end{thm}


\subsection{Moment duality and convergence to the branching coalescing process in random environment}\label{subsec:duality_convergence} 
As foreshadowed in the introduction, there is a genealogical process arising as  the moment dual to the scaling limit of the frequency process of a WF graph.

\begin{defn}\label{def:BCPRE}
For a distribution $\Lambda_{\mathfrak c}$ on $[0,1]$ with $\Lambda_{\mathfrak c}(\{0\})=0$ and for a measure $\mu$ on $[0,1]$ and a kernel $\Q$ satisfying Condition \ref{MasterCondition}, define the \emph{branching coalescing process in random environment} $Z=(Z(t):t\geq 0)$ with branching intensity $(Q,\mu,w)$ and coalescing intensity $(\Lambda_{\mathfrak c}, c, \sigma)$ as the continuous time $\N\cup\{\infty\}$-valued Markov chain with positive jumps from $n\in \N$ to $n+k$ at rate
\begin{equation*}
\int_{[0,1]} \Prob{\sum_{j=1}^nK_{y,j}=n+k} \mu (\dd y) + wn\delta_{1,k}
\end{equation*}
(where $K_{y,1},\ldots,K_{y,n}$ are \emph{iid} with distribution $Q(y)$) and with negative jumps from $n$ to $n-k\in \N$ occurring at rate
\begin{equation*}
c\int_{[0,1]}\binom{n}{k+1}  y^{k+1}(1-y)^{n-k-1} \frac{\Lambda_{\mathfrak c}(\dd y)}{y^2} + \sigma\binom{n}{2}\delta_{1,k},
\end{equation*}
and entrance law $\E^{\infty}[x^{Z(t)}]:=\P_x(X(t)=1)$, where $(X(t):t\geq 0)$ is the two-type Fleming-Viot process with rare and weak selection parametrised by $\Q$, $\mu$, $w$ and $\Lambda_{\mathfrak c}$, $c$ and $\sigma$.
\end{defn}

The first important result about $Z$ is its conservativeness and the Feller property with respect to the topology of the harmonic numbers. Denote by $C$ the set of continuous functions from the metric space $\mathbb{N}\cup \{\infty\}$ with $d(n,m)=|\frac{1}{m}-\frac{1}{n}|$ to itself.

\begin{lemma}\label{lem:fellerZ}
Under Condition \ref{MasterCondition}, the BCPRE $Z=(Z(t): t\geq0)$ given in Definition \ref{def:BCPRE} is Feller with respect to the topology of the harmonic numbers on $\N\cup\{\infty\}$ and conservative. More precisely, 
 \begin{align*}
  \forall \, n \in \N,\quad \forall \, t\geq 0:\qquad \E^n[Z_t]<\infty.
 \end{align*}
\end{lemma}

The Feller property of $Z$ is critically helpful in proving that $Z$ is the moment dual of $X$.

\begin{lemma}[Moment duality]\label{lem:duality}
Under Condition \ref{MasterCondition} let $X=(X(t) : t\geq 0)$ be the {two-type FV-process with rare and weak selection} solution to \eqref{eq:SDEX} and let $Z=(Z(t):t \geq 0)$ be the {branching coalescing process in random environment} from Definition \ref{def:BCPRE}. 
 
 For any $x \in [0,1]$, $n \in \N\cup\{\infty\}$ and $t \geq 0$ we have
 \begin{align*}
  \E_x[X(t)^n]=\E^n[x^{Z(t)}].
 \end{align*} 
\end{lemma}

As an important application of duality, the Feller property for the process $X$ follows as a simple consequence of $Z$ being conservative. This property is used in the proof of Theorem \ref{thm:convergence_rare_and_weak}.

\begin{lemma}\label{lem:fellerX}
 Under Condition \ref{MasterCondition} the two-type Fleming-Viot-process with weak and rare selection defined in Definition \ref{def:X_weak_and_rare} is Feller.
\end{lemma}

A second and important consequence of the moment duality is, finally, the convergence of the ancestral process to the BCPRE from Definition \ref{def:BCPRE}.
\begin{thm}\label{thm: ancestral converg}
Assume the conditions of Theorem \ref{thm:convergence_rare_and_weak} and let $Z^N$  be the {ancestral process on the Wright-Fisher graph with selection in random environment and multiple mergers $WF(N,\bar Y^N,\Q\cup\tilde\Q, c_N, \Lambda_{\mathfrak c})$} for an \emph{iid} environment $\bar Y^N=(Y^N_g)_{g \in \Z}$ with common distribution 
 \begin{align*}
  \mu_N:=(1-\gamma_N([0,1])\rho_N)\delta_{-w_N}+\gamma_N([0,1])\rho_N \bar\gamma_N.
 \end{align*}
 Then 
\begin{align*}
\left(Z^N\left(-\lfloor\rho_N^{-1}t\rfloor\right)\right):t\geq 0)\Longrightarrow (Z(t):t\geq0),
\end{align*}
where $Z=(Z(t):t\geq0)$ is the {branching coalescing process in random environment} given in Definition \ref{def:BCPRE}, when we equip $\N$ with the topology of the harmonic numbers, by considering the distance $d(m,n)=|\frac{1}{n}-\frac{1}{m}|$ for $n,m \in \N\cup\{\infty\}$ with $\frac{1}{\infty}:=0$. 
\end{thm}


\section{Long term behaviour} \label{sec:longterm}
We finally come to consider the long-term behaviour or ergodicity properties of the \emph{two-type Fleming-Viot process with weak and/or rare selection} $X=(X(t) : t\geq 0)$ (Def. \ref{def:X_weak_and_rare}) and the \emph{branching coalescing process in random environment} $Z=(Z(t) : t\geq 0)$ (Def. \ref{def:BCPRE}).

\noindent Throughout this section assume $\Q$ and $\mu$ to be such that Condition \ref{MasterCondition} holds. In order to avoid trivialities, assume also that $\mu(]0,1])+w+c+\sigma >0$. 

Taking a closer look at \eqref{eq:SDEX}, we see that the {two-type FV-process for rare and/or weak selection} $X=(X(t) : t\geq 0)$ is a (bounded) supermartingale, since the last two terms -- corresponding to the neutral genetic drift -- each yield martingales, while the first two terms -- corresponding   to selection -- give a downward drift and downward jumps, respectively. Hence $X$ converges $\P$-a.s.\ to a random variable, which we will name $X_{\infty}$. The distribution of this random variable is not only of mathematical interest, but also of biological relevance as it encodes the probabilities of fixation or extinction of the weak allele or, a priori, coexistence of the two  types traced. As expected, coexistence can be ruled out almost surely, in this case as a direct consequence of the duality between $X$ and the {branching coalescing process in random environment} $Z$: duality does indeed uncover how the weak allele's chance of survival (and thus fixation) depends on the ergodic properties of $Z=(Z(t) : t\geq 0)$, as we point out in the following remark.

\begin{rem}\label{rem:Xto01}
 Applying the same arguments as in Lemma 4.7 of \cite{ADSelection}, the duality obtained in Lemma \ref{lem:duality} implies:
 \begin{enumerate} 
 \item If $Z$ is \emph{positive recurrent}, then it has a unique invariant distribution $\nu$ and 
  \begin{align*}
  \forall \, x \in [0,1]: \; \P_x\left(X_{\infty} \in \,\cdot\,\right) = (1-\varphi_{\nu}(x))\delta_0 +\varphi_{\nu}(x)\delta_1,
 \end{align*}
  where $\varphi_{\nu}$ is the probability generating function of $\nu$. \label{item:posrec}
  \item If, on the other hand, $Z$ is \emph{not positive recurrent}, then 
  \begin{align*}
   \forall \, x \,\in\,[0,1[\ :\; \P_x\left(X_{\infty} = 0\right) = 1.
   \end{align*}
 \end{enumerate}
 
 \noindent In other words, the duality implies that $X_{\infty} \in \{0,1\}$ $\P_x$-a.s.\ for any $x \in [0,1]$ and the weak allele only has a positive probability of fixation if $Z$ is positive recurrent. In this case, the probability generating function of the invariant distribution of $Z$ determines this probability. 
 
 Note also that the dichotomy in particular implies that if there exists an $x \in [0,1[$ such that $\P_x\left(X_{\infty} = 0\right) < 1$, then this must indeed hold for all $x \in [0,1[$.
\end{rem}
 Naturally, the question arises how the chances of survival of the weak allele depend on the strength of selection. We answer this with Theorem \ref{thm:XlongtermGriffiths} below and can use the above observations to consequently deduce the ergodic behaviour of the branching coalescing process in random environment in Corollary \ref{cor:branching}.


\subsection{Probability of fixation of the weak allele}\label{subsec:longterm_X} 

The mechanisms of selection and genetic drift compete in their influence on the probability of fixation (or extinction) of the weak allele. The following quantities are at play: 
\\
(i) the {\em strength of  genetic drift}, characterised by $\sigma$ and the pair $c$ and 
\begin{align}\label{eq:beta*}
 \beta^*:= \frac{1}{2}\E\left[\frac{1}{W(1-W)}\right],
\end{align}
respectively. Here, $W:=Y^{\mathfrak c} U$, where $Y^{\mathfrak c}\sim \Lambda_{\mathfrak c}$, $U$ has density $2u$ on $[0,1]$ and they are independent; \\
(ii) the {\em strength of weak selection}, given by $w,$ and that {\em of rare selection}, given by the pair $\alpha_{\mathfrak s} :=\Lambda_{\mathfrak s}([0,1])$ and 
\begin{align}\label{eq:alpha*}
 \alpha^*:=\E\left[\frac{1}{1+V\E\left[K_{Y^{\mathfrak s}}-1\,\big|\, Y^{\mathfrak s}\right]}\right],
\end{align}
where $V$ is uniform on $[0,1]$ and $Y^{\mathfrak s} \sim \alpha_{\mathfrak s}^{-1}\Lambda_{\mathfrak s}([0,1])$ and they are independent.

As formalised in the following theorem, if the total strength of selection is sufficiently strong, it overcomes the genetic drift and the weak allele will become extinct almost surely. If not, it retains a positive probability of survival even in the presence of selection.
\begin{thm}\label{thm:XlongtermGriffiths} 
Let $X=(X(t) : t\geq 0)$ be the two-type FV-process for weak and rare selection given by \eqref{eq:SDEX} with $\sigma=0$. 
 Define
 \begin{align*}
  p(x):=\P_x(X_{ \infty} = 0) 
 \end{align*}
 to be the probability of extinction of the weak allele 0, given that we start with a frequency $x \in [0,1]$. 
 
 Assume $\beta^*<\infty$. 
 \begin{enumerate}
  \item If $\alpha_{\mathfrak s}\alpha^*+w<c\beta^*$, then for all $ x \in \,\ ]0,1]$\ :\quad  $p(x)<1$. \label{item:geneticdriftwins_frequency}
  \item If $\alpha_{\mathfrak s}\alpha^*+w>c\beta^*$, then for all $ x \in \,[0,1[$\ :\quad $p(x)=1$. \label{item:selectiontwins_frequency}
 \end{enumerate}
\end{thm} 
$c\beta^*$ is the threshold identified for only \emph{weak} selection as defined with this representation in Equation (26) in \cite{GriffithsLambda} (where $c=1$) and coincides with the threshold in \cite{Foucart13}. Note that, given Condition \ref{MasterCondition}, the assumption of Theorem \ref{thm:XlongtermGriffiths} is the assumption of Theorem 3 in \cite{GriffithsLambda}, which, however, also treats the critical case when the strength of (weak) selection and genetic drift are equal. In particular under the condition that $\beta^*<\infty$, $W$ cannot have an atom at zero (implying $\sigma = 0$ also in \cite{Foucart13} and \cite{GriffithsLambda}). 
Let us stress that \cite{Foucart13} succeeds to also cover the case $\beta^*=\infty$ for weak selection.
 
 \begin{rem}
  In the case studied in \cite{huracain} lizards with long fingers have a selective advantage whenever their habitat is hit by a hurricane, as their enhanced ability to hold on prevents them from being -  literally - blown away. A generation under the influence of a hurricane can be modelled as a two-type Wright-Fisher model with selection in a random environment,
  taking $K_y\sim Q(y)$ to be geometric with parameter $1-y$ and adapt the distribution of $Y^{\mathfrak s}$, to model the frequency and intensity of hurricanes. Theorem \ref{thm:XlongtermGriffiths} now gives conditions for which the individuals with long fingers will go to fixation almost surely, and thus help us to understand how pulses of selection shape the evolution of lizards in particular and all forms of life in general.
 \end{rem}

 \begin{rem}\label{rem:geom_v_Bin}
 The strength of rare selection is determined by both the kernel $\Q=(Q(y)\,:\, y \in [0,1])$ and the measure $\mu$, respectively $\Lambda_{\mathfrak s}$. Two relevant examples are the case of \emph{geometric} $\Q$ 
\begin{align*}
 Q_{\text{geo}}(y) = \sum_{i=1}^\infty(1-y)y^{i-1}\delta_i
\end{align*}
through its connection to \emph{weak} selection (see Remark \ref{rem:geometric_important}) and the case of \emph{binary} $\Q$ 
 \begin{align*}
 Q_{\text{bin}}(y) = (1-y)\delta_1 + y\delta_2
\end{align*}
as the simplest branching mechanism (in the dual branching-coalescing process $Z$). In these cases we obtain
 \begin{align*}
 \alpha^*_{\text{geo}} := \E\left[\frac{1}{1+V\frac{Y^{\mathfrak s}}{1-Y^{\mathfrak s}}}\right] \qquad \text{and} \qquad \alpha^*_{\text{bin}}:= \E\left[\frac{1}{1+VY^{\mathfrak s}}\right]. 
 \end{align*}
 Observe that, given the same choice of $\Lambda_{\mathfrak s}$, i.e. the same distribution of $Y_{\mathfrak s}$, $\alpha^*_{\rm{geo}}\leq\alpha^*_{\rm{bin}}$ and we say that the \emph{effective strength of rare selection} is larger in the binary case, than in the geometric case, since this, of course, implies $\alpha_{\mathfrak s}\alpha^*_{\rm {geo}}\leq\alpha_{\mathfrak s}\alpha^*_{\rm{bin}}$.
  \end{rem}
  
  Theorem \ref{thm:XlongtermGriffiths} is obtained as a simple combination of two observations: Lemma \ref{lem:amazinggrace} explains that the frequency process converges to 0 almost surely, i.e. the weak allele dies out almost surely, if and only if it does so quickly, in the sense that its expectation is integrable over time. (Indeed, we will prove that in this case, the expectation decays exponentially.) Lemma \ref{lem:Xlongterm_pure_Griffiths} then establishes the connection between the threshold and the integrability in time of the expected frequency of the weak allele.

\begin{lemma}\label{lem:amazinggrace}
Let $X=(X(t) : t\geq 0)$ be the two-type FV-process for weak and rare selection given by \eqref{eq:SDEX}. Assume $\beta^*<\infty$. 
 For any $x \in [0,1]$:
 \begin{align*}
  \int_{[0,\infty[}\E_{x}[X(s)]\dd s < \infty \quad \Longleftrightarrow \quad   \P_x(X_{\infty}=0)=1.
 \end{align*}
\end{lemma}
The result of Lemma \ref{lem:amazinggrace} can be explained with the following intuition: If one believes that the sample heterozygosity $\E_x[X(t)(1-X(t))]$ decays exponentially over time for any $x \in [0,1[$ as it does in the classical case, the assumption $\P_x(X_{\infty}=0)=1$ means that the term $1-X(t)$ cannot contribute to this decay, hence $\E_x[X(t)]$ should decay exponentially (and thus be integrable), as we will indeed prove.

\begin{lemma}\label{lem:Xlongterm_pure_Griffiths} 
Let $X=(X(t) : t\geq 0)$ be the two-type FV-process for weak and rare selection given by \eqref{eq:SDEX} with $\sigma=0$. 
 Assume $\beta^*<\infty$. 
 \begin{enumerate}
  \item If $\alpha_{\mathfrak s}\alpha^* + w<c\beta^*$, then there exists an  $ \bar x \in \,\ ]0,1]$\ :\quad  $\int_{[0,\infty[}\E_{\bar x}[X(s)]\dd s = \infty$. \label{item:geneticdriftwins_helpfullemma}
  \item If $\alpha_{\mathfrak s}\alpha^* + w>c\beta^*$, then for all $ x \in \,[0,1[$\ :\quad $\int_{[0,\infty[}\E_x[X(s)]\dd s < \infty$.\label{item:selectiontwins_helpfullemma}
 \end{enumerate}
\end{lemma}
The proof of Lemma \ref{lem:Xlongterm_pure_Griffiths} is quite delicate in that it requires analysis of a suitable Lyapunov function. It follows the strategy used in \cite{GriffithsLambda}. The proof, provided in Section \ref{sec:ltr}, is also of independent interest because it provides a rigorous method to fill a technical gap often left open in the proof of similar statements found in the literature. 
 Lemma \ref{lem:amazinggrace} crucially helps circumventing such issues and guarantees our final result.

The proof of both lemmata can be found in Section \ref{sec:ltr}. 

\noindent Once the truth of Lemmata \ref{lem:Xlongterm_pure_Griffiths} and \ref{lem:amazinggrace} is established, the proof of Theorem \ref{thm:XlongtermGriffiths} is straightforward and thus we give it here.

\begin{proof}[Proof of Theorem \ref{thm:XlongtermGriffiths}]
Assuming $\alpha_{\mathfrak s}\alpha^* + w<c\beta^*$, Lemma \ref{lem:Xlongterm_pure_Griffiths}, \ref{item:geneticdriftwins_helpfullemma}. and Lemma \ref{lem:amazinggrace} imply the existence of an $\bar x \in ]0,1]$ such that $\P_{\bar x}(X_{\infty}=0)<1$. Therefore, we are in case \ref{item:posrec}. of Remark \ref{rem:Xto01}, i.e. the dual $Z$ must be positive recurrent. Since this does not depend in the choice of starting point $x$  we may conclude that $\P_{x}(X_{\infty}=0)<1$ for all $x \in [0,1]$ which proves \ref{item:geneticdriftwins_frequency}.. \ref{item:selectiontwins_frequency}. is immediate from Lemmas \ref{lem:Xlongterm_pure_Griffiths}, \ref{item:selectiontwins_helpfullemma}. and \ref{lem:amazinggrace}.
\end{proof}

\subsubsection{Griffiths representation of the generator $\mathcal A$ of $X=(X(t):t \geq 0)$}

The proof of Lemma \ref{lem:Xlongterm_pure_Griffiths} follows the idea of the proof of Theorem 3 in \cite{GriffithsLambda}, but we extend and formalise the arguments. Central to this argument is a representation of the generator $\mathcal A$ (equation \eqref{eq:generatorX}) of $X$, in the spirit of Theorem 1 in \cite{GriffithsLambda}. This representation can be generalised and we believe this observation to be of independet interest, whence we present it here.

\begin{lemma}[Griffiths representation]\label{lem:Griffithsgenerator_general}
The generator $\mathcal A$ of $X$ can be written as
\begin{align*}
 \mathcal A f(x) 	& = \sigma\frac{1}{2}x(1-x)f''(x)+ \frac{1}{2}cx(1-x)\E\left[f''(x(1-W)+VW)\right] \\
			& \quad -w x(1-x)f'(x) \\
			& \quad -\alpha_{\mathfrak s}x(1-x)\E\Bigg[\E\left[\sum_{l=0}^{K_{Y^{\mathfrak s}}-2}x^l\,\Big|\, Y^{\mathfrak s}\right]\frac{1}{ \E[K_{Y^{\mathfrak s}}-1\,\mid\,Y^{\mathfrak s} ]}\times\\
			& \qquad \qquad \qquad \qquad \qquad \qquad \qquad  f'\left(x-x(1-x)V\E\left[\sum_{l=0}^{K_{Y^{\mathfrak s}}-2}x^l\,\mid\, Y^{\mathfrak s}\right]\right)\Bigg],
 \end{align*}
for every $x \in [0,1]$, and $V$, $W$ and $Y^{\mathfrak s}$ chosen as in \eqref{eq:beta*} and \eqref{eq:alpha*}.
\end{lemma}


\subsection{Ergodicity of the branching coalescing process in random environment}\label{subsec:ergodicity_Z}

As characterised in Remark \ref{rem:Xto01}, the chance of survival of the weak allele has a direct correspondence to the ergodic behaviour of the \emph{branching coalescing process in random environment} $(Z(t) : t\geq0)$. Hence, the following is a corollary or Theorem \ref{thm:XlongtermGriffiths}.

\begin{corollary}\label{cor:branching}
Let $Z=(Z(t) : t\geq0)$ be the coordinated branching coalescing process from Definition \ref{def:BCPRE} with $\sigma=0$.
 
 Assume $\beta^*<\infty$. Then 
 \begin{enumerate}
  \item If $\alpha_{\mathfrak s}\alpha^* + w<c\beta^*$, then $Z$ is positive recurrent.\label{item:geneticdriftwins_genealogy}
  \item If $\alpha_{\mathfrak s}\alpha^*+w>c\beta^*$, then $\P(Z(t)\leq M)\rightarrow 0$, as $t \rightarrow \infty$, for all $M\in \N$ i.e.\ $Z$ is null-recurrent or transient.
 \end{enumerate}
 In the case of \;\ref{item:geneticdriftwins_genealogy}., the generating function of $\nu$, the invariant distribution of the positive recurrent $Z$, is given by $\varphi_{\nu}(x)=\P_x(X_{\infty} = 1)$.
\end{corollary}

\section{Proofs}\label{sec:proofs}
We collect in this section mainly the proofs of the results presented in the previous sections. Each subsection corresponds to a section above.

\subsection{Individual based models for a finite population} \label{subsec:proofs_discrete}
The first observation in this section was the precise transition probabilities of the discrete allele-frequency process.

\begin{proof}[Proof of Proposition \ref{pr:discretefreq}] 
Conditionally on the environment $\bar Y = \bar y$, the proof follows the same steps as in Corollary 2.6 of \cite{ADSelection},  the main idea of which we recall briefly here. The individual $(g,i)$ inherits the weak type 0 if and only if it chooses all its $K_{(g,i)}$ potential parents to be of type 0. Hence, the  probability of this event is given by $\E[x^{K_{(g,i)}}\mid \bar y]$. If generation $g$ does not see an extreme reproductive event (which happens with probability $1-c$) all individuals choose all their potential parents independently and the total number of individuals obtaining all type 0 potential parents is binomial with parameters $N$ and 
$$\E[x^{K_{(g,i)}}\mid \bar y]=\varphi_{y_g}(x).$$
For generations with extreme reproductive events (occurring with probability $c$), first $V \sim \Lambda$ and $I_g\sim \mathcal U_{[N]}$ are sampled. $I_g$ will be the label of the preferred individual.
For each of its choices of  potential parents, each individual then  flips a coin with (the same) random bias $V$. In the case of success, it will choose the preferred individual as a potential parent, i.e. label $I_g$.
The random variable $B_x$ is precisely the indicator of the event that the special parent $I_g$  carries type 0, conditionally on the frequency $x$ of 0-type individuals in the previous generation $g-1$. 
If the $V$-coin returns a failure, the individual chooses this potential parent uniformly in $[N]$ and hence the probability of choosing a weak one is $x$. Therefore, the probability of choosing \emph{one} weak potential parent is given by 
\begin{equation*}
  VB_x + (1-V)x.
\end{equation*}
Consequently,  since an individual is weak only if \emph{all} his potential parents  are  of the weak type  0, the  probability for this even is
\begin{align*}
 \varphi_{y_g}\big(VB_x + (1-V)x\big).
\end{align*}
Since the individuals act independently conditioned on $I_g$ and $V$, the number of weak individuals in the next generation, conditioned on $I_g$ and $V$ is a binomial  with parameters $N$ and $\varphi_{y_g}\big(VB_x + (1-V)x\big)$.
\end{proof}

Next we move to the proof of the sampling dualities between the 0-allele-frequency  process given in Definition \ref{def:discrete_frequency} and the genealogy given in Definition \ref{def:discrete_genealogy}. They are a direct consequence of the graph-construction, as observed in \cite{M99}. We begin with the proof of Proposition \ref{prop:hypergeometricsampling}, as it uses the same idea as the proofs of Lemma \ref{lem:consampleduality} and Proposition \ref{prop:iidduality}, but in a simpler set-up. 

The proofs of Lemma \ref{lem:consampleduality} and Proposition \ref{prop:iidduality} follow the same  idea, except that now we consider sampling  probabilities when the frequency is given in the \emph{previous} generation. Since the random environment acts between generations, it will appear explicitely in our calculations here. Also, since the probability of an individual being of the  0 type is then determined by the types  of its potential parents and hese are sampled with replacement, the resulting sructure is that of sampling \emph{with} replacement and thus obtains the form of a genertating function. 
\begin{proof}[Proof of Proposition \ref{prop:hypergeometricsampling}]
 Note that $\tilde H_N$ defined in \eqref{eq:diskretesamplingdualitywithoutreplacemen_functiont} gives the probability that $n \leq N$ individuals sampled uniformly without replacement in a given generation are all of type 0 if the frequency of type  0 is $x\in [N]_0/N$ \emph{in that same generation}.  In order to prove \eqref{eq:diskretesamplingdualitywithoutreplacement}, we show hat both sides express the same quantity, namely the probability, that $n\leq N$ individuals sampled uniformly without replacement from a generation $s \in \Z$  are all of type 0, given that the frequency of 0 was $x\in [N]_0/N$ in generation $r<s$, which we denote by $p(x,n)$.
 
 First, we calculate this probability using the frequency process: We know that the frequency of the 0-allele in generation $s$ is given by $X^{N,r}(s)$, i.e. that there are $X^{N,r}(s)N$ many weak individuals and thus 
 \begin{align*}
  p(x,n)= \E\left[\prod_{k=1}^n\frac{X^{N,r}(s)N-k}{N-k}\;\bigg\vert\;X^{N,r}(r)=x\right]  = \E[\tilde H_N(X^{N,r}(s),n)\mid X^{N,r}(r)=x].
 \end{align*}
 On the other hand, we can also calculate it using the fact that the individuals in generation $s$ will only be of type 0 if all their ancestors in generation $r$ are  of type 0, i.e. using the ancestral process. The frequency of type 0 individuals in generation $r$ is $x$ and thus
 \begin{align*}
  p(x,n)=\E\left[\prod_{k=1}^{Z^{N,s}(r)}\frac{xN-k}{N-k}\;\bigg\vert\;Z^{N,s}(r)=n\right]  = \E[\tilde H_N(x,   Z^{N,s}(r))\mid Z^{N,s}(r)=n],
 \end{align*}
 which completes the proof.
\end{proof}

\begin{proof}[Proof of Lemma \ref{lem:consampleduality}]
To prove equality \eqref{eq:sdual}, we show that both sides express the same quantity. Denote by $p(\bar Y; x,n)$ the conditional probability, given the random environment $\bar Y$, that $n\leq N$ individuals, sampled uniformly at random without replacement from generation $s\in \Z$, are all of type 0, given that the frequency of 0-type individuals was $x\in [N]_0/N$ in generation $r<s$. 
Without loss of generality, we may assume that the sample in generation $s$ consists of the individuals with the labels $1, \ldots, n$ and therefore
\begin{align*}
 p( \bar Y; x,n):= \P\left((s, 1), \ldots, (s,n) \text{ are all of type }0\mid X^{N,r}(r)=x, \bar Y\right).
\end{align*}
We will calculate this probability in two ways: first, by considering the choices of potential parents of the individuals in our sample $(s, 1), \ldots, (s,n) $ in the previous generation. Second, by tracing their potential ancestors further back in time up until generation $r+1$ and then considering the choices of these potential ancestors. We will, thus, be working both with the frequency process $X^{N,r}$ as well as the genealogical process $Z^{N,s}$. Before we come to these calculations let us make the following two observations similar to the calculations done in the proof of Proposition \ref{pr:discretefreq}.

For any generation $g \in \Z$, conditionally on $Y_g$ the random variables $K_{(g, 1)}, \ldots, K_{(g, n)}$ are independent and identically distributed with distribution $Q(Y_{g})$. Hence, for any $\tilde x \in [0,1]$
\begin{equation}
 \E[ {\tilde x}^{\sum_{i=1}^n K_{(g,i)}} \mid \bar Y ] =\left( \E[\, {\tilde x}^{K_{(g,1)}}\mid \bar Y]\right)^n = \left(\E[\, \varphi_{Y_g}(\tilde x)\mid \bar Y]\right)^n = \; \big(\varphi_{Y_g}(\tilde x)\big)^n. 
 \label{pgfid}
\end{equation}
 Recall the independent  variables $V \sim \Lambda$ and $B_{\tilde x} \sim \mathrm{Ber}_{\bar x}$ introduced for \eqref{eq:conditional_sampling_duality_function} and let $(V_g, g \in \Z)$, be \textit{iid} copies of $V$ and $(B_{\bar{x}, g}, g \in \Z)$ be \textit{iid} copies  of $B_{\bar{x}}$, independent of everything else. Then, similarly to the last observation also 
\begin{align}\label{pgfid2}
 \E\bigg[\big((1-V_g)\tilde x +& V_g(B_{\tilde x, g})\big)^{\sum_{i=1}^nK_{(g,i)}}\mid  \bar Y\bigg]  = \E\bigg[\left(\varphi_{Y_g}\big((1-V_g)\tilde x + V_g(B_{\tilde x, g})\big)\right)^n\mid  \bar Y\bigg].
\end{align}
All equalities between random variables in the following hold $\P$-a.s.
 
We now first calculate $p(\bar Y; x,n)$  with the following idea: all individuals in our sample are of type 0, if and only if all the potential parents they choose in the previous generation are of type 0 and the frequency of type 0 in that generation is given by $X^{N,r}(s-1)$. 
Observe that since $K_{(s, 1)}, \ldots, K_{(s, n)}$  are also independent of $X^{N,r}(r)$, $X^{N,r}(s-1)$, $V_s$ and $\xi(I_s)$, this implies :
 \begin{itemize}
  \item[(i)] In a generation $s$ without skewed offspring distribution, the probability for all $(s,i)$, $i=1,\ldots, N$, to be weak, conditional on the frequency of the weak type in the previous generation $X^{N,r}(s-1)$, using \eqref{pgfid} is  
  \begin{align*}
	\E& \left[\left[X^{N,r}(s-1)\right]^{\sum_{i=1}^nK_{(s,i)}}\mid X^{N,r}(s-1), X^{N,r}(r) =x, \bar Y\right] \\
	  &= \E\bigg[\left(\varphi_{Y^N_s}(X^{N,r}(s-1))\right)^n\,\bigg|\,  X^{N,r}(s-1),  X^{N,r}(r)=x,\bar Y\bigg]
  \end{align*}
  Denote this probability by $p_1(X^{N,r}(s-1), \bar Y; x,n)$.
  \item[(ii)] In a generation with skewed offspring distribution, recall that $I_s$ denotes the favoured individual (special marked label) collectively chosen at time $s$ and $\xi(I_s)$ its type.  By construction, the distribution of $1-\xi(I_s)$ conditional on the frequency of weak in the previous generation is Bernoulli with parameter $X^{N,r}(s-1)$. Hence, in such a generation, the conditional probability for all $(s,i)$, $i = 1, \ldots, n$, to be weak, conditional on the frequency of the weak type in the previous generation is, by \eqref{eq:genrepr}, \\ 
  \begin{align*}
    \E	&\left[\left[(1-V_s)X^{N,r}(s-1) + V_s(1-\xi(I_s))\right]^{K_{(s,i)}}\mid X^{N,r}(s-1), X^{N,r}(r)=x, \bar Y\right]\\
	& = \E\bigg[\Big(\varphi_{Y^N_s}\big((1-V_s)X^{N,r}(s-1) + V_s(B_{X^{N,r}(s-1)})\big)\Big)^n\,\bigg|\,X^{N,r}(s-1),  X^{N,r}(r)=x,\bar Y\bigg].
  \end{align*}
Denote this probability by $p_2(X^{N,r}(s-1), \bar Y, x)$.
 \end{itemize}
Combining this, we obtain
 \begin{align*}
 p(\bar Y; x,n)&  
							 = (1-c)\E\left[p_1(X^{N,r}(s-1), \bar Y; x)\,\bigg|\,  X^{N,r}(r)=x,\bar Y\right]\\
							 & \quad +  c\ \E\left[p_2(X^{N,r}(s-1), \bar Y; x)\,\bigg|\,  X^{N,r}(r)=x,\bar Y\right] \\
%
							&=(1-c) \E\bigg[\left(\varphi_{Y^N_s}(X^{N,r}(s-1))\right)^n\,\bigg|\,  X^{N,r}(r)=x,\bar Y\bigg]	\\
							&\quad + c\ \E\bigg[\Big(\varphi_{Y^N_s}\big((1-V_s)X^{N,r}(s-1) + V_s(B_{X^{N,r}(s-1)})\big)\Big)^n\,\bigg|\,  X^{N,r}(r)=x,\bar Y\bigg]	\\
							&= \E\left[H(X^{N,r}(s-1),n; s, \bar Y) \mid X^{N,r}(r)=x, \bar Y\right],
 \end{align*}
 which is the left-hand-side of the statement of Lemma \ref{lem:consampleduality}.
 
 On the other hand, we can calculate the $p(\bar Y^N;x,n)$  by looking at the number of potential ancestors of the sample in generation $r+1$: these, too, need to be all of the weak type 0 in order for our sample to be of type 0. This in turn is only possible if all \emph{their} potential ancestors in the previous generation $r$ are of the weak type, the frequency of which is given by $x$. Note that we can add a condition on the event $Z^{N,s}(s)=n$ to $p(\bar Y; x,n)$ without altering the probability. These two observations mean 
 \begin{align*}
  p(\bar Y; x,n)&  = \P\left((s, 1), \ldots, (s,n) \text{ are all of type }0\mid Z^{N,s}(s)=n, X^{N,r}(r)=x, \bar Y\right).\\
   &= \P(A^N_{s-r-1}(\{(s,1),\ldots, (s,n)\}) \text{ are all of type 0} \mid X^{N,r}(r)=x, \bar Y)
 \end{align*}
 The number of potential ancestors of $(s,1), \ldots, (s,n)$ in generation $r+1$ is 
 \begin{equation*}
 \vert A^N_{s-r-1}(\{(s,1),\ldots, (s,n)\})\vert = Z^{N,s}(r+1)  
 \end{equation*}
 and the frequency of the weak type 0 in generation $r$ is, by assumption, equal to $x$. Since each of the individuals in $A^N_{s-r-1}(\{(s,1),\ldots, (s,n)\})$ chooses its potential parents independently, using the same observations as in the first step the sought probability can be expressed as
\begin{align*}
& p(\bar Y; x) =  (1-c)\E\bigg[x^{\sum_{i=1}^{Z^{N,s}(r+1)} K_{(r+1,i)}}\,\mid Z^{N,s}(s) = n, X^{N,r}(r)=x,  \bar Y\bigg]	\\
		& \qquad + c\ \E\bigg[\big[(1-V_{r+1})x + V_{r+1}(1-\xi(I_{r+1}))\big]^{\sum_{i=1}^{Z^{N,s}(r+1)}K_{(s,i)}} \mid\,  Z^{N,s}(s) = n, X^{N,r}(r)=x, \bar Y\bigg]	\\
		& = (1-c)\E\bigg[\varphi_{Y_g}(x)^{Z^{N,s}(r+1)}\,\mid \,Z^{N,s}(s) = n, \bar Y\bigg]\\
		& \qquad + c\ \E\left[\varphi_{Y_g}\left((1-V_{r+1})x + V_{r+1}(B_x)\right)^{Z^{N,s}(r+1)}\,\mid Z^{N,s}(s) = n, \bar Y\right]\\
									& \ \ = \E\left[H(x, Z^{N,s}(r+1); r+1, \bar Y)\mid Z^{N,s}(s) = n,  \bar Y\right]
\end{align*}
where, again, we have profited from the fact that, conditionally on $\bar Y$,  the vector $(K_{(r+1,i)}:i\in[N])$ has \emph{iid} coordinates and is independent of $Z^{N,s}(r+1)$ and $Z^{N,s}(s)$, along with the observations \eqref{pgfid} and \eqref{pgfid2}.
\end{proof}
We now turn to the proof of the unconditioned duality. 
\begin{proof}[Proof of Proposition \ref{prop:iidduality}]
 The idea of the proof is the same as for Lemma \ref{lem:consampleduality}, but we have to take the \emph{iid} environment carefully into account. Let $g\geq 0$. First note that, by identity in distribution, 
 \begin{align}
  H_{\mu}(x,n)& =  \mathbb{E}[H(x,n;1,\bar Y )]	\notag\\
		& =(1-c)\E[\varphi_{Y_{1}}(x)^n]   + c\  \E\left[\varphi_{Y_{1}}\big((1-V)x + VB_x\big)^n\right]\label{eq:dualityfctn_alt_1}\\
		& =  \mathbb{E}[H(x,n;g+1,\bar Y )]	\notag\\
		& =(1-c)\E[\varphi_{Y_{g+1}}(x)^n]   + c\  \E\left[\varphi_{Y_{g+1}}\big((1-V)x + VB_x\big)^n\right]\label{eq:dualityfctn_alt_g+1}
 \end{align}
and with the new notation
 \begin{align}\label{eq:shift_rhs_iid_duality}
  \E^{n}\left[H_{\mu }(x,Z^{N,0}\left(-g)\right)\right]	& = \E^{n}\left[H_{\mu }(x,Z^{N,g+1}\left(1)\right)\right].
 \end{align}
We again seek to determine the same probability, except that this time we do not condition on the random environment, i.e. we search 
\begin{align*}
 p(x,n)& := \P_x\left((g+1, 1), \ldots, (g+1,n) \text{ are all of type }0\right)
\end{align*}

\noindent As before, this probability can, on one hand, be expressed using $X^{N,0}(g)$ as
 \begin{align*}
  p(x,n)&=(1-c  )\E_x[(X^{N,0}(g))^{\sum_{i=1}^n K_{(g+1,i)}}]	\\
		      &\qquad  + c\  \E_x\bigg[\left((1-V_{g+1}){X^{N,0}(g)}+ V_{g+1} (1-\xi(I_{g+1}))\right)^{\sum_{i=1}^n K_{(g+1,i)}}\bigg]	\\
	& = (1-c  )\E_x\bigg[{\varphi_{Y_{g+1}}(X^{N,0}(g))^n}\bigg]	\\
	& \qquad + c\  \E_x\bigg[\varphi_{Y_{g+1}}\Big((1-V_{g+1})X^{N,0}(g)+ V_{g+1} B_{X^{N,0}(g)}\Big)^n\bigg]	\\
	&
	=\E_x\left[H_{\mu }(X^{N,0}(g),n)\right].
 \end{align*}
For the second equality we used that 
$(K_{(g+1,1)}, \ldots, K_{(g+1,n)}, Y_{g+1})$ and $X^{N,0}(g)$ are independent, and that conditional on $Y_{g+1}$ the $K_{(g+1,1)}, \ldots, K_{(g+1,n)}$ are \emph{iid}, which in particular implies for any $ \tilde x \in [0,1]$:
\begin{align}
 \E[\, {\tilde x}^{\sum_{i=1}^n K_{(g+1,i)}}] 
 = \E[\E[\tilde x^{K_{(g+1,1)}}\mid \bar Y ]^n ] 
 = \E[\varphi_{Y_{g+1}}(\tilde x)^n]
\label{upgfid}\end{align}
and likewise 
\begin{align}\label{upgfid2}
 \E[\, ((1-V_g)&\tilde x + V_g B_{\tilde x,g})^{\sum_{i=1}^n K_{(g+1,i)}}] \\
 & = \E\Big[\E\big[((1-V_g)\tilde x + V_g B_{\tilde x,g})^{ K_{(g+1,1)}}\mid V_g, B_{\tilde x,g}, \bar Y \big]^n \Big]\notag \\
 & = \E[\varphi_{Y_{g+1}}((1-V_g)\tilde x + V_g B_{\tilde x,g})^n ]\notag
\end{align}
Note that the random variables $K_{(g+1,1)}, \ldots, K_{(g+1,n)}$ are normally only \emph{conditionally} independent, given a realisation of $Y _{g+1}$, which is why we cannot, in the unconditioned case, write the duality function as a product of expectations.

\noindent On the other hand, we can calculate the same probability $p(x,n)$ considering the ancestry with $Z^{N,g+1}(1)$ through 
\begin{align*}
 p(x) & = (1-c  )\E[x^{\sum_{i=1}^{Z^{N,g+1}(1)} K_{(1,i)}}]  + c\  \E_x\bigg[\left((1-V_{1})x+ V_{1} (1-\xi(I_{1}))\right)^{\sum_{i=1}^{Z^{N,g+1}(1)} K_{(1,i)}}\bigg]	\\
	& = (1-c  )\E_x\bigg[\varphi_{Y_{1}}(x)^{Z^{N,g+1}(1)}\bigg]	 + c  \E_x\bigg[\varphi_{Y_{1}}\Big((1-V_{1})x+ V_{1} B_{x}\Big)^{Z^{N,g+1}(1)}\bigg]	\\
	& = \E[H_{\mu }(x, Z^{N,g+1}(1))]
\end{align*}
where we again profited from the independence of $Z^{N,g+1}(1)$ and $(K_{(1,1)}, \ldots, K_{(1,N)}, \bar Y)$ and the observations \eqref{upgfid} and \eqref{upgfid2} with $g+1$ replaced by $1$.
Hence, we have proven
\begin{align*}
\E_{x}\left[H_{\mu }\left(X^{N,0}(g),n\right)\right] = \E^{n}\left[H_{\mu }(x,Z^{N,g+1}\left(1)\right)\right]
\end{align*}
which together with \eqref{eq:shift_rhs_iid_duality} completes the proof.
\end{proof}

\subsection{Scaling limits}

We begin with the proof of existence and uniqueness of a solution to the SDE.
\begin{proof}[Proof of Lemma \ref{lem:SDEXsolution}]
For the case without the third summand in the generator \eqref{eq:SDEX} the statement of Lemma \ref{lem:SDEXsolution} was already proved in \cite[Cor. 3.3]{caballero}. For the general case, the proof is an application of Theorem 5.1 in \cite{LiPu}, of which we only need to verify conditions  $3a)$, $3b)$ and $5a)$. Using the notation of \cite{LiPu}, in our case the relevant functions are $\sigma(x)= \sqrt{\sigma x(1-x)}\1_{[0,1]}(x)$, $b(x)=wx(1-x)$, $g_0(x,(z,u))=z(\1_{[0,x]}(u)-x)\1_{[0,1]}(x)$ for $U_0=[0,1]^2$ and $g_1(x,y)=\E[x^{K_y}-x]\1_{[0,1]}(x)$ for $U_1=[0,1]$. 
All conditions are easy to check for $g_0$ and were shown for $b$ and $\sigma$ in Equation (3.a) and (26) of \cite{ADSelection}, respectively, so we only need to concern ourselves with the selection component $g_1$.

Note that $\vert\frac{\partial g_1}{\partial x}(x,y)\vert=\vert\E[K_yx^{K_y-1}-1]\vert\leq\max\{\P(K_y\geq 2),\E[K_y-1]\}\leq\E[K_y-1]$, for all $x \in [0,1]$. Then, by the mean value theorem, for any $x_1,x_2\in[0,1]$ 
\begin{align*}
\int_{[0,1]}|g_1(x_1,y)-g_1(x_2,y)|\mu(\dd y) \leq |x_1-x_2|\int_{[0,1]} \E[K_y-1]\mu(\dd y)
\end{align*} 
Since the integral is finite by Condition \ref{MasterCondition}, this gives us condition $3a)$ of \cite{LiPu}.
The observation above also shows that $\frac{dg_1}{dx}$ is $(\mu(\dd y)\otimes \dd x)$-integrable on $[0,1]\times [0,1]$, because
\begin{align*}
\int_{[0,1]\times [0,1]}\left\vert \frac{\partial g_1}{\partial x}(x,y)\right\vert \mu(\dd y)\otimes \dd x\leq \int_{[0,1]} \E[K_y-1] \mu(dy)<\infty,
\end{align*}
whence we can use Fubini's Theorem and, using the fact that  $g_1(x,y)^2\leq \vert g_1(x,y)\vert$, estimate 
\begin{align*}
\int_{[0,1]} \left( g_1(x,y)\right)^2\mu(\dd y) 	& \leq \int_{[0,1]} \int_{[0,x]} \left\vert \frac{d g_1}{ds}(s,y)\right\vert\ \dd s \ \mu(\dd y)\\
					& 
					\leq x  \int_{[0,1]} \E[K_y-1] \mu(\dd y)
 \end{align*}
for $x \in [0,1]$, which yields the estimates for condition $5a)$ of \cite{LiPu}.
\end{proof}

We now move to the proof of convergence of the discrete frequency process defined in Definition \ref{def:discrete_frequency} to the scaling limit \eqref{eq:SDEX} for $w=0$, i.e with only rare selection, started in $X_0=x$
\begin{proof}[Proof of Theorem \ref{lem:converges}]
We will prove strong convergence of the sequence of generators corresponding to $X^N,N\in \N$ to the generator of $X$ on $C^2([0,1])$. 
Our claim then follows from Theorem 19.28 in \cite{Kallenberg} (Theorem 17.28 in the other editions)which however requires the limit process $X$ to be also Feller. The Feller property is established in Section \ref{subsec:diffusionconvergence}, Lemma \ref{lem:fellerX} and will be assumed true in this proof. \\
Denote with $\mathcal{A}^N$ the discrete generator of $(X^N(\lfloor \rho_N^{-1}t \rfloor): t \geq 0)$. We calculate $\mathcal{A}^N$  conditioning on the four possible cases reflecting occurrence vs non-occurrence of selection or multiple mergers. To this end, starting without loss of generality from generation $g_0=0$, let $B_{\mathfrak c}$ and $B_{\mathfrak s}$ be the random variables indicating whether selection, resp.\ multiple mergers occur in the first time-step. By construction of the reproduction mechanism, they are independent Bernoulli random variables with success parameter $c_N$ and $\rho_N\gamma_N([0,1])$ respectively. For $C_2$ functions $f:[0,1]\mapsto \R$ the generator then acts according to the rule
\begin{align}
\label{eq:convergenceproof_prelimit_generator}
\mathcal{A}^Nf(x) 	& := \rho_N^{-1}\E_x\left[f(X^N_1) -f(x)\right] \\
			& = \rho_N^{-1}\bigg\{(1-\rho_N\gamma_N([0,1]))(1-c_N) \left(\frac{1}{2N}x(1-x)f''(x) + O\left(\frac{1}{N^2}\right)\right)  \notag\\
			& \qquad \qquad + (1-\rho_N\gamma_N([0,1]))c_N\E_x\left[f(X^N_1) -f(x)\big|\, B_{\mathfrak c}=1,B_{\mathfrak s}=0\right] \notag\\
			& \qquad \qquad + \rho_N\gamma_N([0,1])(1-c_N)\E_x\left[f(X^N_1) -f(x)\big|\, B_{\mathfrak c}=0,B_{\mathfrak s}=1\right] \notag\\
			& \qquad \qquad + \rho_N\gamma_N([0,1])c_N\E_x\left[f(X^N_1) -f(x)\big|\, B_{\mathfrak c}=1,B_{\mathfrak s}=1\right]\bigg\}.\notag
\end{align}

The first summand corresponds to the standard neutral Wright-Fisher model without multiple mergers, hence assumption \eqref{eq:assumption_for_WF} implies the convergence to the first component of the generator given in \eqref{eq:generatorX}. 

Since the last expectation is bounded, assumption \eqref{eq:unnecessary_assumption} implies that the last summand vanishes as $N\to\infty$, so we only need to calculate carefully the remaining two summands:

For the first of these two -- the case of only skewed offspring distribution -- observe that, since $B_{\mathfrak c}=1,B_{\mathfrak s}=0$, by virtue of the reproduction mechanism \eqref{eq:genrepr}, for a given initial frequency $X^N_0=x$, strength of correlation (multiple mergers) $C^N_1=u\in [0,1]$ and type $B=b\in\{0,1\}$ of the favoured parent, $NX^N_1$ is simply a Binomial random variable with parameters $N$ and $(1-u)x+u(1-b)$. Also note that given $X^N_0=x$, $B$ is a Bernoulli random variable with success parameter $1-x$.

\noindent Apply Taylor's expansion to $f$ around  $(1-C^N_1)x+C^N_1(1-B)$. We obtain, using the Lagrange remainder term, for an $\eta$ depending on $X^N_1, C^N_1, B$ and $ x$,
\begin{align*}
 \E_x&\big[f(X^N_1) -f(x)\big|\, B_{\mathfrak c}=1,B_{\mathfrak s}=0\big]	\\
			& = \E_x\Big[f\big((1-C^N_1)x + C^N_1(1-B)\big) - f(x)\Big| B_{\mathfrak c}=1,B_{\mathfrak s}=0\Big] \\	
			& \qquad \qquad \qquad + \E_x\Big[\frac{1}{2}f''(\eta)\big(X^N_1 - (1-C^N_1)x - C^N_1(1-B)\big)^2\Big| B_{\mathfrak c}=1,B_{\mathfrak s}=0\Big]
\end{align*}
Using the independence of $U^N_1$, $B$, $B_{\mathfrak{c}}$, and $B_{\mathfrak s}$, we can rewrite the first term:
\begin{align*}
 \rho_N^{-1}&c_N\E_x\Big[f\big((1-C^N_1)x + C^N_1(1-B)\big) - f(x)\Big| B_{\mathfrak c}=1,B_{\mathfrak s}=0\Big] \\	
			& =  \frac{\rho_N^{-1}c_N}{\Lambda^{\alpha}_N([0,1])}\int_{[0,1]}\big(xf((1-u)x+u)+(1-x)f((1-u)x) -f(x)\big)\frac{1}{u^2}\1_{[N^{-\alpha},1]}(u)\Lambda_{\mathfrak c}(\dd u)\\
			& \xrightarrow{N\rightarrow \infty} c \int_{[0,1]}\big(xf((1-u)x+u)+(1-x)f((1-u)x) -f(x)\big)\frac{1}{u^2}\Lambda_{\mathfrak c}(\dd u).
\end{align*}
The convergence in the last step follows from \eqref{eq:assumptions_on_skewedoffspring} and is uniform in $x\in[0,1]$.
\\ Using the conditional variance of $NX^N_1$ given $C^N_1$ and $B$ allows us to estimate 
\begin{align*}
 \bigg\vert&\E_x\Big[\;\frac{1}{2}f''(\eta)\big(X^N_1  - (1-C^N_1)x - C^N_1(1-B)\big)^2\Big| B_{\mathfrak c}=1,B_{\mathfrak s}=0\Big] \bigg\vert\\
	& \leq \frac{1}{2}\sup_{z \in [0,1]}\vert f''(z)\vert \bigg\vert\E_x\left[\;\frac{N\big(1  - (1-C^N_1)x - C^N_1(1-B)\big)\big( (1-C^N_1)x + C^N_1(1-B)\big)}{N^2}\right]\bigg\vert\\
	& = \frac{1}{2}\sup_{z \in [0,1]}\vert f''(z)\vert\frac{1}{N\Lambda^{\alpha}_N([0,1])}\\
	& \qquad \qquad \times\big\vert\int_{[N^{-\alpha},1]}{\E_x\left[\;\big(1  - (1-u)x - u(1-B)\big)\big( (1-u)x + u(1-B)\big)\right]}\frac{1}{u^2}\Lambda_{\mathfrak c}(\dd u) \big\vert\\
	&\leq \frac{1}{2}\sup_{z \in [0,1]}\vert f''(z)\vert N^{-1}.
\end{align*}
In the second step we simply used the independence of $C^N_1$ and $B$ as well as the definition of $\bar \Lambda^{\alpha}_N$ and in the last that the expectation can be bounded by 1. 
Since $(C_N)_{N\in\N}$ converges to 0, assumption \eqref{eq:assumption_for_WF} implies that this term vanishes. Note that also this bound is uniform in $x\in [0,1]$.

In the same spirit as for the skewed offspring distribution, given that we only observe rare selection, i.e. $B_{\mathfrak c}=0,B_{\mathfrak s}=1$, the distribution of $NX^N_1$ conditioned on a strength of selection $Y^N_1=y\in[0,1]$ is again binomial with parameters $N$ and $\E[x^{K_y}]$. Applying Taylor's formula in this point one obtains for an $\eta$ depending on $X^N_1$, $Y^N_1$ and $x$:
\begin{align*}
 \E_x&\left[f(X^N_1) -f(x)\big|\, B_{\mathfrak c}=0,B_{\mathfrak s}=1\right] \\
      & = \E_x\left[f(\E[x^{K_{Y^N_1}}\mid Y^N_1]) -f(x)\big|\, B_{\mathfrak c}=0,B_{\mathfrak s}=1\right] \\
	& \qquad + \E_x\left[\frac{1}{2} f''(\eta)\big(X^N_1 - \E[x^{K_{Y^N_1}}\mid Y^N_1]\big)^2 \big|\, B_{\mathfrak c}=0,B_{\mathfrak s}=1\right].
\end{align*}
Again, using the conditional variance of $NX^N_1$ given $Y^N_1$, we estimate
\begin{align*}
 \E_x&\left[\frac{1}{2} f''(\eta)\big(X^N_1 - \E[x^{K_{Y^N_1}}\mid Y^N_1]\big)^2 \big|\, B_{\mathfrak c}=0,B_{\mathfrak s}=1\right]\\
  & \leq \frac{1}{2}\max_{z \in [0,1]}f''(z)\frac{1}{N}\E_x\left[\E[x^{K_{Y^N_1}}\mid Y^N_1]\big(1 - \E[x^{K_{Y^N_1}}\mid Y^N_1]\big) \big|\, B_{\mathfrak c}=0,B_{\mathfrak s}=1\right]\\
   & \leq \frac{1}{2}\max_{z \in [0,1]}f''(z)\frac{1}{N}
\end{align*}
since the last expectation is bounded by 1. Since $\rho_N\gamma_N([0,1])$ vanishes,  assumption \eqref{eq:assumption_for_WF} proves that this term also vanishes in the limit.\\
To calculate the first expectation on the other hand, note that conditioned on $B_{\mathfrak s}=1$, $Y^N_1$ is distributed according to $\bar \gamma_N$. Hence,
\begin{align*}
  \gamma_N([0,1])\E_x&\left[f(\E[x^{K_{Y^N_1}}\mid Y^N_1]) -f(x)\big|\, B_{\mathfrak c}=0,B_{\mathfrak s}=1\right] \\
      & = \gamma_N([0,1])\int_{[0,1]} (f(\E[x^{K_y}])-f(x)) \bar\gamma_N(\dd y) \\
      & = \int_{[0,1]} (f(\E[x^{K_y}])-f(x)) \1_{I_N}(y)\mu(\dd y) \xrightarrow{N\rightarrow \infty} \int_{[0,1]} (f(\E[x^{K_y}])-f(x))\mu(\dd y).
\end{align*}
To see that this convergence is uniform in $x \in [0,1]$, recall Condition \ref{MasterCondition} and observe that using the mean-value-theorem 
\begin{align*}
 \vert \int_{[0,1]} &(f(\E[x^{K_y}])-f(x)) \1_{I^c_N}(y)\mu(\dd y) \vert \\
    & = \vert \int_{[0,1]} (f(\E[x^{K_y}])-f(x)) \1_{I^c_N}(y)\frac{1}{\E[K_y -1]}\Lambda_{\mathfrak s}(\dd y)\vert \\
    & \leq \max_{z\in [0,1]}\vert f'(z)\vert  \int_{[0,1]}\vert\E[x^{K_y}]- x\vert \1_{I^c_N}(y)\frac{1}{\E[K_y -1]}\Lambda_{\mathfrak s}(\dd y).
\end{align*}
Since $K_y$ only takes integer values, as observed in the proof of Lemma \ref{lem:SDEXsolution}, $\vert\E[x^{K_y}]- x\vert\leq  \E[K_y -1]$ and the claim follows since $\Lambda_{\mathfrak s}$ is a distribution.

Therefore we have proven that the assumptions of the theorem  ensure $\mathcal A^N f \rightarrow \mathcal Af$ uniformly on the compact set $[0,1]$. Hence, using Theorem 19.28 in \cite{Kallenberg} (Theorem 17.28 in the other editions) we conclude the  desired weak convergence.
\end{proof}

We now move to the proof of convergence of the discrete frequency process defined in Definition \ref{def:discrete_frequency} to the scaling limit \eqref{eq:SDEX} with weak and rare selection.
\begin{proof}[Proof of Theorem \ref{thm:convergence_rare_and_weak}]
 The proof is exactly analogous to that of Theorem \ref{lem:converges}. As done there, one can condition, with the help of suitable Bernoulli random variables, on whether or not now \emph{rare} selection or multiple mergers occur. For the first summand in \eqref{eq:convergenceproof_prelimit_generator}, however, one now calculates
\begin{align*}
 \E_x&\big[f(X^N_1) -f(x)\big|\, B_{\mathfrak c}=0,B_{\mathfrak s}=0\big] \\
      & = -w_N\frac{x(1-x)}{1-xw_N}f'(x) + \frac{1}{2N}x(1-x)f''(x)\frac{1-w_N}{(1-xw_N)^2} + O(w^2_N)+  O(N^{-2})
 \end{align*}
 (doing a Taylor expansion around $x$).
 The three other summands coincide with those in the proof of Theorem \ref{lem:converges} and the claim follows.
\end{proof}

After having proved the results for the forward frequency process, we now turn to the proofs on the branching coalescing process. We begin by showing that it is Feller and conservative under our main Condition \ref{MasterCondition}.

\begin{proof}[Proof of Lemma \ref{lem:fellerZ}]
To prove that $Z$ is Feller we have to show that its semigroup given by $P_tf(n)=\mathbb{E}_n[f(Z_t)]$ sends $C$ to $C$. (Recall that we consider $\N\cup\{\infty\}$ with the metric of the harmonic numbers and are thus on a compact, where $C$, $C_b$ and $C_0$ coincide.) Note that, for such a continuous function $f$, only the continuity of $p_tf$ in $\infty$ needs to be checked. 

In order to do that we will construct a coupling of $Z$ started at any possible number of initial particles/lineages. To this end, for any finite number of initial particles $l\in \N$ we define the process $(B^{(l)}(t): t\geq 0)$ with $B^{(l)}(t):=(B^{(1,l)}_t, B^{(2,l)}_t,...,B^{(l,l)}_t )$  with starting condition $(B^{(1,l)}_0, B^{(2,l)}_0,...,B^{(l,l)}_0 )=(1,1,..,1)$. Each of the $B^{(i,l)}$ describes the individuals resulting from one initial particle in $Z$, i.e. it traces the offspring numbers of that one individual resulting from (spontaneous or coordinated) branching events and (pairwise or coordinated) coalescence events. While the former can be assigned  to each individual, the coalescence events require interaction between the $B^{(i,l)}, i,j =1,  \ldots, l$, as is reflected in the following rates. $(B^{(l)}(t): t\geq 0)$ branches from $(n_1,...,n_l)\in \N\times\N_0\times\cdots\times\N_0$ to $(n_1+k_1,...,n_l+k_l)$ given by
\begin{equation*}
\int_{[0,1]} \prod_{i=1}^l\Prob{\sum_{j=1}^{n_i}K_{y,j}=n_i+k_i} \mu (\dd y) + \sum_{i=1}^l wn_i\delta_{e_i,(k_1,...,k_l)}
\end{equation*}
which corresponds to each individual independently following its own (coordinated and spontaneous) branching mechanism. Coalescence events allow the process $(B^{(l)}(t): t\geq 0)$ to go from a state  $(n_1,...,n_l)$ to $(n_1-k_1,...,n_l-k_l)+e(k_1, \ldots, k_l)$, where $e(k_1, \ldots, k_l)$ is the unit vector in $\min\{j\mid  k_j >0\}$. A coalescence can thus reduce the number of lineages in each $B^{(i,l)}$ and will assign the surviving lineage to the offspring group with the lowest label. The rate for this event is
\begin{align*}
c\int_{[0,1]}&\prod_{i=1}^l\binom{n_i}{k_i}  y^{k_i}(1-y)^{n_i-k_i} \frac{\Lambda_{\mathfrak c}(\dd y)}{y^2} \\
	      & \qquad \qquad \qquad + \sum_{i=1}^l\sigma \binom{n_i}{2}\delta_{2e_i,(k_1,...,k_l)}+ \sigma \sum_{i=1}^l \sum_{j=1}^{i-1}n_in_j\delta_{e_i+e_j,(k_1,...,k_l)}.
\end{align*}
The first term accounts for the coordinated coalescences across all offspring populations; the second accounts for pairwise coalescences within the same offspring group and the last term for pairwise coalescences across offspring populations. 

\noindent Note that the processes are constructed such that a $B^{(i,l)}$ will always only notice the action of or interaction with a $B^{(j,l)}$ of a lower label $j \leq i$. Thus the laws of the processes $(B^{(l)}(t): t\geq 0)$ form  consistent family in $l \in \N$ and by Kolmogorov's extension theorem can be extended to a process  $(B^{(\infty)}(t): t \geq 0)$ with $B^{(\infty)}(t):=(B^{(1)}(t), B^{(2)}(t),\ldots)$. 

\noindent This gives the desired coupling of $Z$ started with different initial conditions in the following way: If we denote by $Z^{(n)}$ the process $Z$ starting in $n$, then, in distribution,
\begin{align*}
  Z^{(n)}_t\overset{\mathcal L}{=}\sum_{i=1}^{n}B_t^{(i)} \text{\quad and\quad}Z^{(\infty)}_t\overset{\mathcal L}{=}\sum_{i=1}^{\infty }B_t^{(i)}. 
\end{align*}
Note that, by construction $\lim_{n \rightarrow \infty}\sum_{i=1}^{n}B_t^{(i)}=\sum_{i=1}^{\infty }B_t^{(i)}$,   almost surely. Hence, for $f$ continuous on $N\cup\{\infty\}$, we get
\begin{align*}
\vert P_tf(n) - P_tf(\infty)\vert & = \vert \mathbb{E}^n[f(Z(t))]-\mathbb{E}^\infty[f(Z(t))] \vert 	 = \vert \mathbb{E}[f(Z^{(n)}(t))]-\mathbb{E}[f(Z^{(\infty)}(t))] \vert\\
								& = \bigg\vert \mathbb{E}\left[f\left(\sum_{i=1}^{n}B^{(i)}(t)\right)-f\left(\sum_{i=1}^{\infty}B^{(i)}(t)\right)\right]\bigg\vert \xrightarrow{n \rightarrow \infty} 0
\end{align*}
by bounded convergence, which implies that $P_tf \in C$ and thus that $Z$ is indeed Feller.

We now turn to the proof of conservativeness of $Z$. We will show that the branching mechanism itself is non-explosive, even without the ``help'' of the fast coalescence rate. To this end, let $W=(W(t): t \geq 0)$ be the process with the same branching mechanism as $Z=(Z(t): t \geq 0)$ but without coalescence. It is clear that then stochastically $W(t)\geq Z(t)$ for any $t  \geq 0$.
For any $m \in \N$ define $W^m=(W^m(t): t \geq 0)$ with $W^m(t):= \min\{W(t),m\}$, the branching process truncated at $m$. Since $W$ is non-decreasing, $W^m$ is also Markovian with generator 
\begin{align*}
 Gf(k)	& = \sum_{l = k}^{m-1} \left\{\int_{[0,1]}\P\left(\sum_{j=1}^k K_{y,j}=l\right)\mu(\dd y) + w k \delta_{l-k,1}\right\}\left[f(l) - f(k)\right]\\
	&  \qquad \qquad + \int_{[0,1]}\P\left(\sum_{j=1}^k K_{y,j}\geq m\right)\mu(\dd y)\left[f(m) - f(k)\right]
\end{align*}
where the $K_{y,1}, \ldots, K_{y, k}$ are \emph{iid} with common distribution $Q(y)$. For the identity on $\{1,\ldots, m\}$, this simplifies to
\begin{align*}
 G{\rm Id} (k)	& = \sum_{l = k}^{m-1} \left\{\int_{[0,1]}\P\left(\sum_{j=1}^k K_{y,j}=l\right)\mu(\dd y) + w k \delta_{l-k,1}\right\}\left[l - k\right]\\
	&  \qquad \qquad + \int_{[0,1]}\P\left(\sum_{j=1}^k K_{y,j}\geq m\right)\mu(\dd y)\left[m - k\right]\\
	& =  wk+ \int_{[0,1]} \left\{ \sum_{l = k}^{m-1}\P\left(\sum_{j=1}^k K_{y,j}=l\right)\left[l - k\right]+\P\left(\sum_{j=1}^k K_{y,j}\geq m\right)\left[m - k\right]\right\}\mu(\dd y)\\
	& = wk+ \int_{[0,1]} \E\left[\left(\sum_{j=1}^k K_{y,j}\right)\wedge m - k\right]\mu(\dd y)\\
	& \leq wk+ \int_{[0,1]} \E\left[\sum_{j=1}^k K_{y,j} - k\right]\mu(\dd y)\\
	& = \left(w+ \int_{[0,1]} \E\left[K_{y,1} - 1\right]\mu(\dd y)\right)k
\end{align*}
which is finite by Condition \ref{MasterCondition}.
Using Dynkin's formula we observe that
\begin{align*}
\E^n	& [W^m(t)] - n 	 = \int_0^t \E^n\left[ G{\rm Id}(W^m(s))\right]\dd s\\
	& \leq \int_0^t \left(w+ \int_{[0,1]} \E\left[K_{y,1} - 1\right]\mu(\dd y)\right)\E^n\left[W^m(s)\right]\dd s\\
\end{align*}which with  Gr\"{o}nwall's inequality implies
\begin{align*}
 \E^n[W^m(t)] \leq n \exp\left(t \left\{w+\int_{[0,1]} \E[K_{y,1} - 1]\mu(\dd y)\right\}\right) <\infty.
\end{align*}
Since the bound is uniform in $m\in \N$ and $W^m(t) \uparrow W(t)$ a.s.\ as $m\rightarrow \infty$ for any $t \geq 0$, we have completed the proof.  
\end{proof}

Given these properties of the branching coalescing process $Z=(Z(t):t\geq 0)$, we now prove the moment duality between $Z$ and $X=(X(t):t\geq 0)$.
\begin{proof}[Proof of Lemma \ref{lem:duality}]
The case $n=\infty$ holds by the choice of the entrance law. By the Markov property, for every other $n$,  the proof can be done by calculating the generator applied to the function(s) $f_x(n):=f^n(x):=x^n$. Recall the generator $\mathcal A$ of $X$ from \eqref{eq:generatorX} and let $\mathcal B$ be the generator of $Z$.

Since the moment duality relations between the Wright-Fisher diffusion and the Kingman coalescent, between\ the $\Lambda$-jump diffusion and the $\Lambda$-coalescents and between binary branching and the logistic ODE are well-known \cite{BLG03}, the additive structure of the generator allows us to only consider the component of the generators responsible for \emph{rare} selection, which we denote by $\mathcal A_{\mathfrak s}$ and $\mathcal B_{\mathfrak s}$ respectively. 
 For any $y \in [0,1]$ and $n \in \N$ let $K_{y,j}, j =1,\ldots,n $ be \emph{iid}\ with distribution $Q(y)$. Then, for any $x \in [0,1]$ and $n \in \N_0$
 \begin{align*}
 \mathcal A_{\mathfrak s} f^n(x) 	& =  \int_{[0,1]} \left\{(\E[x^{K_y}])^n-x^n\right\}\ \mu(dy) \\
			& =  \int_{[0,1]} \left\{\E[x^{\sum_{j=1}^nK_{y,j}}]-x^n\right\}\mu(dy)\\
			& =  \int_{[0,1]} \left[\sum_{k=n}^{\infty}x^k\P\left(\sum_{j=1}^nK_{y,j}=k\right)-x^n\right]\mu(dy)\\
			& =  \int_{[0,1]} \left[\sum_{k=n}^{\infty}\P\left(\sum_{j=1}^nK_{y,j}=k\right)\left(x^k-x^n\right)\right] \mu(dy) = \mathcal B_{\mathfrak s}f_x(n)
 \end{align*}
 As mentioned above, by linearity of the operators this implies $\mathcal A f^n(x)  =  \mathcal B f_x(n)$ for any $x \in [0,1]$ and $n \in \N_0$. Hence, since $X$ is a Markovian solution to the martingale problem associated to $\mathcal A$, and Lemma \ref{lem:fellerZ} implies that $Z$ is Feller, the moment duality between $X$ and $Z$ follows immediately (see e.g. Proposition 6.1 in \cite{GCPP}).
\end{proof}

\begin{proof}[Proof of Lemma \ref{lem:fellerX}]
 Let $f(x):=x^n$ for fixed $n \in \N_0$. Let $(\tilde P_t)_{t \geq 0}$ be the semigroup generated by $\mathcal A$ given in \eqref{eq:generatorX}. Then, using the moment duality from Lemma \ref{lem:duality}
 \begin{align*}
  \tilde P_tf(x) = \E_x[f(X(t))]=\E_x[X(t)^n]=\E^n[x^{Z(t)}].
 \end{align*}
 Note that, since $Z$ is conservative, $x \mapsto x^{Z(t)(\omega)}$ is continuous for all $x \in [0,1]$ for $\P$-almost all $\omega \in \Omega$. Hence, by bounded convergence and using duality again
 \begin{align*}
  \lim_{x\rightarrow \bar x} \tilde P_tf(x) = \lim_{x\rightarrow \bar x} \E^n[x^{Z(t)}] = \E^n[{\bar x}^{Z(t)}] = \E_{\bar x}[X(t)^n] = \tilde P_tf(\bar x)
 \end{align*}
 for all $\bar x \in [0,1]$. Therefore, for every $t\geq 0$, $\tilde P_t$ maps monomials to continuous functions. By the linearity of the expectation this also holds for polynomials and by the Stone-Weierstrass-Theorem we have proven the claim.
\end{proof}

Finally, we use the duality and the previous convergence result for the frequency process, to prove the analogous result for the genealogy.
\begin{proof}[Proof of Theorem \ref{thm: ancestral converg}]
The key in this proof is to show that the discrete  semigroup $P^N$ of $(Z^N(\lfloor-\rho_N^{-1}t\rfloor):t\geq 0)$ converges for any $t>0$ strongly to the semigroup $P$ of $Z:=(Z(t):t\geq0)$, when applied to functions of the form $f_\lambda(n)=e^{-\lambda n}$, $\lambda >0$. Since  $Z$ is Feller by Lemma \ref{lem:fellerZ} and the functions $f_{\lambda}$ are dense in the set of functions that are continuous and vanish at infinity in this topology, this implies weak convergence by Theorem 19.28, (ii) of \cite{Kallenberg} (Theorem 17.28 in the other editions).

The convergence of the semigroups will prove to be a consequence of the convergence in Theorem \ref{lem:converges} using the moment duality from Lemma \ref{lem:duality} for $X$ and $Z$ and the sampling duality from Proposition \ref{prop:iidduality} for $X^N$ and $Z^N$. First, observe that the sampling duality approximates the moment duality for large $N$ for our choice of $\mu_N$: conditioning on the choice of distribution for $Y^N_0$, we can rewrite the sampling dualiy function in the notation of Proposition \ref{prop:iidduality} as
\begin{align*}
 H_{\mu_N}(x,n) 	& = (1-c_N)\E[\varphi_{Y_0^N}(x)^n] + c_N\E[\varphi_{Y_0^N}((1-U)x+U(1-B_x))^n]\\
			& = (1-c_N)(1-\rho_N\gamma_N([0,1]))x^n \left(\frac{1-w_N}{1-xw_N}\right)^n \\
			& \qquad \qquad \qquad + \E[\varphi_{Y_0^N}(x)^n\,\mid\,Y^N_0\sim\bar\gamma_N]\rho_N\gamma_N([0,1]) + O(c_N)\\
			& = (1-c_N)(1-\rho_N\gamma_N([0,1]))x^n \left(\frac{1-w_N}{1-xw_N}\right)^n + O(\rho_N\gamma_N([0,1])) + O(c_N)\\
			& = x^n \left(\frac{1-w_N}{1-xw_N}\right)^n + O(\rho_N\gamma_N([0,1])) + O(c_N)
\end{align*}
where we crudely estimated the expectations and the expression in the first summand by 1. Note that this way the convergence of the remainder terms is uniform in $n$ and $x$.
Then the sampling duality can be written as
\begin{align*}
 \E^n&\left[x^{Z^N(-g)}\left(\frac{1-w_N}{1-xw_N}\right)^{Z^N(-g)}\right]\\
      & \hspace{40pt}= \E_x\left[(X^N(g))^n\left(\frac{1-w_N}{1-X^N(g)w_N}\right)^n\right] + O(\rho_N\gamma_N([0,1])) + O(c_N)
\end{align*}
which we can rearrange to
\begin{align}\label{eq:samplig_to_moment}
 \E^n\left[x^{Z^N(-g)}\right] 	& = \E_x\left[(X^N(g))^n\left(\frac{1-w_N}{1-X^N(g)w_N}\right)^n\right]+ O(\rho_N\gamma_N([0,1]))  + O(c_N)\notag  \\
				& \qquad  + \E^n\left[x^{Z^N(-g)}\left(1-\left(\frac{1-w_N}{1-xw_N}\right)^{Z^N(-g)}\right)\right] .
\end{align}
Note that the function $h_N:[0,1] \rightarrow [0,1]$, $h_N(x):= x^n \left(\frac{1-w_N}{1-xw_N}\right)^n$, is continuous and bounded and converges uniformly to $h(x):=x^n$. In addition, the sequence is decreasing in $n$. Hence, the weak convergence of $X^N(\lfloor\rho_N^{-1}t\rfloor) \rightarrow X(t)$ from Theorem \ref{lem:converges}, implies
\begin{align*}
 \E_x\left[(X^N(\lfloor\rho_N^{-1}t\rfloor))^n\left(\frac{1-w_N}{1-X^N(\lfloor\rho_N^{-1}t\rfloor)w_N}\right)^n\right] \xrightarrow{N\rightarrow \infty} \E_x[X(t)^n]
\end{align*}
for any $t\geq 0$ and any $x \in [0,1]$, uniformly in $n$, which takes care of the first summand. For the second, 
consider the functions $\tilde h_{N,x}:[0,1]\rightarrow [0,1]$, $\tilde h_{N,x}(u):=\exp\left(u^{-1}\log\left(x \frac{1-w_N}{1-xw_N}\right)\right)\1_{\{0\}}(u)$ and note that by Dini's Theorem this converges to $\tilde h_x(u):=\exp(u^{-1}\log(x))\1_{\{0\}}(u)=x^{(u^{-1})}\1_{\{0\}}(u)$ uniformly over all $u \in [0,1]$. This allows us to estimate
\begin{align*}
 0 & \leq \;\E^n\bigg[x^{Z^N(-g)}	 \left(1-\left(\frac{1-w_N}{1-xw_N}\right)^{Z^N(-g)}\right)\bigg]	\\
			& \leq \sup_{k\in\N}\;\;\underbrace{x^k\left(1-\left(\frac{1-w_N}{1-xw_N}\right)^{k}\right)}_{\tilde h_x(k^{-1})-\tilde h_{N,x}(k^{-1})}\\
			& \leq \sup_{u \in [0,1]}\;\;(\tilde h_x(u)-\tilde h_{N,x}(u))\xrightarrow{N\rightarrow \infty} 0\\
\end{align*}
for any $g \in \Z$ and $x \in [0,1]$. Note that this convergence is again uniform in $n$.
Therefore, rescaling time and taking the limits in \eqref{eq:samplig_to_moment} we obtain
\begin{align*}
 \lim_{N\rightarrow \infty} \E^n&\left[x^{Z^N(-\lfloor\rho_N^{-1}t\rfloor)}\right] = \E_x[X(t)^n]
\end{align*}
for every $t\geq 0$ and $x \in [0,1]$, uniformly in $n \in \N_0$. Now to prove convergence of the semigroups fix $t>0$ and $\lambda>0$. Define $x:=e^{-\lambda}$ and observe that, using the observations above and the moment duality from Lemma \ref{lem:duality}
\begin{align*}
P^N_tf_\lambda(n)	& = \E_n[f_\lambda(Z^N(-\lfloor\rho_N^{-1}t\rfloor))] = \E_n[x^{Z^N(-\lfloor\rho_N^{-1}t\rfloor)}]\\			
			& \quad\qquad\qquad  \xrightarrow{N\rightarrow \infty}\; \E_x[X(t)^n]=\E_n[x^{Z(t)}]=P_tf_\lambda(n),
\end{align*}
uniformly in $n$, which completes the proof. 
\end{proof}


\subsection{Long-term behaviour}\label{sec:ltr}

Given its simplicity the proof of Theorem \ref{thm:XlongtermGriffiths} was given in Section \ref{subsec:longterm_X}, hence we only have to give the proofs of Lemma \ref{lem:amazinggrace} (Section \ref{subsec:proof_of_amazingGrace}) and Lemma \ref{lem:Xlongterm_pure_Griffiths} (Section \ref{subsec:proof_of_Lyapunov_lemma}). The proof of the latter requires the representation of the generator of the $\Lambda$-Fleming-Viot process with weak and rare selection given in Lemma \ref{lem:Griffithsgenerator_general} as well as two technical estimates summarized in Proposition \ref{prop:bounds_operator}. The proposition is stated in Section \ref{subsec:proof_of_Lyapunov_lemma} together with the proof of Lemma \ref{lem:Xlongterm_pure_Griffiths}, but the proofs of the proposition and Lemma are in turn postponed to Sections \ref{subsec:help_longterm} and \ref{subsec:Griffiths_proof}.

\subsubsection{Proof of Lemma \ref{lem:amazinggrace}}\label{subsec:proof_of_amazingGrace}

The proof of this Lemma makes use of the approximation of the $\Lambda$-Fleming-Viot process with the individual-based-model defined in Section \ref{subsec:DASG}. 

\begin{proof}[Proof of Lemma \ref{lem:amazinggrace}]
 As mentioned, the assumption $\P_x(X_{\infty}=0)=1$ should imply that the term $\E_x[X(t)]$ decays exponentially. The idea of the proof is to obtain a geometric bound for the analogous quantity in the finite population size model. This will be done by estimating the probability that a uniformly chosen individual and an individual of the strong type picked from the same generation have a disjoint ancestry. A crucial observation for this estimate is that the assumption $\P_x(X_{\infty}=0)=1$ implies an infinite genealogy of strong individuals, which will allow us to sample the strong individual uniformly and independently of the first individual. 
\\
Since we know that the limit for $t\rightarrow \infty$ of $X(t)$ exists and takes values in $\{0,1\}$ $\P_x$-a.s.\ for any $x \in [0,1]$, the implication ``$\Rightarrow$'' is immedate.
 
  The converse ``$\Leftarrow$'', however, requires more care. Let us first consider the case where $\sigma + c = 0$, i.e.\ there is no genetic drift. Since we assumed the process not to be constant, we know $\alpha_{\mathfrak s}\alpha^* + w > 0$, and by Lemma \ref{lem:Xlongterm_pure_Griffiths}.\ref{item:selectiontwins_helpfullemma} $\int_{[0,\infty[}\E_x[X(s)]\dd s < \infty$ (for all $ x \in \,[0,1[$), which suffices for ``$\Leftarrow$''.

 Assume now that $\sigma + c >0$.
The proof is split in three steps and uses that $(X(t))_{t\geq 0}$ is the (appropriately rescaled) weak limit of a finite particle system $(X^N(g))_{g \in \N_0}$, as observed in Theorem \ref{thm:convergence_rare_and_weak}. (Without loss of generality we let the time-index of $X^N$ start in $g_0=0$ and hence omit the additional superscript.) Since $(X^N(g))_{g \in \N_0}$ is a Markov chain on a finite state-space with two absorbing states 0 and 1, it will be absorbed in finite time, hence $X^N(\infty):=\lim_{g \rightarrow \infty}X^N(g) \in \{0,1\}$ exists $\P_x$-a.s.\ for any $x \in [N]_0/N$. 
 
 Fix $x\in [0,1[$ rational. All $N$ appearing henceforth are assumed to be such that $x \in [N]_0/N$.
 
 \textbf{Step 1:} We begin by proving that the assumption $\P_x(X_{\infty}=0)=1$ in particular implies
 \begin{align}\label{eq:prob_fix_N_converges}
  \lim_{N\rightarrow \infty}\P_x\left(X^N(\infty)=0\right) = 1.
 \end{align}
 Recall the duality from Proposition \ref{prop:iidduality}:
 \begin{align*}
  \E^n[H_{\mu^N}(x, Z^N(-g))] = \E_x[H_{\mu^N}(X^N(g), n)]
 \end{align*}
 for any $n \in \N$. The map $\bar x \mapsto H_{\mu^N}(\bar x,n)$ is continuous and bounded (for any $n \in \N_0$) with $H_{\mu^N}(0,n)=0$ and $H_{\mu^N}(1,n)=1$ therefore
 \begin{align}\label{eq:duality_converges_to_fixation_prob}
  \lim_{g \rightarrow \infty} \E_x[H_{\mu^N}(X^N(g), n)] & = \P_x(X^N(\infty)=1).
 \end{align}
Since for fixed $N \in \N$ the state space of $(Z^N(-g))_{g \in \N_0}$ is finite, it has an invariant distribution which we will denote by $\nu^N$ (and assume to be defined on $\N_0$). Let $\E^{\nu^N}$ denote the expectation under the law of the process with initial distribution $\nu^N$. Then, using again the duality
 \begin{align*}
  \E^{\nu^N}[H_{\mu^N}(x, Z^N(-g))] & = \sum_{n \in \N_0} \nu^N(n)\E^n[H_{\mu^N}(x, Z^N(-g))] \\
& = \sum_{n \in \N_0} \nu^N(n)\E_x[H_{\mu^N}(X^N(g), n)]\\
& \xrightarrow{g \rightarrow \infty} \P_x(X^N(\infty)=1),
 \end{align*}
 by \eqref{eq:duality_converges_to_fixation_prob} as the sum is actually finite and $\nu^N$ a distribution. However, the left-hand-side does actually not depend on $g$, since $\nu^N$ is precisely the invariant distribution. Therefore we can substitute the limit by an equality and obtain the representation
 \begin{align}\label{eq:fixation_prob_for_any_g}
   \sum_{n \in \N_0} \nu^N(n)\E_x[H_{\mu^N}(X^N(g), n)] = \P_x(X^N(\infty)=1)
 \end{align}
 for any $g \in \N_0$. One can check that for any $N \in \N$, $\bar x \in [0,1]$ and $n \in \N_0$ 
 \begin{align*}
  H_{\mu^N}(\bar x, n) \leq H_{\mu^N}(\bar x, 1)
 \end{align*}
 and that $\bar x\mapsto H_{\mu^N}(\bar x, 1)$ converges to the identity map $\bar x\mapsto \bar x$ uniformly in $[0,1]$, as $N\rightarrow \infty$. By the latter, the weak convergence of $X^N(\lfloor \rho_N^{-1} t\rfloor)$ to $X(t)$ implies
 \begin{align*}
  \E_x[H_{\mu^N}(X^N(\lfloor \rho_N^{-1} t\rfloor), 1)] \xrightarrow{N\rightarrow \infty} \E_x[X(t)]
 \end{align*}
 for every $t \geq 0$. Thus, using \eqref{eq:fixation_prob_for_any_g} with $g = \lfloor \rho_N^{-1} t\rfloor$ and the observations above
 \begin{align*}
  0 \leq \; \limsup_{N\rightarrow \infty} \P_x(X^N(\infty)=1) & = \limsup_{N\rightarrow \infty} \sum_{n \in \N_0} \nu^N(n)\E_x[H_{\mu^N}(X^N(\lfloor \rho_N^{-1} t\rfloor), n)] \\ 
& \leq \limsup_{N\rightarrow \infty} \E_x[H_{\mu^N}(X^N(\lfloor \rho_N^{-1} t\rfloor), 1)]{\sum_{n \in \N_0} \nu^N(n)} \\
& = \E_x[X(t)]
 \end{align*}
 for any $t \geq 0$, where the last equality follows from the fact that ${\sum_{n \in \N_0} \nu^N(n)} =1$. Under the assumption $\P_x(X_{\infty}=0)=1$, however, the right-hand-side converges to 0, as $t \rightarrow \infty$, and since the left-hand-side does not depend on $t$, we have proven \eqref{eq:prob_fix_N_converges}.
 
 \textbf{Step 2:} Step 1 now allows us to condition on the event of extinction of the finite-population process (for $N$ sufficiently large). Denote by $\bar \P_x^N$ the probability measure obtained from $\P_x$ by conditioning on $\{X^N(\infty)=0\}$. Let $T^N_0$ be the time to extinction of $X^N$, i.e.
 \begin{align*}
  T^N_0:=\inf\{g \geq 0 \mid X^N(g)=0\}.
 \end{align*}
Hence, we are precisely conditioning on $T^N_0$ being finite. By definition of $T^N_0$, any individual in a generation after $T^N_0$ is of the strong type 1. Recall from Definition \ref{def:discrete_genealogy}, that we denote by $A^N_l(w)$ the potential ancestors of an individual $w$ alive $l$ generations back in time (counting from $g(w)\in \N_0$, the generation of $w$). Analogously, define the \emph{strong potential ancestors} of $w$ in the typed process alive $l$ generations back
 \begin{align*}
  S\hspace{-2pt} A^N_l(w):=\{v \in A^N_l(w) \mid v \text{ is of type 1}\},\quad  l = 0, \ldots, g(w).
 \end{align*}
 Note that $S\hspace{-2pt} A^N_l(w) = \emptyset$ if, and only if, $w$ is of the weak type 0. 
 
 Fix a generation $g \geq 0$. We want to prove a geometric bound on $\bar \E_x^N[X^N(g)]$. First, note that for $I$ uniform on $[N]$ and independent of all other randomness in the model, we can express 
 \begin{align*}
  \bar\E_x^N[X^N(g)] & = \bar \P_x^N((g,I) \text{ is of type 0}).
  \end{align*}
 We will use the fact that, for an individual from generation $g$ to be of the weak type 0, its ancestry must be disjoint from the strong ancestry of any strong individual in the same generation $g$ (which we know to exist $\P^N_x$-a.s.). However, in order to estimate the probability of this event, said individual must be chosen with care.
 
 Let $J$ be another random variable uniform on $[N]$, independent of any other randomness and use it to choose uniformly at random an individual from the first generation not before $g$ in which the strong type has fixated: 
 \begin{align*}
 W:=(\max\{g,T^N_0\}, J).
 \end{align*}
 Like all individuals in its generation, $W$ is of the strong type 1. $ S\hspace{-2pt} A^N_{\max\{g,T^N_0\}-g}(W)$ are its strong potential ancestors alive in generation $g$. Consequently, choose $W'$ uniformly at random (independently of all other randomness) among those  $ S\hspace{-2pt} A^N_{\max\{g,T^N_0\}-g}(W)$. Then $W'$ is by definition a strong individual in generation $g$. Due to exchangeability of the underlying ancestral process, however, its label $i(W')$ is uniform on $[N]$ and independent of $I$, the label of $(g,I)$, the first individual we chose. 
 
 As remarked, for $(g,I)$ to be of the weak type 0, its ancestry $A^N_l((g,I))$ must be, in particular, disjoint from the (non-empty) strong ancestry $S\hspace{-2pt} A^N_{l}(W')$ of $W'$, for every previous generation, allowing us to estimate 
 \begin{align*}
  \bar \P_x^N((g,I) \text{ is of type }0) & \leq \bar \P_x^N(\,\forall\, l=0.\ldots, g:\;A^N_l((g,I))\cap S\hspace{-2pt} A^N_{l}(W')=\emptyset).
 \end{align*}
The probability for two given individuals to find a common ancestor of any type in the previous generation can be bounded from above by the probability $p_N$ for this same event with the selection mechanism \lq \lq turned off\rq\rq, whereby each individual always samples one and only one parent. A close look at the parent-choosing mechanism described in Section \ref{subsec:DASG} reveals that 
\begin{align}\label{eq:p_N}
 p_N = (1-c_N)/N + m_1c_N + m_2c_N/N 
\end{align}
 for positive, finite constants $m_1$ and $m_2$. Indeed, recall that in this set-up we have a generation with no skewed offspring distribution with probability $1-c_N$, in which case each individual chooses its parent uniformly, independently of the others and hence the probability of two given individuals choosing the same parent is $1/N$. In a generation with skewed offspring distribution (ocurring with probability $c_N$) the strength of the correlation, i.e. the probability of choosing the specially marked parent was given by $V\sim \Lambda_{\mathfrak c}$ which was sampled once for the whole generation. In this case we have three possibilities of choosing the same parent: Either both individuas choose the specially marked one, which happens with probability $V^2$. Or one chooses the specially marked and the other chooses uniformly at random, but still ends up with the same individual, which has probability $V(1-V)/N$. Or both choose uniformly and happen to choose the same, which has probability $(1-V)^2/N$. Together this means that
 \begin{align*}
  p_N	& = (1-c_N)\frac{1}{N} + c_N\int_{[0,1]} \left\{ v^2 + v(1-v)\frac{1}{N} + (1-v)^2\frac{1}{N}\right\}\Lambda_{\mathfrak c}(\dd v)\\
	& =  (1-c_N)\frac{1}{N} + c_N\int_{[0,1]} v^2 \Lambda_{\mathfrak c}(\dd v) +  \frac{c_N}{N}\int_{[0,1]}(1-v)\Lambda_{\mathfrak c}(\dd v).
 \end{align*}

 Since we constructed $W'$ and $(g,I)$ such that they are both uniform on $[N]$ and independent, by exchangeability and independence of the generations we can use this estimate for every generation and obtain
\begin{align*}
 \bar\E_x^N[X^N(g)] &\leq \bar \P_x^N(\,\forall\, l=0.\ldots, g:\;A^N_l((g,I))\cap S\hspace{-2pt} A^N_{l}(W')=\emptyset)\\
		    & \leq (1-p_N)^{g}\left(1-\frac{1}{N}\right),
\end{align*}
 where the last factor is the probability of the two individuals $(g,I)$ and $W'$ not being identical.

\textbf{Step 3:} All that remains to be done is to translate this bound into a bound for the scaling limit. Since the identity map is bounded and continuous on $[0,1]$ Theorem \ref{thm:convergence_rare_and_weak} implies
\begin{align*}
 \E_x[X^N(\lfloor \rho_N^{-1}t\rfloor)] \xrightarrow{N\rightarrow \infty} \E_x[X(t)]
\end{align*}
for any $t \geq 0$. At the same time, since $ \lim_{N\rightarrow \infty}\P_x\left(X^N(\infty)=0\right) = 1$ by step 1, using bounded convergence, we see
\begin{align*}
 \big\vert  \E_x[X(t)]- & \bar\E_x^N[X^N(\lfloor \rho_N^{-1}t\rfloor)] \big\vert \\
			& \leq \vert  \E_x[X(t)]-\E_x[X^N(\lfloor \rho_N^{-1}t\rfloor)] \vert \\
			& \qquad + \left\vert \E_x\left[X^N(\lfloor \rho_N^{-1}t\rfloor)\left(1 - \frac{\1_{\{X^N(\infty)=0\}}}{\P_x(X^N(\infty)=0)}\right)\right] \right\vert\\
			& \qquad \xrightarrow{N\rightarrow \infty} 0.
\end{align*}
Therefore the estimate from step 2 yields
\begin{align*}
  \E_x[X(t)] & = \lim_{N\rightarrow \infty} \bar\E_x^N[X^N(\lfloor \rho_N^{-1}t\rfloor)] \leq \lim_{N\rightarrow \infty}(1-p_N)^{\lfloor \rho_N^{-1}t\rfloor} 
\end{align*}
for every $t \geq 0$. Under the assumptions of Lemma \ref{lem:converges}, $\lim_{N\rightarrow\infty} p_N\rho_N^{-1} = \sigma + m_1c>0$, and therefore
 \begin{align*}
   \E_x[X(t)] \leq \lim_{N\rightarrow \infty}(1-p_N)^{\lfloor \rho_N^{-1}t\rfloor} = \exp(-(\sigma+m_1c)t)
 \end{align*}
which in turn implies $\int_{0}^\infty\E_x[X(s)]\dd s <\infty$.

Finally we remove the assumption of rationality on $x$. Thus, let $x \in [0,1[$ and assume $\P_x(X_{\infty}=0)=1$. By Remark \ref{rem:Xto01}, we may therefore choose an $\tilde x\in[x,1[$ rational, such that  $\P_{\tilde x}(X_{\infty}=0)=1$. For such $\tilde x$ we have just proven that $\int_{0}^\infty \E_{\tilde x}[X(s)] \dd s<\infty$. The claim follows since $\E_{x}[X(t)] \leq \E_{\tilde x}[X(t)]$ for every $t \geq 0$. 
\end{proof}


\subsubsection{Proof of Lemma \ref{lem:Xlongterm_pure_Griffiths} via Lyapunov functions}\label{subsec:proof_of_Lyapunov_lemma}

For the proof of Lemma \ref{lem:Xlongterm_pure_Griffiths} we use the same Lyapunov approach as in \cite{GriffithsLambda} to attain slightly weaker conclusions than \cite{GriffithsLambda} in that extra care was needed to avoid the risk of some illegal exchanges of expectation and integration.
To this end, let $\kappa \in \R$ and $N \in \N$ and consider the functions $f_{N,\kappa}: [0,1] \rightarrow \R$ defined as
 \begin{align}\label{eq:def_Lyapunov_truncated_functions}
  f_{N,\kappa}(x):= \sum_{n=1}^N\frac{1}{n} x^n - \kappa x.
  \end{align}
 Note that for $N \rightarrow \infty$ the sum converges to $-\log(1-x)$, but for each finite $N$ they have convenient integrability properties which will allow us to use these functions for a Lyapunov-type argument to prove both parts of the lemma.
 
 Before we come to the proof we first summarize some convenient observations on these functions and their interplay with the generator of $X$ given in \eqref{eq:generatorX} in the following proposition. Its proof is postponed to Section \ref{subsec:help_longterm}.

 Recall from Section \ref{subsec:longterm_X} the definitions of the strength of selection $w$, $\alpha_{\mathfrak s}$ and $\alpha^*$ and the strength of genetic drift $\beta^*$ and the quantities involved in it: 
   $W:=Y^{\mathfrak c} U$, where $Y^{\mathfrak c}\sim \Lambda_{\mathfrak c}$, $U$ has density $2u$ on $[0,1]$ and they are independent; $\alpha_{\mathfrak s} :=\Lambda_{\mathfrak s}([0,1])$ and  $V$ is uniform on $[0,1]$ and $Y^{\mathfrak s} \sim \alpha_{\mathfrak s}^{-1}\Lambda_{\mathfrak s}([0,1])$ and they are independent.

 \begin{prop}\label{prop:bounds_operator}
  For $\kappa \in \R$ and $N \in \N$ let $f_{N,\kappa}: [0,1] \rightarrow \R$ be defined as in \eqref{eq:def_Lyapunov_truncated_functions} above and recall the generator $\mathcal A$ of the $\Lambda$-Fleming-Viot process with weak and rare selection  given in \eqref{eq:generatorX}. Then the following holds:
  \begin{enumerate}
   \item We can define an operator $\tilde{\mathcal{A}}$ such that for all $\kappa \in \R$, $N \in \N$ and $x \in [0,1]$ 
   \begin{align*}
      \mathcal Af_{N, \kappa}(x)= x\tilde{ \mathcal{A}}f_{N, \kappa}(x).
   \end{align*}\label{prop_subclaim:bounds_operator_Atilde}
   \item For this operator we have lower and upper bounds $\underline{a}$ and $\overline{a}$ such that for all $\kappa \in \R$, $N \in \N$ and $x \in [0,1]$
   \begin{align*}
    \underline{a}(N,x,\kappa) \leq \tilde{ \mathcal{A}}f_{N, \kappa}(x) \leq \overline{a}(N,x,\kappa)
   \end{align*}
    where
   \begin{align}\label{eq:lower_bound_a}
    \underline{a}(N,x,\kappa)	& = \frac{1}{2}c\E\left[\frac{1}{(1-W)(1-x(1-W))}\right] -w \notag\\
				& \qquad -\alpha_{\mathfrak s}\E\left[\E\left[\sum_{l=0}^{K_{Y^{\mathfrak s}}-2}x^l\,\Big|\, Y^{\mathfrak s}\right]\frac{1}{ \E[K_{Y^{\mathfrak s}} -1\,\mid\,Y^{\mathfrak s} ]}\frac{1}{1+xV\E\left[\sum_{l=0}^{K_{Y^{\mathfrak s}}-2}x^l\,\mid\, Y^{\mathfrak s}\right]}\right]\notag\\
				& \qquad +\kappa w (1-x)+\kappa\alpha_{\mathfrak s}(1-x)\E\left[\E\left[\sum_{l=0}^{K_{Y^{\mathfrak s}}-2}x^l\,\Big|\, Y^{\mathfrak s}\right]\frac{1}{ \E[K_{Y^{\mathfrak s}} -1\,\mid\,Y^{\mathfrak s} ]}\right]\notag\\
				& \qquad - \frac{1}{2}c\E\left[\frac{(x(1-W)+W)^N}{W(1-W)}\right] 
   \end{align}
   \begin{align}\label{eq:upper_bound_a}
  \overline{a}(N,x,\kappa)	& = \frac{1}{2}c\E\left[\frac{1}{W(1-W)} \right] - w(1-x^N)\notag\\
				& \qquad + w\kappa(1-x)+\alpha_{\mathfrak s}\kappa(1-x)\E\Bigg[\E\left[\sum_{l=0}^{K_{Y^{\mathfrak s}}-2}x^l\,\Big|\, Y^{\mathfrak s}\right]\frac{1}{ \E[K_{Y^{\mathfrak s}} -1\,\mid\,Y^{\mathfrak s} ]}\Bigg] \notag\\
				& \qquad  -\alpha_{\mathfrak s}\E\Bigg[\E\left[\sum_{l=0}^{K_{Y^{\mathfrak s}}-2}x^l\,\Big|\, Y^{\mathfrak s}\right]\frac{1}{ \E[K_{Y^{\mathfrak s}} -1\,\mid\,Y^{\mathfrak s}]}\notag\\
				& \qquad \qquad \qquad \qquad \qquad \qquad \times\frac{1-\left(x-x(1-x)V\E\left[\sum_{l=0}^{K_{Y^{\mathfrak s}}-2}x^l\,\mid\, Y^{\mathfrak s}\right]\right)^N}{1+xV\E\left[\sum_{l=0}^{K_{Y^{\mathfrak s}}-2}x^l\,\mid\, Y^{\mathfrak s}\right]}\Bigg].
   \end{align}
  and thus
  \begin{align}
   \vert \underline{a}(N,x,\kappa) \vert \leq c\beta^* + w + \alpha_{\mathfrak s} + w  +\vert\kappa\vert\alpha_{\mathfrak s}+ c\beta^*<\infty,\label{eq:lower_bound_a_upperbounded}\\
   \vert \overline{a}(N,x,\kappa) \vert \leq c\beta^* + w + w\vert\kappa\vert + \alpha_{\mathfrak s} \vert\kappa\vert+ \alpha_{\mathfrak s}\alpha^* <\infty. \label{eq:upper_bound_a_upperbounded}
  \end{align}
  \item In particular, 
  \begin{align}
 \vert {\tilde{\mathcal A}}f_{N,\kappa}(x)\vert \leq 2c\beta^* + (2+\vert\kappa\vert)w+ (1+\vert\kappa\vert +\alpha^*)\alpha_{\mathfrak s}< \infty. \label{eq:A_with_Lyapunov_is_bounded}
\end{align} \label{prop_subclaim:A_with_Lyapunov_is_bounded}
 \end{enumerate}
 \end{prop}

\begin{proof}[Proof of Lemma \ref{lem:Xlongterm_pure_Griffiths}]
Proposition \ref{prop:bounds_operator}.\ref{prop_subclaim:A_with_Lyapunov_is_bounded}, i.e.\ equation \eqref{eq:A_with_Lyapunov_is_bounded}, in particular implies that $ \vert\tilde{\mathcal A}f_{N,\kappa}(x)\vert < \infty $ uniformly in $x \in [0,1]$. Therefore, for any  $N\in \N$ and $\kappa \in \R$ 
\begin{align*}
 \left(f_{N,\kappa}(X(t)) - f_{N,\kappa}(X(0)) - \int_{[0,t]}\mathcal Af_{N,\kappa}(X(s))\dd s\right)_{t\geq 0}
\end{align*}
is a martingale and we know
\begin{align}\label{eq:Martingale}
 \E_x[f_{N,\kappa}(X(t))] - f_{N,\kappa}(x) = \int_{[0,t]}\E_x[\mathcal Af_{N,\kappa}(X(s))]\dd s
\end{align}
holds for all $N\in \N$, $\kappa \in \R$, $x \in [0,1]$ and $t\geq 0$. 

Before we turn to the proof of the two statements of the lemma, we draw a first helpful conclusion: Under the assumption $\beta^*<\infty$, \eqref{eq:Martingale} implies that for any $x \in [0,1[$ and any $t\geq 0$ the probability of the weak allele having reached fixation is zero, i.e.
\begin{align}\label{eq:helpneverhit1}
 \P_x(X(t)=1)=0. 
\end{align}
 If this were not the case there would exist $\bar x\in[0,1[$ and $\bar t>0$ such that $\P_{\bar x}(X(\bar t)=1)>0$. For this choice of parameters
\begin{align*}
 \E_{\bar x}[f_{N,0}(X(\bar t))] - f_{N,0}(\bar x) \quad \geq \quad  \P_{\bar x}(X(\bar t)=1)
\sum_{n=1}^N\frac{1}{n} - \sum_{n=1}^N\frac{1}{n}{\bar x}^n \xrightarrow{N\rightarrow \infty} \infty.
\end{align*}
On the other hand, \eqref{eq:Martingale} implies
\begin{align*}
 \E_{\bar x}[f_{N,0}(X(\bar t))] - f_{N,0}(\bar x) & = \int_{[0,\bar t]}\E_{\bar x}[\mathcal Af_{N,\kappa}(X(s))]\dd s \\
                                                    & \leq \bar t( 2c\beta^* + (2+\vert\kappa\vert)w+ (1+\vert\kappa\vert +\alpha^*)\alpha_{\mathfrak s}) <\infty
\end{align*}
for any $N \in \N$, which is a contradiction and proves the claim. 

Indeed, since $1$ is an absorbing state the sets $\{X_t=1\}$ are increasing in $t\geq 0$ and therefore we have actually proven that the fixation time is almost surely infinite, but we will not use this here. 

We are now ready to tackle the lemma itself.

\textbf{Proof of \textit{(i)}:} Assume $\alpha_{\mathfrak s}\alpha^* +w<c\beta^* $. Equivalently, $c\beta^*-w-\alpha_{\mathfrak s}\alpha^* >0$. For this part we will work with the lower bound $\underline{a}$. Note that parts of the first and third summand of $\underline{a}$ converge (e.g.\ by bounded convergence) to the critical values as $x$ tends to 1:
\begin{align*}
 \frac{1}{2}\E\left[\frac{1}{(1-W)(1-x(1-W))}\right] & \xrightarrow{x\rightarrow 1} \beta^*
 \end{align*}
 and
 \begin{align*}
 \E\left[\E\left[\sum_{l=0}^{K_{Y^{\mathfrak s}}-2}x^l\,\Big|\, Y^{\mathfrak s}\right]\frac{1}{ \E[K_{Y^{\mathfrak s}} -1\,\mid\,Y^{\mathfrak s} ]}\frac{1}{1+xV\E\left[\sum_{l=0}^{K_{Y^{\mathfrak s}}-2}x^l\,\mid\, Y^{\mathfrak s}\right]}\right] & \xrightarrow{x\rightarrow 1} \alpha^* 
\end{align*}
whence we may choose an $x_0\in [0,1[$ such that for all $x \geq x_0$
\begin{align*}
 c\frac{1}{2}& \E\left[\frac{1}{(1-W)(1-x(1-W))}\right] -w \\
	    & \qquad -\alpha_{\mathfrak s}\E\left[\E\left[\sum_{l=0}^{K_{Y^{\mathfrak s}}-2}x^l\,\Big|\, Y^{\mathfrak s}\right]\frac{1}{ \E[K_{Y^{\mathfrak s}} -1\,\mid\,Y^{\mathfrak s} ]}\frac{1}{1+xV\E\left[\sum_{l=0}^{K_{Y^{\mathfrak s}}-2}x^l\,\mid\, Y^{\mathfrak s}\right]}\right]\\
	    & \geq \frac{c\beta^*-w -\alpha_{\mathfrak s}\alpha^*}{2}=:\delta>0.
\end{align*}
We use $\kappa$ to control the remainig $x\leq x_0$. Hence, fix $\kappa_0\geq 0$ sufficiently large such that 
\begin{align}\label{eq:helpkappa0}
 \kappa_0w(1-x_0)+\kappa_0\alpha_{\mathfrak s}(1-x_0)\E\left[\E\left[\sum_{l=0}^{K_{Y^{\mathfrak s}}-2}x_0^l\,\Big|\, Y^{\mathfrak s}\right]\frac{1}{ \E[K_{Y^{\mathfrak s}} -1\,\mid\,Y^{\mathfrak s} ]}\right] \notag\\
 \geq w + \alpha_{\mathfrak s} - \frac{1}{2}\E\left[\frac{1}{1-W}\right] + \delta.
\end{align}
Note that $(1-x_0)\sum_{l=0}^{k-2} x_0^l = 1-x_0^{k-1}$, thus the left-hand-side is decreasing in $x$ and therefore \eqref{eq:helpkappa0} actually holds for all $x \leq x_0$. 

Combined, we obtain for all but the last summand of $\underline{a}$: For all $x \in [0,1]$
\begin{align*}
 c\frac{1}{2}& \E\left[\frac{1}{(1-W)(1-x(1-W))}\right] - w\\
	    & \qquad  -\alpha_{\mathfrak s}\E\left[\E\left[\sum_{l=0}^{K_{Y^{\mathfrak s}}-2}x^l\,\Big|\, Y^{\mathfrak s}\right]\frac{1}{ \E[K_{Y^{\mathfrak s}} -1\,\mid\,Y^{\mathfrak s} ]}\frac{1}{1+xV\E\left[\sum_{l=0}^{K_{Y^{\mathfrak s}}-2}x^l\,\mid\, Y^{\mathfrak s}\right]}\right]\\
	    & \qquad + \kappa_0w(1-x) + \kappa_0\alpha_{\mathfrak s}(1-x)\E\left[\E\left[\sum_{l=0}^{K_{Y^{\mathfrak s}}-2}x^l\,\Big|\, Y^{\mathfrak s}\right]\frac{1}{ \E[K_{Y^{\mathfrak s}} -1\,\mid\,Y^{\mathfrak s} ]}\right] \\
	    & \qquad \geq \qquad  \delta
\end{align*}
which in turn implies
\begin{align*}
 \underline{a}(N,x,\kappa_0) \geq \delta - c\frac{1}{2}\E\left[\frac{(x(1-W)+W)^N}{W(1-W)}\right]
\end{align*}
for all $N\in\N$ and all $x\in[0,1]$.

All these estimates serve to control the right-hand-side of \eqref{eq:Martingale}. While we were able to do the previous estimates for deterministic $x$, the last term will be tackled using bounded convergence after substituting $X(s)$ for $x$ in the expressions above. 
In order to obtain a contradiction, assume
\begin{align}\label{eq:helpcontradiction_finite_integral}
 \forall x\in]0,1[\quad \int_{[0,\infty[}\E_x[X(s)]\dd s<\infty,
\end{align}
which in particular implies that for any $x \in [0,1[$ $\P_x(X_{\infty}=1)=0$ and therefore
\begin{align}\label{eq:supsmall}
X_{\sup}:=\sup_{t\in[0,\infty[}X(t) <1 \quad \P_x\text{-a.s.}
\end{align}
Without loss of generality we may assume $(X(t))_{t\geq 0}$ and $W$ to be independent and we can estimate
\begin{align}\label{eq:helpboundI1}
  \E_x\left[X(s)\underline{a}(N,X(s),\kappa_0)\right]	& \geq \E_x\left[X(s)\left(\delta - \frac{c}{2}\frac{(X(s)(1-W)+W)^N}{W(1-W)}\right)\right]\notag	\\
							& \geq \E_x\left[X(s)\left(\delta - \frac{c}{2}\frac{(X_{\sup}(1-W)+W)^N}{W(1-W)}\right)\right]
\end{align}
for all $N \in\N$. Using bounded convergence and \eqref{eq:supsmall} we see that the right-hand-side converges
\begin{align*}
 \E_x\left[X(s)\left(\delta - \frac{c}{2}\frac{(X_{\sup}(1-W)+W)^N}{W(1-W)}\right)\right] \xrightarrow{N\rightarrow \infty} \delta \E_x[X(s)] > 0.
\end{align*}
Assumption \eqref{eq:helpcontradiction_finite_integral} allows us to use bounded convergence again, to estimate
\begin{align*}
 \lim_{N\rightarrow \infty} \int_{[0,\infty]}\E_x&\left[X(s)\underline{a}(N,X(s),\kappa_0)\right]\dd s \\
		  & \geq \lim_{N\rightarrow \infty} \int_{[0,\infty]} \E_x\left[X(s)\left(\delta - \frac{c}{2}\frac{(X_{\sup}(1-W)+W)^N}{W(1-W)}\right)\right]\dd s\\
		  & = \delta\int_{[0,\infty]} \E_x\left[X(s)\right]\dd s := \tilde \delta(x)> 0
\end{align*}
 for any  $x \in [0,1[$. Recalling that $\underline{a}$ was designed as a lower bound (and using \eqref{eq:helpneverhit1}) this yields that for every $x \in [0,1[$ there exists a $ N_0 = N_0(x)$ such that for every $ N\geq N_0$ 
\begin{align}\label{eq:helplowerboudpositive}
  \int_{[0,\infty]}\E_x\left[\mathcal A f_{N,\kappa_0}(X(s))\right]\dd s & = \int_{[0,\infty]}\E_x\left[X(s)\tilde{\mathcal A} f_{N,\kappa_0}(X(s))\right]\dd s \notag \\
								  & \geq \int_{[0,\infty]}\E_x\left[X(s)\underline{a}(N,X(s),\kappa_0)\right]\dd s \geq  \frac{\tilde \delta(x)}{2}>0.
\end{align}
Now choose $\bar x$ sufficiently close to 1, such that 
\begin{align*}
 \log(1-\bar x) + \kappa_0\bar x < -1.
\end{align*}
Then choose $\bar N \geq N_0(\bar x)$ sufficiently large such that
\begin{align*}
 - \sum_{n=1}^{\bar N}\frac{1}{n}\bar x^n + \kappa_0 \bar x < -\frac{1}{2}.
\end{align*}
Recall that assumption \eqref{eq:helpcontradiction_finite_integral} in particular implies $\P_{\bar x}(X_{\infty}=0)=1$ and therefore
\begin{align*}
 \lim_{t\rightarrow \infty} \E_{\bar x}[f_{\bar N,\kappa_0}(X(t))] - f_{\bar N,\kappa_0}(\bar x) =  0 - \sum_{n=1}^{\bar N}\frac{1}{n}\bar x^n + \kappa_0 \bar x.
\end{align*}
Together with equation \eqref{eq:Martingale} and \eqref{eq:helplowerboudpositive} this implies
\begin{align*}
 -\frac{1}{2} >  \lim_{t\rightarrow \infty} \E_{\bar x}[f_{\bar N,\kappa_0}(X(t))] - f_{\bar N,\kappa_0}(\bar x) = \lim_{t\rightarrow \infty} \int_{[0,t]}\E_{\bar x}\left[\mathcal A f_{\bar N,\kappa_0}(X(s))\right]\dd s\geq  \frac{\tilde \delta(\bar x)}{2}>0,
\end{align*}
which clearly is a contradiction.

Hence, we have proven that assuming $\alpha_{\mathfrak s}\alpha^*+w < c\beta^*$, we know 
\begin{align*}
 \exists \bar x\in]0,1[\quad \int_{[0,\infty[}\E_{\bar x}[X(s)]\dd s=\infty.
\end{align*}

\textbf{Proof of \textit{(ii)}:} Assume $\alpha_{\mathfrak s}\alpha^* + w  > c\beta^*$. Equivalently $c\beta^*-w - \alpha_{\mathfrak s}\alpha^* <0$. For this part we will work with the upper bound $\overline{a}$ which can be rearranged to 
 \begin{align*}
  \overline{a}&(N,x,\kappa)	= c\beta^* - w - \alpha_{\mathfrak s}\E\Bigg[\E\left[\sum_{l=0}^{K_{Y^{\mathfrak s}}-2}x^l\,\Big|\, Y^{\mathfrak s}\right]\frac{1}{ \E[K_{Y^{\mathfrak s}} -1\,\mid\,Y^{\mathfrak s}]}\\
				& \qquad \qquad \qquad \qquad \qquad \qquad \qquad \qquad \qquad \qquad \times\frac{1}{1+xV\E\left[\sum_{l=0}^{K_{Y^{\mathfrak s}}-2}x^l\,\mid\, Y^{\mathfrak s}\right]}\Bigg] \\
				& \qquad +w\kappa(1-x) + \alpha_{\mathfrak s}\kappa(1-x)\E\left[\E\left[\sum_{l=0}^{K_{Y^{\mathfrak s}}-2}x^l\,\Big|\, Y^{\mathfrak s}\right]\frac{1}{ \E[K_{Y^{\mathfrak s}} -1\,\mid\,Y^{\mathfrak s} ]}\right] \\
				& \qquad  + w x^N + \alpha_{\mathfrak s}\E\Bigg[\E\left[\sum_{l=0}^{K_{Y^{\mathfrak s}}-2}x^l\,\Big|\, Y^{\mathfrak s}\right]\frac{1}{ \E[K_{Y^{\mathfrak s}} -1\,\mid\,Y^{\mathfrak s}]}\\
				& \qquad \qquad \qquad \qquad \qquad \qquad \qquad \times\frac{x^N\left(1-(1-x)V\E\left[\sum_{l=0}^{K_{Y^{\mathfrak s}}-2}x^l\,\mid\, Y^{\mathfrak s}\right]\right)^N}{1+xV\E\left[\sum_{l=0}^{K_{Y^{\mathfrak s}}-2}x^l\,\mid\, Y^{\mathfrak s}\right]}\Bigg].
 \end{align*}
 We proceed along the lines of the previous proof. Again, note that the third summand tends to the critical values for $x \rightarrow 1$
 \begin{align*}
  \E\left[\E\left[\sum_{l=0}^{K_{Y^{\mathfrak s}}-2}x^l\,\Big|\, Y^{\mathfrak s}\right]\frac{1}{ \E[K_{Y^{\mathfrak s}} -1\,\mid\,Y^{\mathfrak s}]}\frac{1}{1+xV\E\left[\sum_{l=0}^{K_{Y^{\mathfrak s}}-2}x^l\,\mid\, Y^{\mathfrak s}\right]}\right] \xrightarrow{x\rightarrow 1} \alpha^*.
 \end{align*}
Therefore there exists an $x_0 \in]0,1[$ such that for all $x\geq x_0$
\begin{align*}
  c\beta^* - w - \alpha_{\mathfrak s}\E&\left[\E\left[\sum_{l=0}^{K_{Y^{\mathfrak s}}-2}x^l\,\Big|\, Y^{\mathfrak s}\right]\frac{1}{ \E[K_{Y^{\mathfrak s}} -1\,\mid\,Y^{\mathfrak s}]}\frac{1}{1+xV\E\left[\sum_{l=0}^{K_{Y^{\mathfrak s}}-2}x^l\,\mid\, Y^{\mathfrak s}\right]}\right]\\
	    & \leq \frac{c\beta^* - w -\alpha_{\mathfrak s}\alpha*}{2}=:-\delta <0.
\end{align*}
Again, we control the remaining $x$ with the help of $\kappa$. Choose $\kappa_0\leq 0$ with $\vert \kappa_0\vert$ sufficiently large such that
\begin{align*}
w\kappa_0(1-x_0) + \alpha_{\mathfrak s}\kappa_0&(1-x_0)\E\left[\E\left[\sum_{l=0}^{K_{Y^{\mathfrak s}}-2}x_0^l\,\Big|\, Y^{\mathfrak s}\right]\frac{1}{ \E[K_{Y^{\mathfrak s}} -1\,\mid\,Y^{\mathfrak s} ]}\right] \leq - c\beta^* - \delta.
\end{align*}
Since the left-hand-side is increasing in $x$ (recall that $\kappa_0\leq0$), the estimate actually holds for all $x\in[0,x_0[$. Thus we now know that for all $x \in[0,1]$ we can bound the first five summands of $\overline{a}$ by
\begin{align*}
  c\beta^* -  \alpha_{\mathfrak s}\E&\left[\E\left[\sum_{l=0}^{K_{Y^{\mathfrak s}}-2}x^l\,\Big|\, Y^{\mathfrak s}\right]\frac{1}{ \E[K_{Y^{\mathfrak s}} -1\,\mid\,Y^{\mathfrak s}]}\frac{1}{1+xV\E\left[\sum_{l=0}^{K_{Y^{\mathfrak s}}-2}x^l\,\mid\, Y^{\mathfrak s}\right]}\right] \\
				&+ \alpha_{\mathfrak s}\kappa_0(1-x)\E\left[\E\left[\sum_{l=0}^{K_{Y^{\mathfrak s}}-2}x^l\,\Big|\, Y^{\mathfrak s}\right]\frac{1}{ \E[K_{Y^{\mathfrak s}} -1\,\mid\,Y^{\mathfrak s} ]}\right] \leq -\delta < 0.
\end{align*}
Note that we also find a simple bound for the last two summands and thereby obtain
\begin{align}\label{eq:help_upperbounda}
 \overline{a}(N,x,\kappa_0) \leq - \delta + w x^N + \alpha_{\mathfrak s} x^N
\end{align}
for all $x \in [0,1]$ and $N \in \N$. We will return to this estimate below.

In order to obtain a contradiction, assume 
\begin{align}\label{eq:helpcontradiction_integral_infinite}
 \exists x\in]0,1[\quad \int_{[0,\infty[}\E_x[X(s)]\dd s = \infty.
\end{align}
Fix one such $x$ for the remainder of the proof. 
By monotone (or bounded) convergence, 
\begin{align}\label{eq:help_con_to_minus_delta_2}
 \E_x[X(s)(-\delta+(w+&\alpha_{\mathfrak s}) X(s)^N)]\notag\\
						& \xrightarrow{N\rightarrow \infty}    (-\delta +w+\alpha_{\mathfrak s})\underbrace{\P_x(X(s)=1)}_{\overset{\eqref{eq:helpneverhit1}}{=}  0} -\delta\underbrace{\E_x[X(s)\1_{\{X(s)<1\}}]}_{\overset{\eqref{eq:helpcontradiction_integral_infinite}}{>}0} \notag \\
		    & \qquad \qquad  = -\delta\E_x[X(s)]<0.
\end{align}
We can now estimate the right-hand-side of \eqref{eq:Martingale}, where we again use monotone (or bounded) convergence in the last step: For every 
$t>0$
\begin{align}\label{eq:help_est_rhs_ii_2}
 \int_{[0,t]}\E_x[\mathcal Af_{N,\kappa_0}(X(s))]\dd s 	& \leq \int_{[0,t]}\E_x[X(s)\overline{a}(N,X(s),\kappa_0)]\dd s\notag\\
							& \overset{\eqref{eq:help_upperbounda}}{\leq} \int_{[0,t]}\E_x[X(s)(-\delta+(w+\alpha_{\mathfrak s})X(s)^N)]\dd s\notag\\
							&  \xrightarrow{N\rightarrow \infty} -\delta \int_{[0,t]}\E_x[X(s)]\dd s<0.
\end{align}

By assumption \eqref{eq:helpcontradiction_integral_infinite} we may choose $\bar t >0$ such that
\begin{align*}
 \int_{[0,\bar t]}\E_x[X(s)]\dd s \geq -4\frac{\log(1-x)-\kappa_0x}{\delta}
\end{align*}
 and by \eqref{eq:help_est_rhs_ii_2} an $\bar N=\bar N(\bar t)$ such that
 \begin{align}\label{eq:helpcontradiction_integral_very_negative}
  \int_{[0,\bar t]}\E_x[\mathcal Af_{\bar N,\kappa_0}(X(s))]\dd s  \leq 2(\log(1-x)+\kappa_0x). 
\end{align}

Now that we have the right-hand-side under control, we turn to the left-hand-side of \eqref{eq:Martingale}.
Remembering that $\kappa_0\leq0$, we crudely estimate
\begin{align*}
 \E_x\left[f_{\bar N,\kappa_0}(X(\bar t))\right] - f_{\bar N,\kappa_0}(x)	& = \sum_{n=1}^{\bar N}\frac{1}{n}\E_x[X(t)^n] -\kappa_0\E_x[X(t)] - \sum_{n=1}^{\bar N}\frac{1}{x^n} + \kappa_0x \\
								& \geq - \sum_{n=1}^{\bar N}\frac{1}{n}x^n + \kappa_0x \geq \log(1-x) +\kappa_0x
\end{align*}
which is a contradiction to \eqref{eq:Martingale} given \eqref{eq:helpcontradiction_integral_very_negative}.
\end{proof}

\subsubsection{Proof of Proposition \ref{prop:bounds_operator}: Bounds on the generator} \label{subsec:help_longterm}

It will be convenient to consider the summands of the generator individually. Therefore, we denote by $\mathcal A_{\mathfrak s}$ the part of the generator describing the rare selection mechanism, i.e.\
\begin{align}\label{eq:def_generator_for_rare_selection}
\mathcal A_{\mathfrak s}f(x):=\int_{[0,1]} \left(f(\E[x^{K_y}]) - f(x)\right) \frac{1}{\E[K_y-1]}\Lambda_{\mathfrak s}(\dd y) 
\end{align}
and the part describing the coordinated random genetic drift as 
\begin{align}\label{eq:def_generator_coordinated_genetic_drift}
\mathcal A_{\mathfrak c}f(x):=\mathcal Af(x)-\sigma\frac{1}{2}x(1-x)f''(x) +w x(1-x)f'(x) - \mathcal A_{\mathfrak s}f(x). 
\end{align}
 Recall that we abbreviate $\alpha_{\mathfrak s}:=\Lambda_{\mathfrak s}([0,1])$.
\begin{proof}[Proof of Proposition \ref{prop:bounds_operator}]
 
 We begin with a few observations on the functions \eqref{eq:def_Lyapunov_truncated_functions} and the generators applied to them. Their first derivative is
 \begin{align*}
  f'_{N,\kappa}(x)= \sum_{n=1}^N x^{n-1} - \kappa = \begin{cases}
                                                     \frac{1-x^N}{1-x} - \kappa,	& \text{if }x\neq 1,\\
                                                     N - \kappa, 				& \text{if }x=1.
                                                    \end{cases}
 \end{align*}
 Since the calculations are simple, but tedious, we consider the parts of the generator separately for clarity, using their representations given in Lemma \ref{lem:Griffithsgenerator_general}. Recall that we assumed $\sigma=0$. 
 
 Taking the expectation with respect to $V$, then in the next step plugging in the definition of the derivative we obtain the following expression: 

 \begin{align*}
  \mathcal A_{\mathfrak c} f_{N,\kappa}(x)	& = \frac{1}{2}cx(1-x)\E\left[f''_{N,\kappa}(x(1-W)+VW)\right]\\
					& = \frac{1}{2}cx(1-x)\E\left[\frac{1}{W}(f'_{N,\kappa}(x(1-W)+W)-f'_{N,\kappa}(x(1-W)))\right]\\
					& = x\underbrace{\frac{1}{2}c(1-x)\E\left[\frac{1}{W}\left(\frac{1-(x(1-W)+W)^N}{1-(x(1-W)+W)} - \frac{1-(x(1-W))^N}{1-x(1-W)}\right)\right]}_{=:\tilde{\mathcal{A}}_{\mathfrak c}f_{N,\kappa}(x)}.
 \end{align*}
 
Note that the expression $f'_{N,\kappa}(1)$ can only occur for $x=1$ and since in this case the right-hand-side is 0 if we assume something in the spirit of $0\cdot \infty =0$, we take the liberty to write this expression for all $x\in[0,1]$. 

Following the same reasoning we obtain a similar expression for the coordinated selection term for any $x \in [0,1]$:
\begin{align*}
\mathcal A_{\mathfrak s}&f_{N,\kappa}(x)	 \\
			& =  -\alpha_{\mathfrak s}x(1-x)\E\Bigg[\E\left[\sum_{l=0}^{K_{Y^{\mathfrak s}}-2}x^l\,\Big|\, Y^{\mathfrak s}\right]\frac{1}{ \E[K_{Y^{\mathfrak s}} -1\,\mid\,Y^{\mathfrak s} ]}\\
			& \qquad \qquad \qquad \qquad \qquad \qquad \qquad \qquad \qquad \quad \times f'_{N,\kappa}\left(x+V\E[x^{K_{Y^{\mathfrak s}}}-x\,\mid\, Y^{\mathfrak s}]\right)\Bigg]\\
			& =  -x\alpha_{\mathfrak s}(1-x)\E\Bigg[\E\left[\sum_{l=0}^{K_{Y^{\mathfrak s}}-2}x^l\,\Big|\, Y^{\mathfrak s}\right]\frac{1}{ \E[K_{Y^{\mathfrak s}} -1\,\mid\,Y^{\mathfrak s} ]}\\
			& \qquad \qquad \qquad \qquad \qquad \times\Bigg(\frac{1-\left(x-x(1-x)V\E\left[\sum_{l=0}^{K_{Y^{\mathfrak s}}-2}x^l\,\mid\, Y^{\mathfrak s}\right]\right)^N}{1-\left(x-x(1-x)V\E\left[\sum_{l=0}^{K_{Y^{\mathfrak s}}-2}x^l\,\mid\, Y^{\mathfrak s}\right]\right)}-\kappa\Bigg)\Bigg].
 \end{align*}
 Let $\tilde{\mathcal{A}}_{\mathfrak s}f_{N,\kappa}(x)$ be such that ${\mathcal{A}}_{\mathfrak s}f_{N,\kappa}(x)=x\tilde{\mathcal{A}}_{\mathfrak s}f_{N,\kappa}(x)$.  Define
  \begin{align*}
 \tilde{\mathcal A}f_{N,\kappa}(x):=\tilde{\mathcal A}_{\mathfrak c}f_{N,\kappa}(x)-w (1-x)f'_{N,\kappa}(x)+\tilde{\mathcal A}_{\mathfrak s}f_{N,\kappa}(x)  
 \end{align*}
 which is the claim of \ref{prop_subclaim:bounds_operator_Atilde}. in this proposition.
 
 From these expressions, we can deduce bounds for $\tilde{\mathcal A}f_{N,\kappa}(x)$ that are monotone in $N$. 
 
 We obtain a lower bound, by omitting certain positive terms in $\tilde{\mathcal A}f_{N,\kappa}$ as follows: For $x <1$, define
\begin{align*}
  \underline{a}&(N,x,\kappa)	:= \frac{1}{2}c(1-x)\E\left[\frac{1}{W}\left(\frac{1-(x(1-W)+W)^N}{1-(x(1-W)+W)} - \frac{1}{1-x(1-W)}\right)\right] \\
                & \qquad -w (1-x) \left(\frac{1}{1-x} - \kappa\right) \\
				& \qquad -\alpha_{\mathfrak s}(1-x)\E\Bigg[\E\left[\sum_{l=0}^{K_{Y^{\mathfrak s}}-2}x^l\,\Big|\, Y^{\mathfrak s}\right]\frac{1}{ \E[K_{Y^{\mathfrak s}} -1\,\mid\,Y^{\mathfrak s} ]}\\
				& \qquad \qquad \qquad \qquad \qquad  \times\Bigg(\frac{1}{1-\left(x-x(1-x)V\E\left[\sum_{l=0}^{K_{Y^{\mathfrak s}}-2}x^l\,\mid\, Y^{\mathfrak s}\right]\right)}-\kappa\Bigg)\Bigg].
\end{align*}
Then, trivially $\tilde{\mathcal A}f_{N,\kappa}(x) \geq \underline{a}(N,x,\kappa) $ for any $x \in [0,1[$ (and $N \in \N$, $\kappa \in \R$). 

We rearrange the terms of $\underline{a}$ to obtain a more convenient form. Separating the terms that do/do not depend on $N$ and $\kappa$, then canceling several $(1-x)$-terms in a second step
 \begin{align*}
  \underline{a}&(N,x,\kappa)	= \frac{1}{2}c(1-x)\E\left[\frac{1}{(1-W)(1-x)(1-x(1-W))}\right] \\
                & \qquad - \frac{1}{2}c(1-x)\E\left[\frac{(x(1-W)+W)^N}{W(1-W)(1-x)}\right] -w\\ 
				& \qquad -\alpha_{\mathfrak s}(1-x)\E\Bigg[\E\left[\sum_{l=0}^{K_{Y^{\mathfrak s}}-2}x^l\,\Big|\, Y^{\mathfrak s}\right]\frac{1}{ \E[K_{Y^{\mathfrak s}} -1\,\mid\,Y^{\mathfrak s} ]}\\
				& \hspace{150pt} \times\frac{1}{(1-x)\left(1+xV\E\left[\sum_{l=0}^{K_{Y^{\mathfrak s}}-2}x^l\,\mid\, Y^{\mathfrak s}\right]\right)}\Bigg]\\
				& \qquad +\kappa w (1-x) + \kappa\alpha_{\mathfrak s}(1-x)\E\left[\E\left[\sum_{l=0}^{K_{Y^{\mathfrak s}}-2}x^l\,\Big|\, Y^{\mathfrak s}\right]\frac{1}{ \E[K_{Y^{\mathfrak s}} -1\,\mid\,Y^{\mathfrak s} ]}\right]\\
				&= \frac{1}{2}c\E\left[\frac{1}{(1-W)(1-x(1-W))}\right] -w \\
				& \qquad -\alpha_{\mathfrak s}\E\left[\E\left[\sum_{l=0}^{K_{Y^{\mathfrak s}}-2}x^l\,\Big|\, Y^{\mathfrak s}\right]\frac{1}{ \E[K_{Y^{\mathfrak s}} -1\,\mid\,Y^{\mathfrak s} ]}\frac{1}{1+xV\E\left[\sum_{l=0}^{K_{Y^{\mathfrak s}}-2}x^l\,\mid\, Y^{\mathfrak s}\right]}\right]\\
				& \qquad +\kappa w (1-x)+\kappa\alpha_{\mathfrak s}(1-x)\E\left[\E\left[\sum_{l=0}^{K_{Y^{\mathfrak s}}-2}x^l\,\Big|\, Y^{\mathfrak s}\right]\frac{1}{ \E[K_{Y^{\mathfrak s}} -1\,\mid\,Y^{\mathfrak s} ]}\right]\\
				& \qquad - \frac{1}{2}c\E\left[\frac{(x(1-W)+W)^N}{W(1-W)}\right] .
 \end{align*}
 By extending $\underline{a}(N,1,\kappa):=-w - \alpha_{\mathfrak s}\alpha^*$, we do indeed obtain \eqref{eq:lower_bound_a} and
 \begin{align}\label{eq:helplowerboud}
  \tilde{\mathcal A}f_{N,\kappa}(x) \geq \underline{a}(N,x,\kappa)  
 \end{align}
 for all $N \in \N$, $\kappa \in \R$ and $x \in [0,1]$. 
 
 Note that we can estimate
 \begin{align*}
  \vert\underline{a}&(N,x,\kappa) \vert \leq c\beta^* + w + \alpha_{\mathfrak s} + w  +\vert\kappa\vert\alpha_{\mathfrak s}+ c\beta^*<\infty
 \end{align*}
 and the right-hand-side does not depend on $x$ nor $N$, which is \eqref{eq:lower_bound_a_upperbounded}.
 
 In a similar fashion we define an upper bound by omitting certain negative terms in $\tilde{\mathcal A}f_{N,\kappa}$: For $x<1$, define
 \begin{align*}
  \overline{a}&(N,x,\kappa)	:= \frac{1}{2}c(1-x)\E\left[\frac{1}{W}\left(\frac{1}{1-(x(1-W)+W)} \right)\right] - w(1-x)\left(\frac{1-x^N}{1-x} -\kappa\right)\\
				& \quad - \alpha_{\mathfrak s}(1-x)\E\Bigg[\E\left[\sum_{l=0}^{K_{Y^{\mathfrak s}}-2}x^l\,\Big|\, Y^{\mathfrak s}\right]\frac{1}{ \E[K_{Y^{\mathfrak s}} -1\,\mid\,Y^{\mathfrak s} ]}\\
				& \qquad \qquad \qquad \qquad \qquad \times\Bigg(\frac{1-\left(x-x(1-x)V\E\left[\sum_{l=0}^{K_{Y^{\mathfrak s}}-2}x^l\,\mid\, Y^{\mathfrak s}\right]\right)^N}{1-\left(x-x(1-x)V\E\left[\sum_{l=0}^{K_{Y^{\mathfrak s}}-2}x^l\,\mid\, Y^{\mathfrak s}\right]\right)}-\kappa\Bigg)\Bigg]\\
				& = \frac{1}{2}c\E\left[\frac{1}{W(1-W)} \right] - w(1-x^N)\\
				& \quad + w\kappa(1-x)+\alpha_{\mathfrak s}\kappa(1-x)\E\Bigg[\E\left[\sum_{l=0}^{K_{Y^{\mathfrak s}}-2}x^l\,\Big|\, Y^{\mathfrak s}\right]\frac{1}{ \E[K_{Y^{\mathfrak s}} -1\,\mid\,Y^{\mathfrak s} ]}\Bigg] \\\\
				& \quad  -\alpha_{\mathfrak s}\E\Bigg[\E\left[\sum_{l=0}^{K_{Y^{\mathfrak s}}-2}x^l\,\Big|\, Y^{\mathfrak s}\right]\frac{1}{ \E[K_{Y^{\mathfrak s}} -1\,\mid\,Y^{\mathfrak s}]}\\
				& \qquad \qquad \qquad \qquad \qquad \qquad \times\frac{1-\left(x-x(1-x)V\E\left[\sum_{l=0}^{K_{Y^{\mathfrak s}}-2}x^l\,\mid\, Y^{\mathfrak s}\right]\right)^N}{1+xV\E\left[\sum_{l=0}^{K_{Y^{\mathfrak s}}-2}x^l\,\mid\, Y^{\mathfrak s}\right]}\Bigg]
 \end{align*}
and extend it by $\overline{a}(N,1,\kappa):=c\beta^*$ to obtain \eqref{eq:upper_bound_a}. As before, by definition,
\begin{align}\label{eq:helpupperbound}
 \tilde{\mathcal A}f_{N,\kappa}(x) \leq \overline{a}(N,x,\kappa) 
\end{align}
for all $N \in \N$, $\kappa \in \R$ and $x \in [0,1]$. In addition, if we recall that $-x(1-x)\sum_{l=0}^kx^l=x^k-x$, we can simply estimate
\begin{align*}
  \vert \overline{a}&(N,x,\kappa) \vert \leq c\beta^* + w + w\vert\kappa\vert + \alpha_{\mathfrak s} \vert\kappa\vert+ \alpha_{\mathfrak s}\alpha^* <\infty
\end{align*}
and the right-hand-side does not depend on $x$ nor $N$, as claimed in \eqref{eq:upper_bound_a_upperbounded}. 

In particular, combining the above observations, we now know that
\begin{align*}
 \vert\tilde{\mathcal A}f_{N,\kappa}(x)\vert \leq 2c\beta^* + (2+\vert\kappa\vert)w+ (1+\vert\kappa\vert +\alpha^*)\alpha_{\mathfrak s}< \infty
\end{align*}
uniformly in $N\in \N$ and $x \in [0,1]$ and thus also \eqref{eq:A_with_Lyapunov_is_bounded}, i.e.\ \ref{prop_subclaim:A_with_Lyapunov_is_bounded}. in this proposition holds and the proof is completed.
\end{proof}

\subsubsection{Proof of Lemma \ref{lem:Griffithsgenerator_general}: the Griffiths representation}\label{subsec:Griffiths_proof}
Recall that we abbreviated the parts of the generator $\mathcal A$ corresponding to rare selection and coordinated genetic drift as $\mathcal A_{\mathfrak s}$ and $\mathcal A_{\mathfrak c}$ in \eqref{eq:def_generator_for_rare_selection} and \eqref{eq:def_generator_coordinated_genetic_drift} respectively.

\begin{proof}
It is sufficient to find the correct representations for $\mathcal A_{\mathfrak c}$ and $\mathcal A_{\mathfrak s}$.
Theorem 1 in \cite{GriffithsLambda} already states
\begin{align*}
 \mathcal A_{\mathfrak c}f(x) = \frac{1}{2}cx(1-x)\E\left[f''(x(1-W)+VW)\right].
\end{align*}

On the other hand, in the spirit of the proof of said theorem, we can calculate

\begin{align*}
 \mathcal A_{\mathfrak s}f(x)	& = \int_{[0,1]} \left(f(\E[x^{K_y}]) - f(x)\right) \frac{1}{\E[K_y -1 ]}\Lambda_{\mathfrak s}(\dd y) \\
			& = \int_{[0,1]} \left(f(x+\E[x^{K_y}-x]) - f(x)\right) \frac{1}{ \E[K_y -1 ]}\Lambda_{\mathfrak s}(\dd y) \\
			& = \int_{[0,1]} \int_{[0,1]}\E[x^{K_y}-x]f'(x+u\E[x^{K_y}-x])\dd u \frac{1}{ \E[K_y -1 ]}\Lambda_{\mathfrak s}(\dd y) \\
			& = -\alpha_{\mathfrak s}\int_{[0,1]}\int_{[0,1]} x(1-x)\E\left[\sum_{l=0}^{K_y-2}x^l\right]\frac{1}{ \E[K_y -1 ]}\\
			& \qquad \qquad \qquad \qquad \qquad \qquad \qquad \qquad \times f'(x+u\E[x^{K_y}-x]) \dd u\frac{\Lambda_{\mathfrak s}(\dd y)}{\Lambda_{\mathfrak s}([0,1])} \\
			& = -\alpha_{\mathfrak s}x(1-x) \E\Bigg[\E\left[\sum_{l=0}^{K_{Y^{\mathfrak s}}-2}x^l\,\Big|\, Y^{\mathfrak s}\right]\frac{1}{ \E[K_{Y^{\mathfrak s}} -1\,\mid\,Y^{\mathfrak s} ]}\\
			& \qquad\qquad \qquad \qquad \qquad \qquad \times f'\left(x-x(1-x)V\E\left[\sum_{l=0}^{K_{Y^{\mathfrak s}}-2}x^l\,\mid\, Y^{\mathfrak s}\right]\right)\Bigg].
\end{align*}
We used once again the simple observation that $x^k-x = -x(1-x)\sum_{l=0}^{k-2}x^l$ for any $k \in \N$ and $x \in [0,1]$ if we interpret the empty sum as zero.
\end{proof}

\section*{Acknowledgements}
We thank Martin Slowik for pointing out a mistake in an earlier version of this manuscript, which we corrected with Lemma \ref{lem:amazinggrace}. We are also grateful to an anonymous referee for suggesting to comment on hypergeometric duality function, as an alternative to our sampling duality. ACG was supported by grants UNAM PAPIIT IA100419 and CONACYT CIENCIA BASICA A1-S-14615. MWB was supported by the SPP 1590, projects BL 1105/5-1 2 and KU 2886/1-1, and by the Deutsche Forschungsgemeinschaft (DFG, German Research Foundation) under Germany's Excellence Strategy -- The Berlin Mathematics
Research Center MATH+ (EXC-2046/1, project ID: 390685689) and acknowledges support from DFG CRC/TRR 388 ``Rough Analysis, Stochastic Dynamics and Related Fields'', Project A09.

\end{document}